\documentclass[11pt,reqno]{amsart}
\usepackage[T1]{fontenc}
\usepackage[utf8]{inputenc}
\usepackage{lmodern}
\usepackage{amsmath,amssymb,amsthm,mathtools}
\usepackage{mathrsfs}
\usepackage[margin=1.05in]{geometry}
\usepackage[expansion=false]{microtype}
\usepackage{enumitem}
\usepackage[colorlinks=true,linkcolor=blue,citecolor=blue,urlcolor=blue]{hyperref}
\hypersetup{pdftitle={Solution to Dixmier's Problem about spectra of C*-algebras},pdfauthor={Martino Lupini}}

\newtheorem{theorem}{Theorem}[section]
\newtheorem{proposition}[theorem]{Proposition}
\newtheorem{lemma}[theorem]{Lemma}
\newtheorem{corollary}[theorem]{Corollary}
\theoremstyle{definition}
\newtheorem{definition}[theorem]{Definition}
\newtheorem{remark}[theorem]{Remark}
\newtheorem{question}[theorem]{Question}
\theoremstyle{plain}
\newtheorem{introtheorem}{Theorem}

\newenvironment{theoremA}{\begin{introtheorem}}{\end{introtheorem}}
\newenvironment{theoremB}{\begin{introtheorem}}{\end{introtheorem}}
\newenvironment{theoremC}{\begin{introtheorem}}{\end{introtheorem}}
\newenvironment{theoremD}{\begin{introtheorem}}{\end{introtheorem}}
\numberwithin{equation}{section}

\newcommand{\N}{\mathbb N}
\newcommand{\Z}{\mathbb Z}
\newcommand{\C}{\mathbb C}
\newcommand{\T}{\mathbb T}
\newcommand{\Hil}{\mathcal H}
\newcommand{\BH}{\mathcal B}
\newcommand{\KH}{\mathcal K}
\newcommand{\Schat}{\mathcal S}
\newcommand{\CAR}{M_{2^\infty}}
\newcommand{\id}{\operatorname{id}}
\newcommand{\Irr}{\operatorname{Irr}}
\newcommand{\Rep}{\operatorname{Rep}}
\newcommand{\Tr}{\operatorname{Tr}}

\newcommand{\Sym}{\operatorname{Sym}}
\newcommand{\SL}{\operatorname{SL}}
\newcommand{\DSet}{\mathsf{DSet}}
\newcommand{\defdual}[1]{\widehat{#1}}
\newcommand{\norm}[1]{\left\lVert #1\right\rVert}
\newcommand{\restr}{\mathord{\upharpoonright}}
\newcommand{\bdef}{\preccurlyeq}
\newcommand{\sdef}{\prec}
\newcommand{\isoBD}{\mathrel{\cong_{\mathrm{BD}}}}
\newcommand{\E}{\mathbb E}

\newenvironment{theoremE}{\begin{introtheorem}}{\end{introtheorem}}
\newenvironment{theoremF}{\begin{introtheorem}}{\end{introtheorem}}
\newenvironment{theoremG}{\begin{introtheorem}}{\end{introtheorem}}

\newcommand{\SOT}{\operatorname{SOT}}
\newcommand{\qemb}{\mathrel{\lhook\joinrel\longrightarrow}_{\!\mathrm q}}
\newcommand{\qiso}{\mathrel{\cong_{\mathrm q}}}
\newcommand{\bd}{\isoBD}
\newcommand{\ADR}{\mathrm{AD}_{\mathbb R}}
\newcommand{\AD}{\mathrm{AD}}
\title[Solution to Dixmier's Problem]
{Solution to Dixmier's Problem\\about spectra of $C^*$-algebras}
\author{Martino Lupini}
\address{Dipartimento di Matematica, Universit\`a di Bologna,
Piazza di Porta S. Donato, 5, 40126 Bologna, Italy}
\email{martino.lupini@unibo.it}
\urladdr{https://www.lupini.org/}
\date{October 13, 2026}

\begin{document}
\begin{abstract}
Dixmier asked in 1967 whether the spectra of simple separable
non-type-I $C^*$-algebras all have the same Borel structure. We solve
the formulation using ordinary isomorphisms of quotient Borel
$\sigma$-algebras by proving that its universal assertion is
independent of ZFC, relative to the consistency of ZFC. Under the
continuum hypothesis, all separable non-type-I spectra are measurably
isomorphic, without simplicity or unitality assumptions. In the model
obtained by iterating Sacks forcing for $\omega_2$ stages with countable
supports over a GCH ground model, a fixed simple unital universal
spectrum is not measurably isomorphic to the CAR spectrum. Neither
consistency direction requires a large cardinal.

The Borel-definable formulation, which requires isomorphisms to admit
Borel lifts between irreducible representations, has a negative answer
in ZFC. For the CAR algebra $A$ and a simple separable exact crossed
product $B$ associated to $\mathrm{SL}_3(\mathbb Z)$, there is no
Borel-definable injection $\widehat B\to\widehat A$. Nevertheless,
all separable nuclear non-type-I spectra are Borel-definably isomorphic.

We also determine the answer under several choiceless hypotheses.
The axiom of determinacy together with dependent choice gives a stronger
negative conclusion: the spectrum of $B$ does not even inject as a set
into the spectrum of $A$.
The symmetric Solovay model gives the same conclusion. Over ZF, a
wellordering of the reals of order type $\omega_1$ suffices for a
positive answer. The finite-permutation Feferman--L\'evy model gives
another positive answer for the full sigma-generated Borel structures.
These comparisons concern
the quotient structures themselves, without measure completion.
\end{abstract}
\maketitle

\section{Introduction}

The spectrum of a separable $C^*$-algebra has two closely related
descriptive structures. Its irreducible representations form a standard
Borel space, with unitary equivalence as a Borel equivalence relation.
The resulting set of classes carries the quotient Borel
$\sigma$-algebra. A Borel assignment of representations induces a
measurable function on classes, but an arbitrary measurable function on
the quotient need not admit such an assignment. For maps between type-I
spectra the distinction disappears because their orbit presentations
admit Borel selectors (Lemma~\ref{lem:smooth-strong-selector}). A
standard Borel quotient alone does not ensure this for an arbitrary
presentation; see Remark~\ref{rem:app-one-way}. For non-type-I spectra
the distinction is decisive. The purpose of this paper is to determine what can be proved
about spectrum isomorphism in each category. Our main result proves
that the ordinary quotient-measurable assertion is independent of ZFC:
CH yields a positive answer, while the iterated Sacks model yields a
negative one, with no large-cardinal assumption. The Borel-definable
assertion is false in ZFC.

For a separable $C^*$-algebra $D$, let $\Irr(D)$ be the disjoint union
of the spaces of irreducible representations on $\mathbb C^n$,
$1\leq n<\infty$, and on $\ell^2(\mathbb N)$. Write
$E_D^{\mathrm{irr}}$ for unitary equivalence. Its quotient is denoted
$\widehat D$ when the morphisms under discussion admit Borel lifts,
and $Q_D$ when only its quotient measurable structure is used. These
symbols have the same underlying set. Explicitly,
\[
 \Sigma_{Q_D}=\{S\subseteq\Irr(D)/E_D^{\mathrm{irr}}:
                  q_D^{-1}(S)\text{ is Borel in }\Irr(D)\},
\]
where $q_D$ is the quotient map. Equivalent pure-state presentations,
including the nonunital case, are established in
Section~\ref{sec:duals}.

A function between spectra is \emph{Borel-definable} if it has a
Borel lift between these spaces of representations. Thus a
Borel-definable injection corresponds to a Borel reduction of the
equivalence relations. We write $\widehat D\bdef\widehat C$ for this
comparison, and $\widehat D\sdef\widehat C$ when
$\widehat D\bdef\widehat C$ and the reverse injection does not exist. The notation $\widehat D\isoBD\widehat C$ means that
there is a bijection with Borel lifts in both directions. Borel-definable
functions are the morphisms, and bijections with Borel-definable inverses
are the isomorphisms, of the category of definable sets of
Bergfalk--Lupini--Panagiotopoulos \cite[Section~3]{BLP2024}. For the
spectral presentations used here, a bijective Borel-definable map
automatically has a Borel-definable inverse, by the source hypothesis
in Lemma~\ref{lem:strong-inverse-lift}.
We reserve $Q_D\qiso Q_C$ for an isomorphism of the measurable spaces
$(Q_D,\Sigma_{Q_D})$ and $(Q_C,\Sigma_{Q_C})$, without requiring a
lift. Consequently
\[
          \widehat D\isoBD\widehat C
                  \quad\Longrightarrow\quad Q_D\qiso Q_C.
\]
One conclusion of the paper is that the converse fails under the
continuum hypothesis, even for simple unital exact algebras.

\subsection{Dixmier's question}

Dixmier's work on Borel structures on spectra
\cite{Dixmier1960} and Glimm's characterization of type I
\cite[Theorem~2]{Glimm1961} place the question in its natural setting.
For separable algebras, type I is equivalent to smoothness of
irreducible-representation equivalence. These orbit presentations then
admit Borel selectors, so the quotient is standard Borel and its
isomorphism type in either category is determined by its cardinality. The
non-type-I alternative leaves open the possibility of a common
nonsmooth spectrum.

Elliott \cite[Section~1, p.~59]{Elliott1977} reports that Dixmier
asked in a 1967 lecture in Baton Rouge whether the spectra of simple
antiliminary separable $C^*$-algebras are isomorphic. Elliott also
records Dixmier's suggestion to examine AF algebras and establishes
the positive comparison for non-type-I AF spectra. The attribution
and date used here come from Elliott's published account, rather
than a transcript of the lecture. Farah restates the simple-algebra
question as Problem~8 of \cite{Farah2011}; Thomas
\cite[Remark~2.1]{Thomas2015} discusses the broader comparison of all
non-type-I separable unital algebras. Theorem~\ref{intro:A} gives a
negative answer to the formulation in Thomas's Remark~2.1, where
isomorphisms are required to admit Borel lifts in both directions.

The terminology in these formulations requires care. Farah's
discussion after Problem~8 in arXiv:0908.1943v2
\cite[p.~6]{Farah2011} describes the quotient Borel structure and
then reformulates the comparison by Borel isomorphisms of the
representation equivalence relations. The latter condition supplies
Borel lifts and is stronger than an isomorphism of the bare quotient
$\sigma$-algebras. We address both formulations explicitly. In
particular, our independence theorem answers the literal
quotient-measurable question, while our ZFC separation answers the
representation-code formulation. Read literally for quotient
measurable structures, the equivalence asserted in that paragraph of
Farah's note fails under CH, by Theorems~\ref{intro:A} and~\ref{intro:E}:
a Borel isomorphism of representation codes would give the reverse
Borel reduction excluded by Theorem~\ref{intro:A}.

\subsection{The measurable formulation: independence from ZFC}

Put $A=\CAR$. Let $\Phi$ denote the assertion that all simple separable non-type-I
$C^*$-algebras have isomorphic quotient measurable spectra.
We prove that $\Phi$ is equivalent to the same assertion for all
separable non-type-I algebras, and to the single endpoint comparison
$Q_{C^*(F_\infty)}\qiso Q_A$. A fixed simple purely infinite nonexact
algebra realizes the entire Borel-definable spectrum of
$C^*(F_\infty)$, so this reduction preserves the simple-algebra
content of the question.

\begin{theoremG}[The measurable Dixmier problem]\label{intro:G}
The ordinary quotient-measurable assertion $\Phi$ is independent of
ZFC, relative to the consistency of ZFC. More precisely, CH implies
$\Phi$, whereas the countable-support iteration of full Sacks forcing
of length $\omega_2$ over a model of ZFC+GCH satisfies
\[
                Q_{\mathcal U}\not\qiso Q_A
\]
for the fixed simple unital universal algebra $\mathcal U$ of
Theorem~\ref{thm:universal}. Thus
\[
 \operatorname{Con}(\mathrm{ZFC})
 \Longleftrightarrow\operatorname{Con}(\mathrm{ZFC}+\Phi)
 \Longleftrightarrow\operatorname{Con}(\mathrm{ZFC}+\neg\Phi).
\]
\end{theoremG}

This is the main solution to the measurable formulation of Dixmier's
question: the usual axioms of set theory decide neither the universal
assertion nor its negation. Both models satisfy Choice. In the negative
model the continuum is $\aleph_2$, and the two endpoint spectra still
have set-theoretic bijections; the obstruction is to preserving their
quotient Borel $\sigma$-algebras. The theorem uses the universal simple
algebra $\mathcal U$, whose nonexactness is relevant to the scope of
the example. It does not establish measurable nonisomorphism for the
exact crossed product introduced below.

The positive half follows from a general theorem about orbit quotients.
Recall that $E_0$ on $2^{\N}$ is eventual equality: $xE_0y$ if
$x(n)=y(n)$ for all sufficiently large $n$.

\begin{theoremE}\label{intro:E}
Assume ZFC and CH. If a Polish group acts continuously on a Polish
space $X$ and its orbit relation $E$ is nonsmooth and $F_\sigma$,
then
\[
                         X/E\qiso2^{\N}/E_0.
\]
Consequently the quotient measurable spectra of any two separable
non-type-I $C^*$-algebras are isomorphic, without simplicity or
unitality assumptions.
\end{theoremE}

The isomorphisms preserve the original quotient Borel
$\sigma$-algebras. The proof constructs an exactly coherent increasing
family of classwise Borel maps on essentially hyperfinite pieces.
Families of zero--one probabilities extend the correspondence of
measurable sets beyond each piece. At countable limits, fusion realizes
every compatible orbit-closure profile on both sides. CH schedules all
quotient Borel events in this recursion. The $F_\sigma$ hypothesis
supplies the fixed closed cover used in fusion and ensures that the
locally closed orbit locus behaves continuously under countable
topology joins.

The positive construction also applies to all nonsmooth countable
Borel equivalence relations. Thus, under CH, $E_0$ and the universal
countable Borel equivalence relation $E_\infty$ have isomorphic
ordinary quotient Borel spaces despite $E_0<_B E_\infty$. We use
$E_\infty$ throughout for the orbit relation of the full Bernoulli
shift $F_2\curvearrowright2^{F_2}$, a fixed universal countable
Borel relation; other universal countable relations need only
be Borel bireducible with this representative.
This gives a conditional negative answer to the quotient-isomorphism
question recorded by Kechris \cite[Problem~2.34]{KechrisCBER}; see
Corollary~\ref{cor:ch-countable-quotients}.

The negative half must control an arbitrary measurable isomorphism,
whose graph need not have a real code. We use the reflection theorem
of Farah--Shelah \cite[Proposition~4.6, arXiv v2]{FarahShelah}
to obtain an intermediate model in which its restrictions to smooth
pieces have compatible Borel lifts. A family of smooth events records
incidence with whole orbit classes. Adding one Sacks real produces an
event whose image cannot yet have a Borel code. We then compare two
independent copies of the remaining iteration over that common real.
A canonization argument supplies enough old smooth pieces in this
product to force the two proposed image codes to agree. Proper forcing
then makes the image descend to the earlier model, a contradiction.

The required canonization uses the geometry of partial-order Sacks
conditions developed by Kanovei \cite[Sections~1--3]{Kanovei1997}.
His later continuous-reading theorem and proof of the selector lemma
already provide the continuous-reading, paired-pasting, and
minimal-initial-segment methods used here
\cite[Theorem~25, p.~568; proof of Theorem~20, p.~563;
Lemma~28 and its proof, pp.~571--574]{Kanovei1999}; see also the
reduced-product treatment in \cite[Sections~9.1 and~10.5]{KSZ}.
Section~\ref{sec:sacks-fork} adapts these methods to two frozen tails
sharing a first real, with a full proof of the forcing presentation and
the support calculation that keeps the smooth pieces coded before that
real was added. These are the forcing inputs that turn the orbit-class
obstruction into the negative ZFC model.

\subsection{The Borel-definable formulation}

For the Borel-definable comparison, put
\begin{equation}\label{eq:example}
 \Gamma=\SL_3(\Z),\qquad
 \Gamma_n=\ker\bigl(\Gamma\to\SL_3(\Z/2^n\Z)\bigr),\qquad
 K=\varprojlim_n\Gamma/\Gamma_n,
\end{equation}
and let $\Gamma$ act on $K$ by left translation. The Borel-definable
comparison uses the following fixed pair of algebras.

\begin{theoremA}\label{intro:A}
The algebras $A$ and $B=C(K)\rtimes_r\Gamma$ are simple,
separable, unital, exact and antiliminary. The algebra $A$ is nuclear,
whereas $B$ is nonnuclear, and
\[
                         \widehat A\sdef\widehat B.
\]
In particular their spectra admit no Borel-definable bijection.
\end{theoremA}

The forward injection is the Glimm--Elliott lower bound. The reverse
obstruction is measured: on a fixed Borel family of irreducible
representations of $B$, every Borel-definable map to a nuclear spectrum
is essentially constant, while individual source classes are null.
This remains true on every positive-measure Borel restriction. Thus
the obstruction excludes even partially defined Borel lifts of an
injective quotient map on such a restriction.

Under CH, Theorem~\ref{intro:E} gives $Q_B\qiso Q_A$. Thus the same
fixed pair that refutes the Borel-definable assertion has measurably
isomorphic quotients in the positive model. Every such measurable
isomorphism from $Q_B$ to $Q_A$ fails to admit a Borel lift. This
separates the two formulations without changing either the algebras
or their quotient Borel structures.

\subsection{Choiceless axioms: regularity and positive models}

The independence theorem uses ZFC models for both answers. Additional
choiceless hypotheses illuminate a different issue: they can restrict
the existence of functions between the underlying quotient sets, or
enlarge the full Borel structures beyond their real-coded parts.
In particular, determinacy for games on the integers, together with
dependent choice, yields a negative solution here.

\begin{theoremF}\label{intro:F}
The following hold, with all spectra and functions interpreted in
the indicated universe and with type I understood in the GCR sense
in the choiceless cases.
\begin{enumerate}[label=(\roman*),leftmargin=*]
\item In $\mathrm{ZF}+\mathrm{DC}+\AD$, there is no set-theoretic
injection $Q_B\to Q_A$.
\item The same failure holds in the symmetric Solovay model obtained
from a ZFC model with an inaccessible cardinal.
\item The finite-permutation Feferman--L\'evy model constructed over
$L$ satisfies $\Phi$, using the full sigma-generated Borel structures
and the GCR definition of type I.
\end{enumerate}
\end{theoremF}

The negative comparisons in Theorem~\ref{intro:F} require no
measurability of a proposed injection. Under ordinary determinacy
$\AD$, the measure-game theorem of Crone--Fishman--Jackson--Schuerger--Simmons
\cite[Theorem~5.7]{MeasureGame} gives a Borel representative lift on a
conull set, where the obstruction behind Theorem~\ref{intro:A} applies.
Full uniformization of arbitrary relations, as supplied by real
determinacy $\ADR$, is unnecessary for this conclusion.
For the Solovay model it suffices to uniformize relations whose
sections are single pinned equivalence classes, using
Larson--Zapletal \cite[Theorem~10.2.2]{LZ}. All locators for that
work refer to the authors' \texttt{balanced14.pdf} draft identified
in the bibliography. The positive
Feferman--L\'evy construction proceeds differently: an internal
wellordered family of whole tuple orbits compares the quotient sets
without choosing a representative from every class. In that model
the full Borel $\sigma$-algebra of the reals is its power set. This
must be distinguished from the smaller family of real-coded Borel
sets. The coding and type-I conventions, and the passage from fixed
endpoint codes to arbitrary algebra codes, are given in
Sections~\ref{sec:zf-endpoints} and~\ref{sec:positive-consistency}.

The exact role of the axioms is consequently different in the three
constructions. CH permits a coherent recursion through the Borel events.
Determinacy supplies enough uniformization on sets of large measure to
expose the measured obstruction. In the Feferman--L\'evy model, the full Borel
structures are power sets, and a direct comparison of whole classes
suffices. Its positive conclusion uses that model's particular orbit
construction; the failure of Choice alone gives no such implication.

A positive answer also follows over ZF from the explicit assertion that
there is a bijection $\omega_1\to2^{\N}$
(Proposition~\ref{prop:zf-real-wellorder}). This retains the cardinal
content of CH while dispensing with full Choice. Projective determinacy
has a different scope from $\AD$: it is compatible with ZFC and CH, and
hence with a positive answer. Proposition~\ref{prop:pd-positive} gives
the relative-consistency argument. No conclusion from projective
determinacy alone is asserted.

Table~\ref{tab:answers} summarizes the two comparisons. The model rows
are interpreted internally, with the Borel conventions specified below.
\begin{table}[tbp]
\caption{The two formulations and their set-theoretic hypotheses}
\label{tab:answers}
\centering
\small
\renewcommand{\arraystretch}{1.2}
\begin{tabular}{@{}p{0.30\textwidth}p{0.65\textwidth}@{}}
\hline
Theory or model & Conclusion \\
\hline
ZFC & Borel-definable: $\widehat A\sdef\widehat B$
(Theorem~\ref{intro:A}). \\
ZFC+CH & Quotient measurable: all separable non-type-I spectra are
isomorphic (Theorem~\ref{intro:E}). \\
ZF + a bijection $\omega_1\to2^{\N}$ & Positive quotient measurable
comparison; full Choice is unnecessary
(Proposition~\ref{prop:zf-real-wellorder}). \\
Iterated Sacks model over GCH & Quotient measurable:
$Q_{\mathcal U}\not\qiso Q_A$; ZFC and
$2^{\aleph_0}=\aleph_2$ hold (Theorem~\ref{intro:G}). \\
$\mathrm{ZF}+\mathrm{DC}+\AD$ & No set-theoretic injection
$Q_B\to Q_A$ (Theorem~\ref{thm:determinacy}). \\
Symmetric Solovay model & The same negative comparison; obtained from
a ZFC model with an inaccessible cardinal (Theorem~\ref{thm:solovay}). \\
Finite-permutation Feferman--L\'evy model over $L$ & Positive quotient
measurable comparison for full sigma-generated Borel structures
(Section~\ref{sec:positive-consistency}). \\
\hline
\end{tabular}
\end{table}

\subsection{Nuclear spectra and the measured obstruction}

The stronger category also has a substantial positive class. For
separable $C^*$-algebras, amenability as a Banach algebra is equivalent to
nuclearity; see \cite[Corollary~2]{Connes1978} and
\cite[Theorem~3.1]{Haagerup1983}.

\begin{theoremD}\label{intro:D}
Every separable nuclear spectrum admits a Borel-definable injection
into the CAR spectrum. If $D_0,D_1$ are separable nuclear non-type-I
algebras, then
\[
 \widehat{D_0}\isoBD\widehat{D_1}
       \isoBD\widehat{\CAR}\isoBD\widehat{\mathcal O_2}.
\]
No simplicity, unitality, or UCT assumption is required.
\end{theoremD}

This extends Elliott's AF comparison and Thomas's comparison of
amenable group duals \cite[Theorem~6.1 and Corollaries~6.2 and~1.7]{Thomas2015}
to arbitrary separable nuclear algebras. It also shows that there are
no incomparable nuclear spectra: the type-I degrees are standard
Borel cardinalities, and all remaining nuclear algebras have the
CAR degree. Together with Theorem~\ref{intro:A} it answers the nuclear-to-exact
part of Farah's Problem~9 \cite{Farah2011} in the Borel-definable,
existential sense. Nuclearity cannot be replaced by exactness, even
for simple algebras.

The proof first transports representations of the aperiodic Cuntz
groupoid through a fixed Borel isomorphism with the stabilized CAR
groupoid. Borel spectral calculus for bisections makes this transport
an everywhere-defined operation on representation codes, preserving
all intertwiners. The periodic part has a smooth classification and
is inserted separately. Fully faithful Fock induction into Kumjian's
simple algebra, followed by $\mathcal O_2$-absorption, then places
every nuclear spectrum below the Cuntz spectrum. The
Glimm--Elliott lower bound and the definable Cantor--Bernstein
theorem \cite[Proposition~2.3]{MottoRos2012} finish the comparison,
using the strongly idealistic inverse lifts proved below. The same
Fock construction supplies the universal simple endpoint used for
the measurable assertion.

\begin{theoremC}\label{intro:C}
Let a countable property-$(T)$ group $\Lambda$ act ergodically by
probability-preserving Borel automorphisms of a standard Borel
probability space $(X,\mu)$. If $D$ is a separable nuclear
$C^*$-algebra, every Borel homomorphism from the orbit relation
$R_\Lambda^X$ to $E_D^{\mathrm{irr}}$ is essentially constant modulo
$E_D^{\mathrm{irr}}$. The same holds on every positive-measure Borel
restriction. If $\mu$ is nonatomic, none of these homomorphisms is
a reduction.
\end{theoremC}

Borel intertwiners along the action satisfy the cocycle identity only
up to scalars. Conjugation on Hilbert--Schmidt operators removes
this defect. Positive trace-class averages, constructed from
Haagerup's positive diagonal \cite[Theorem~3.1]{Haagerup1983} in the
form of \cite[Lemma~3.1]{CSSWW2012}, produce almost invariant
square-root fields. Property~$(T)$ gives an invariant field and hence
a common normal state. Irreducibility then forces equivalence.
Preserving normality is essential; a weak-star limit of normal
states would not suffice. Finite-dimensional matrix averages give
a direct CAR proof. Its fixed coded obstruction is transferred to ZF
in Lemma~\ref{lem:zf-measured-obstruction}, which supplies the measured
input for the negative choiceless comparisons.

\subsection{Group duals}

For a countable group $H$, its unitary dual is the spectrum of the
full group algebra $C^*(H)$. Thomas \cite[p.~1762]{Thomas2015}
proposed the strict comparison between $F_\infty$ and
$G_0=\bigoplus_{n\in\N}\Sym(3)$. The measured obstruction answers
this question and treats every amenable target at once.

\begin{theoremB}\label{intro:B}
For every countable amenable group $H$, there is no Borel-definable
injection $\widehat{F_\infty}\to\widehat H$. Moreover
\[
                  \widehat{G_0}\sdef\widehat{F_\infty},
\]
and the same strict comparison holds on the infinite-dimensional
parts. A fixed probability on a Borel family of infinite-dimensional
irreducible representations of $F_\infty$ witnesses essential
constancy of every Borel-definable map into any amenable group dual.
\end{theoremB}

\subsection{Proof dependencies and reading routes}

Section~\ref{sec:duals} fixes the two categories and their
presentations. The main independence theorem combines two arguments
with different dependencies: the CH construction uses the $F_\sigma$
orbit presentation, while the Sacks argument also uses the measured
obstruction and the universal simple endpoint. The proofs are arranged
so that these inputs are available before they are used.
\begin{enumerate}[label=(\roman*),leftmargin=*]
\item \emph{Borel-definable separation.}
The trace-class arguments of Sections~\ref{sec:operators}--\ref{sec:averages}
give the measured rigidity theorem. The crossed-product and group
constructions in Sections~\ref{sec:crossed} and~\ref{sec:groups} then
prove Theorems~\ref{intro:A} and~\ref{intro:B}.
\item \emph{The common nuclear spectrum.}
The Cuntz--CAR comparison in Section~\ref{sec:groupoids} and Fock
induction in Section~\ref{sec:fock} feed into
Section~\ref{sec:nuclear-duals}, proving Theorem~\ref{intro:D}.
\item \emph{The positive CH comparison.}
The local fusion and synchronized recursion in
Sections~\ref{sec:ch-local} and~\ref{sec:ch-synchronized} prove a
theorem about $F_\sigma$ orbit quotients. Its application to spectra
uses the pure-state presentation of Lemma~\ref{lem:pure-state-fsigma}
and Glimm's nonsmoothness theorem; it does not depend on the nuclear
comparison or the rigidity theorem.
\item \emph{The negative comparison in ZFC.}
Section~\ref{sec:sacks-negative} combines the fixed spectral obstruction
with orbit-class incidence and finite-fork Sacks canonization. Reflection
of arbitrary quotient graphs and proper descent of Borel event codes
give the Sacks negative model and Theorem~\ref{intro:G}.
\item \emph{The negative choiceless comparisons.}
Section~\ref{sec:choiceless-obstruction} transfers the fixed CAR
measured obstruction to ZF and ZF+DC. The determinacy and Solovay
arguments in Sections~\ref{sec:determinacy} and~\ref{sec:solovay}
combine it with the point representations of $B$ and the respective
uniformization principles.
\item \emph{The positive Feferman--L\'evy comparison.}
The universal endpoint and its ZF coding in Section~\ref{sec:universal}
reduce the global assertion to fixed spectra. The orbit construction
and full Borel convention in Section~\ref{sec:positive-consistency}
then give the positive model.
\end{enumerate}
Section~\ref{sec:questions} records the questions that remain open,
including whether the positive measurable comparison can hold in
ZFC with $\neg\mathrm{CH}$.
Appendix~\ref{app:idealistic} compares the parameterized definitions
of idealistic equivalence relations and proves the statements about
inverse lifts and changes of presentation used here.

\section{Borel-definable and quotient measurable spectra}\label{sec:duals}

\subsection{Quotients and Borel-definable functions}

Unless another background theory is explicitly specified, we work in
ZFC. In the choiceless sections the spaces and algebras have fixed real
codes. There, Borel sets always mean the full smallest $\sigma$-algebra
containing the open sets; real-coded Borel sets are mentioned separately.
In ZFC, type I has its usual equivalent meanings. In every choiceless
argument, including those in ZF+DC, we use the GCR formulation: each
irreducible representation has image containing the compact operators
on its Hilbert space, as in Section~\ref{sec:zf-endpoints}. We do not
use an equivalence with the factor-representation formulation there.

We recall idealistic equivalence relations and definable sets from
\cite[Definition~3.1 and Section~3.2]{BLP2024}. We use ``Borel
homomorphism'' and ``Borel reduction'' for maps on the presentations;
the induced quotient maps are the homomorphisms and reductions of
\cite[Definition~3.5]{BLP2024}.
A sigma-filter on a set $C$ is a nonempty family of nonempty subsets of
$C$, closed under supersets in $C$ and countable intersections.

\begin{definition}
An equivalence relation $E$ on a Polish space $X$ is \emph{idealistic} if
there are sigma-filters $\mathcal F_C$ on its classes $C$ and a Borel map
$\zeta:X\to X$ with $xE\zeta(x)$, such that, for every Borel
$D\subseteq X\times X$, the set
\[
 \{x\in X:\{y\in[x]_E:(\zeta(x),y)\in D\}\in\mathcal F_{[x]_E}\}
\]
is Borel. A \emph{definable set} is a presentation $X/E$ with $X$ Polish
and $E$ Borel and idealistic.
\end{definition}

A standard Borel space can be equipped with a compatible Polish topology;
we use such presentations interchangeably when the identity maps are
Borel. A subset of $X/E$ is definable if its preimage in $X$ is Borel.

\begin{definition}
A function $F:X/E\to Y/F_0$ is \emph{Borel-definable} if there is a
Borel map $f:X\to Y$ such that
\[
 F([x]_E)=[f(x)]_{F_0}\qquad(x\in X).
\]
Such an $f$ is a \emph{Borel lift} of $F$.
\end{definition}

Thus a Borel-definable function is represented by a Borel homomorphism,
and it is injective precisely when its lift is a Borel reduction:
\[
 xEx'\quad\Longleftrightarrow\quad f(x)F_0f(x').
\]
The lift itself need not be injective. We use $\DSet$ for the category
with these objects and morphisms. For the Borel orbit presentations
used in this paper, a bijective Borel-definable map has a
Borel-definable inverse, and injections in both directions yield an
isomorphism. We prove the image and inverse assertions under the
hypothesis below that the source relation is strongly idealistic,
with the proof given in Appendix~\ref{app:idealistic}, and use them
in Proposition~\ref{prop:dual-cantor-bernstein}. No corresponding assertion
for arbitrary idealistic presentations is needed here.

We recall the stronger, parameterized hypothesis of
\cite[Introduction]{deRancourtMiller}. A Borel equivalence relation
$E$ on a Polish space $X$ is \emph{strongly idealistic} if there is an
$E$-invariant assignment $x\mapsto\mathcal I_x$ of sigma-ideals on
$X$, with $[x]_E\notin\mathcal I_x$, such that for every Polish space
$Z$ and Borel set $R\subseteq(X\times Z)\times X$ the set
\[
 \{(x,z)\in X\times Z:R_{(x,z)}\in\mathcal I_x\}
\]
is Borel. The external parameter $z$ is part of this requirement.

\begin{lemma}\label{lem:orbit-idealistic}
An orbit equivalence relation of a Borel action of a Polish group
on a standard Borel space is idealistic. If the relation is Borel, it
is strongly idealistic and its quotient is a definable set.
\end{lemma}

\begin{proof}
For an orbit $C=Gx$ and $S\subseteq C$, put $S\in\mathcal F_C$ if
$\{g\in G:gx\in S\}$ is comeager in $G$. This does not depend on the
choice of $x\in C$: replacing $x$ by $hx$ changes the preimage by right
translation. These families are sigma-filters, since $G$ is a Baire
space. Take $\zeta=\id$. For Borel $D\subseteq X\times X$, the set
\[
 \{(x,g):(x,gx)\in D\}
\]
is Borel. The Borel theorem for the category quantifier
\cite[Theorem~16.1]{Kechris1995} says that the set of $x$ for which its
$G$-section is comeager is Borel. This is the required condition.
To verify that the relation is strongly idealistic, put
\[
 \mathcal I_x=\{S\subseteq X:\{g\in G:gx\in S\}
                              \text{ is meager in }G\}.
\]
Right translation gives $E$-invariance, and $[x]_E$ is not in
$\mathcal I_x$. For a Borel $R\subseteq(X\times Z)\times X$, apply
the same category-quantifier theorem, now with parameter $(x,z)$, to
$\{(x,z,g):((x,z),gx)\in R\}$. This proves the required
parameterized Borelness. See also \cite[Section~1, p.~285]{Kechris1992}.
Any compatible Polish topology on the standard Borel space suffices
for these Borel tests; continuity of the action is not needed.
\end{proof}

\begin{lemma}[Images and inverse lifts]\label{lem:strong-inverse-lift}
Let $E$ on $X$ and $F$ on $Y$ be Borel equivalence relations on Polish
spaces, and suppose $E$ is strongly idealistic. If $f:X\to Y$ is a
Borel reduction, then for every invariant Borel set $A\subseteq X$
the saturated image $Y_A=[f(A)]_F$ is Borel and there is a Borel map
$g_A:Y_A\to A$ satisfying
\[
 f(g_A(y))\mathrel F y\quad(y\in Y_A),\qquad
 g_A(f(x))\mathrel E x\quad(x\in A).
\]
Thus the induced injection $X/E\to Y/F$ is a classwise Borel
embedding, and the same holds on every invariant Borel carrier.
\end{lemma}
\begin{proof}
Apply Corollary~\ref{cor:app-inverse-lifts} with $D=A$.
The appendix gives the parameterized section argument and records
the invariance under changes of presentation.
\end{proof}

\begin{lemma}[Smooth presentations with selectors]
\label{lem:smooth-strong-selector}
Let $E$ be a smooth strongly idealistic Borel equivalence relation on a
standard Borel space $X$, and let $c:X\to2^\N$ be a Borel classifier.
Then $T=c(X)$ is Borel and there is a Borel map $s:T\to X$ with
\[
                 c(s(t))=t,\qquad s(c(x))Ex.
\]
Thus $X/E$ is classwise Borel isomorphic to $T$, and $s\circ c$ is a
Borel selector. Every measurable map from a quotient Borel space
$W/R$, with $W$ standard Borel and $R$ Borel, into $X/E$ has a Borel
lift to $X$.
\end{lemma}
\begin{proof}
Apply Lemma~\ref{lem:strong-inverse-lift} to the reduction $c$ from
$E$ to equality on $2^\N$. This gives $T$ and $s$ with the displayed
identities. The induced map $\bar c:X/E\to T$ and its inverse are
measurable and have Borel lifts. If $h:W/R\to X/E$ is measurable and
$q_R:W\to W/R$ is the quotient map, then
$s\circ\bar c\circ h\circ q_R:W\to X$ is Borel and lifts $h$.
The same argument applies to invariant Borel restrictions of $E$,
since the strongly idealistic condition passes to such restrictions.
\end{proof}

\subsection{Irreducible representations}

Fix $\Hil_d=\C^d$ for $1\le d<\infty$ and $\Hil_\infty=\ell^2(\N)$,
with fixed orthonormal bases. For a separable unital $C^*$-algebra $D$,
let $\Rep_d(D)$ be its unital representations on $\Hil_d$, equipped
with pointwise strong-operator convergence. One may describe this Polish
space using a countable dense rational $*$-subalgebra of $D$ and strong-*
operator balls: the algebraic, contractivity and unitality conditions
are closed. Pointwise strong convergence on representations includes
strong convergence on adjoints and gives the same topology.

The irreducible representations form a $G_\delta$ subspace
$\Irr_d(D)$. Here is one description. Fix a countable strongly dense
family in the unit ball of $\BH(\Hil_d)$, a countable norm-dense
family in the unit ball of $D$, and a countable dense family of vectors.
Irreducibility is equivalent to the assertion
that every operator in the selected family in $\BH(\Hil_d)$ can be approximated on any
finite list of those vectors by an image of an element of the unit
ball of $D$. The equivalence follows from the bicommutant theorem and
Kaplansky density; allowing arbitrarily small norm errors in lifting a
contraction from $\pi(D)$ gives the same test with the unit ball of
$D$. The resulting countable intersection of unions of open conditions
is $G_\delta$.

Set
\[
 \Irr(D)=\coprod_{d\in\N_{\ge1}\cup\{\infty\}}\Irr_d(D),
\]
and let $E_D^{\mathrm{irr}}$ be unitary equivalence on this Polish space.

\begin{lemma}\label{lem:intertwiners}
The relation $E_D^{\mathrm{irr}}$ is Borel. On the set of equivalent
pairs in $\Irr_d(D)$ there is a Borel choice $U(\pi,\rho)$ of unitary
intertwiners from $\pi$ to $\rho$.
\end{lemma}

\begin{proof}
The unitary group $U(\Hil_d)$ with the strong topology is Polish, and
its group operations are continuous. The set of triples $(\pi,\rho,U)$
satisfying $U\pi(a)=\rho(a)U$ for every $a\in D$ is Borel: it suffices
to test the equations on a countable dense subalgebra and the fixed
Hilbert basis. By Schur's lemma, the fiber over an equivalent pair is
exactly one circle of scalar multiples of an intertwiner.

Enumerate the matrix entries of $U$. Require its first nonzero entry to
be a positive real number. This is a Borel condition and leaves exactly
one point in each nonempty fiber. The projection of this normalized
Borel graph onto the space of pairs is injective. By the Lusin--Souslin
theorem its image is Borel and its inverse is Borel
\cite[Corollary~15.2]{Kechris1995}. The image is the equivalence relation
on this dimension stratum. Taking the countable disjoint union proves
the assertion for $\Irr(D)$.
\end{proof}

The Polish group $\prod_d U(\Hil_d)$ acts continuously on $\Irr(D)$,
using its $d$-th coordinate on the $d$-th stratum. Its orbit relation is
$E_D^{\mathrm{irr}}$. Lemma~\ref{lem:orbit-idealistic} and
Lemma~\ref{lem:intertwiners} therefore justify the following definition.

\begin{definition}
The \emph{Borel-definable spectrum}, or simply the \emph{definable
spectrum}, of $D$ is
\[
 \defdual D=\Irr(D)/E_D^{\mathrm{irr}}.
\]
Its definable subset of infinite-dimensional classes is denoted
$\defdual D^{\,\infty}$.
\end{definition}

For a countable group $H$, define $\Irr_d(H)$ using the closed Polish
space of homomorphisms $H\to U(\Hil_d)$ and its irreducible subspace.
Restriction to the canonical group unitaries and integration to $C^*(H)$
give inverse Borel identifications
\begin{equation}\label{eq:group-dual}
 \defdual H\isoBD\defdual {C^*(H)}.
\end{equation}
For example, Borelness of integration follows first on the rational group
algebra and then by uniform norm approximation. The represented images
generate the same algebra, so irreducibility and equivalence are
preserved. The full group algebra is essential in \eqref{eq:group-dual}:
 arbitrary unitary representations need not factor through the reduced
 group algebra.

For a nonunital separable algebra $D$, throughout we extend each
nondegenerate irreducible representation to the unitization $\widetilde D$
and omit the one-dimensional augmentation class. This is an invariant
Borel restriction of the unital presentation. In particular, the unit and
the element $a_0=1$ used in the following GNS construction belong to
$\widetilde D$; restriction back to $D$ gives the asserted nonunital model.

\subsection{Pure-state presentations}

Write $P(D)$ for the pure states of $D$, with the weak-star topology. If
$D$ is unital, this is the $G_\delta$ set of extreme points of its compact
metrizable state space. If $D$ is nonunital, we identify $P(D)$ with
$P(\widetilde D)$ minus the augmentation state, in accordance with the
preceding convention. Thus $P(D)$ is Polish in either case. Let
\[
 \varphi E_D^{\mathrm p}\psi
 \quad\Longleftrightarrow\quad
 \pi_\varphi\simeq\pi_\psi
\]
for the corresponding GNS representations.

\begin{proposition}\label{prop:gns}
The quotient $P(D)/E_D^{\mathrm p}$ is a definable set and admits a
canonical classwise Borel isomorphism with $\defdual D$.
\end{proposition}

\begin{proof}
Choose a countable norm-dense rational $*$-subalgebra of $D$, enumerated
as $(a_j)$ with $a_0=1$. For each $\varphi$, apply Gram--Schmidt to
the GNS vectors $[a_j]_\varphi$, always choosing the least index with
nonzero residual. Their inner products are
$\varphi(a_i^*a_j)$, so all tests and coefficients are Borel. Whether
the procedure stops at dimension $d$ is also Borel. Identifying the
resulting basis with the fixed basis of $\Hil_d$ yields a Borel GNS
model $g(\varphi)$ whose cyclic vector is $e_0$. Matrix coefficients
are Borel on the dense algebra; norm approximation proves the assertion
on all of $D$. Columns then give the strong-operator Borel structure.

In the other direction, put
$v(\pi)(a)=\langle e_0,\pi(a)e_0\rangle$. This is a Borel pure state,
since every nonzero vector in an irreducible representation is cyclic.
We have $v(g(\varphi))=\varphi$ and $g(v(\pi))\simeq\pi$.
These maps implement inverse bijections on the quotient sets and provide
Borel lifts in both directions.
The relation $E_D^{\mathrm p}$ is Borel by Lemma~\ref{lem:intertwiners}
and the map $g$. The classwise Borel isomorphism above and
Proposition~\ref{prop:app-transport} show that $E_D^{\mathrm p}$ is
strongly idealistic. Thus the quotient is definable and the
induced bijection is Borel-definable. All the constructions use fixed
countable tests and operations, and therefore give real Borel codes.
\end{proof}

Here and throughout, the \emph{Mackey Borel structure} means the quotient
$\sigma$-algebra attached to the chosen Polish presentation of the spectrum.
Proposition~\ref{prop:gns} gives considerably more than an isomorphism of
these quotient measurable spaces: the representation-code and pure-state
presentations are connected by a classwise Borel isomorphism. No
presentation-independent standard Borel structure on the quotient is being
assumed.

For brevity we also write $F_D=E_D^{\mathrm p}$ for pure-state
equivalence. The following stronger regularity is the reason that
spectra fall within the CH theorem proved later.

\begin{lemma}[Pure-state orbit presentations]
\label{lem:pure-state-fsigma}
For every separable $C^*$-algebra $D$, equivalence $F_D$ of pure
states is an $F_\sigma$ relation on the weak-star Polish space
$P(D)$. It is the orbit relation of the continuous inner action
of the norm-Polish unitary group of $D$, using its unitization
when $D$ is nonunital.
\end{lemma}
\begin{proof}
Put $D^+=D$ when $D$ is unital, and let $D^+$ be its unitization
otherwise. The pure-state space of the unital algebra $D^+$ is
a $G_\delta$ subspace of its compact metrizable state space.
In the nonunital case, identify $P(D)$ with
$P(D^+)\setminus\{\chi_\infty\}$, where $\chi_\infty$ is the
character annihilating $D$. This is an invariant open Polish
subspace. The unitary group of $D^+$ is norm-Polish, its action
on pure states is jointly norm/weak-star continuous. Kadison's
transitivity theorem \cite[Theorem~1]{Kadison1957}, in the unitary
form of \cite[Theorem~1 and Corollary~8]{GlimmKadison1960}, identifies
its orbits with $F_D$.

Choose countable norm-dense sets $U\subseteq\mathcal U(D^+)$
and $S\subseteq(D^+)_{\leq1}$. Using the preceding identification
of pure states with their unital extensions, the norm criterion is
\[
 \varphi F_D\psi\quad\Longleftrightarrow\quad
 (\exists u\in U)(\forall a\in S)\,
       |\varphi(a)-\psi(uau^*)|\leq1.
\]
Indeed, equivalent pure states are related by an inner automorphism;
approximating its unitary by an element of $U$ makes their distance
less than $1$. Conversely, the displayed condition gives distance at
most $1$, whereas inequivalent pure states have distance $2$
\cite[Corollary~9]{GlimmKadison1960}. The converse implication
``distance $2$ implies inequivalence'' is not used and is false
\cite[Remark~10, p.~551]{GlimmKadison1960}.
For fixed $u$ the displayed condition is weak-star closed, also
after restriction to the nonunital pure-state space. This gives
an explicit closed-set version of the criterion in
\cite[Proposition~7]{Farah2011}: the non-strict inequality makes each
set in the countable union closed.
\end{proof}

\subsection{Bireducibility and Borel-definable bijections}

The passage from injections to a bijection uses a general theorem,
not an additional operator-algebraic property of the nuclear class.

\begin{proposition}\label{prop:dual-cantor-bernstein}
For any separable $C^*$-algebras $A,D$,
\[
 E_A^{\mathrm{irr}}\sim_B E_D^{\mathrm{irr}}
 \quad\Longleftrightarrow\quad
 \defdual A\isoBD\defdual D.
\]
Thus the displayed relation asserts a classwise Borel isomorphism,
with Borel lifts in both directions.
\end{proposition}
\begin{proof}
For Borel orbit equivalence relations of Borel Polish-group actions this is the
Cantor--Bernstein theorem attributed to Motto Ros in
\cite[Theorem~1.5]{Thomas2015}; Thomas records the corresponding unitary
dual statement in \cite[Corollary~1.7]{Thomas2015}. Both spectra have
the Borel orbit presentations established above. By
Lemma~\ref{lem:orbit-idealistic} these presentations are strongly
idealistic, so Lemma~\ref{lem:strong-inverse-lift} supplies Borel
saturated images and inverse lifts. We apply the classwise
Cantor--Bernstein construction of
\cite[Proposition~2.3]{MottoRos2012}. Its hypotheses are classwise
Borel embeddings; the following image and inverse-lift argument
upgrades the given reductions to those embeddings.

Suppose $f:\mathsf X\to\mathsf Y$ and
$g:\mathsf Y\to\mathsf X$ are the induced injections between these
orbit quotients. By Lemma~\ref{lem:strong-inverse-lift}, their images
are definable subsets and their inverse maps on the images have
Borel lifts. The same lemma applies after restriction to any
invariant Borel subset of either presenting space. Set
\[
 \mathsf X_0=\mathsf X\setminus g(\mathsf Y),\qquad
 \mathsf X_{n+1}=g(f(\mathsf X_n)),\qquad
 \mathsf X_* =\bigcup_{n\ge0}\mathsf X_n.
\]
All these subsets are definable, so the usual formula
\[
 h(x)=
 \begin{cases}
 f(x),&x\in\mathsf X_*,\\
 g^{-1}(x),&x\notin\mathsf X_*
 \end{cases}
\]
has a Borel lift, obtained by using the two branch lifts on the
corresponding invariant Borel subsets of the presenting space.
The second branch is defined since
$\mathsf X\setminus\mathsf X_*\subseteq g(\mathsf Y)$.
Injectivity of $g$ gives
\[
 g(y)\in\mathsf X_*\quad\Longleftrightarrow\quad
 y\in f(\mathsf X_*).
\]
Indeed $g(y)$ is never in $\mathsf X_0$, and its membership in
$\mathsf X_{n+1}$ is equivalent to $y\in f(\mathsf X_n)$.
Consequently $h$ is a bijection with inverse
\[
 h^{-1}(y)=
 \begin{cases}
 f^{-1}(y),&y\in f(\mathsf X_*),\\
 g(y),&y\notin f(\mathsf X_*).
 \end{cases}
\]
This inverse also has a Borel lift. Apply the construction to the
injections induced by the two Borel reductions. Conversely,
Borel lifts of inverse Borel-definable bijections give reductions
in both directions.
\end{proof}

\subsection{The quotient measurable category}

Write $Q_D$ for the spectrum of a separable $C^*$-algebra $D$, equipped
with the quotient $\sigma$-algebra from its standard Borel space of
irreducible representations (equivalently from its pure-state space).
A \emph{measurable embedding} $e:Q\qemb R$ means an injection such that
$e(Q)$ is measurable in $R$ and $e:Q\to e(Q)$ is measurable in both
directions. A measurable-space isomorphism is denoted by $\qiso$.
These definitions impose no Borel-lift requirement. We reserve $\bd$
for isomorphisms having Borel lifts in both directions.
We also call these \emph{classwise Borel isomorphisms}.
A \emph{classwise Borel embedding} is a measurable embedding
whose map and inverse on its image both have Borel lifts.

\begin{lemma}[Measurable Cantor--Bernstein]\label{lem:cb}
If there are measurable embeddings $f:Q\qemb R$ and $g:R\qemb Q$,
then $Q\qiso R$. If these are quotient Borel spaces and $f,g$ are
classwise Borel embeddings, the resulting isomorphism is classwise Borel.
\end{lemma}
\begin{proof}
Such embeddings send every measurable subset to a measurable subset:
use the inverse on the measurable image. Set
\[
 Q_0=Q\setminus g(R),\qquad Q_{n+1}=g(f(Q_n)),\qquad
 Q_*=\bigcup_{n\ge0}Q_n.
\]
These sets are measurable. The usual Cantor--Bernstein bijection is
$f$ on $Q_*$ and $g^{-1}$ on its complement. Its inverse is $f^{-1}$ on
$f(Q_*)$ and $g$ on the complement. Both are measurable. Under the
additional lift hypothesis, the preimages of these branch sets in the
presenting spaces are invariant Borel sets. Pasting the given lifts of
$f,g^{-1}$ on the source branches, and of $f^{-1},g$ on the target
branches, gives Borel lifts in both directions. This last assertion
uses the inverse lifts supplied with the embeddings.
\end{proof}

\begin{remark}\label{rem:embedding}
Let $E$ on $X$ and $F$ on $Y$ be Borel equivalence relations.
If the \emph{source} relation $E$ is strongly idealistic, a Borel
reduction $f:X\to Y$ induces a measurable embedding with a Borel
inverse lift, by Lemma~\ref{lem:strong-inverse-lift}. The target $F$
need only be Borel. Borel Polish-group orbit relations satisfy the
source hypothesis by Lemma~\ref{lem:orbit-idealistic}; this includes
the spectrum presentations used here. The conclusion applies on
invariant Borel carriers without changing the ambient presentation.
It starts with a Borel reduction and does not lift an arbitrary
measurable map of quotient spaces. Nor can the source hypothesis be
replaced by the same condition on the target alone, even for a quotient
bijection; see Remark~\ref{rem:app-one-way}.
\end{remark}

\section{Trace-class preliminaries}\label{sec:operators}

For a separable Hilbert space $\Hil$, let $\Schat_1(\Hil)$ and
$\Schat_2(\Hil)$ be the trace-class and Hilbert--Schmidt operators,
with their usual norms. For a unit vector $\xi$, write $P_\xi$ for
the rank-one orthogonal projection onto $\C\xi$. We use inner products
linear in the second variable.

\begin{samepage}
\begin{lemma}[Powers--St\o rmer]\label{lem:sqrt}
For positive $S,T\in\Schat_1(\Hil)$,
\begin{equation}\label{eq:PS}
 \norm{S^{1/2}-T^{1/2}}_2^2\le\norm{S-T}_1.
\end{equation}
In particular the positive square-root map is continuous from trace
norm to Hilbert--Schmidt norm.
\end{lemma}
\end{samepage}

\begin{proof}
We include the trace-class case of
\cite[Lemma~4.1]{PowersStormer1970}.
Put $a=S^{1/2}$, $b=T^{1/2}$, $c=a-b$, and $j=\operatorname{sgn}(c)$.
Since $c=a-b$, the identity $a^2-b^2=ac+cb$ holds without any
commutativity assumption on $a$ and $b$.
Trace duality and cyclicity for products of Hilbert--Schmidt operators give
\[
 \norm{S-T}_1
 \ge\left|\Tr\bigl(j(a^2-b^2)\bigr)\right|
 =\Tr\bigl(|c|(a+b)\bigr).
\]
The last trace is nonnegative. Writing $c=c_+-c_-$, we obtain
\begin{align*}
 \Tr\bigl(|c|(a+b)\bigr)-\Tr(c^2)
 &=\Tr\bigl(c_+(a+b-c)\bigr)
   +\Tr\bigl(c_-(a+b+c)\bigr)\\
 &=2\Tr(c_+b)+2\Tr(c_-a)\ge0.
\end{align*}
Every displayed product is trace-class. This proves \eqref{eq:PS}.
\end{proof}

\begin{lemma}\label{lem:common-state}
Let $\pi:D\to\BH(\Hil_\pi)$ and
$\rho:D\to\BH(\Hil_\rho)$ be irreducible representations.
Suppose a state $\theta$ of $D$ has the form
\[
 \theta(a)=\Tr(S\pi(a))=\Tr(T\rho(a))
\]
for positive trace-one operators $S$ and $T$. Then $\pi\simeq\rho$.
\end{lemma}

\begin{proof}
Diagonalize $S$, retaining only the positive eigenvalues, and put
\[
 S=\sum_j\lambda_jP_{\xi_j},\qquad
 \zeta_\pi=\sum_j\sqrt{\lambda_j}\xi_j\otimes e_j.
\]
Then $\zeta_\pi$ is a unit vector and
\[
 \langle\zeta_\pi,(\pi(a)\otimes I)\zeta_\pi\rangle
 =\sum_j\lambda_j\langle\xi_j,\pi(a)\xi_j\rangle
 =\theta(a).
\]
Consequently the map
\[
 W_\pi:\pi_\theta(a)\xi_\theta\longmapsto
       (\pi(a)\otimes I)\zeta_\pi
\]
extends from the cyclic GNS subspace to an isometric intertwiner from the
GNS representation $(\pi_\theta,H_\theta,\xi_\theta)$ into
$\pi\otimes I$. Its range is invariant under both the representation and
its adjoint, and hence is reducing. Repeating the construction for $T$
gives an isometric intertwiner $W_\rho:H_\theta\to
\Hil_\rho\otimes\ell^2$.

The nonzero partial isometry $V=W_\rho W_\pi^*$ therefore intertwines the
two amplifications. Relative to the standard basis of $\ell^2$, its
matrix slices
\[
 V_{kl}=(I\otimes\langle e_k,\cdot\rangle)V(I\otimes e_l)
 :\Hil_\pi\longrightarrow\Hil_\rho
\]
satisfy $V_{kl}\pi(a)=\rho(a)V_{kl}$ for every $a\in D$. At least one
slice is nonzero, since otherwise all matrix coefficients of $V$ would
vanish. Schur's lemma then says that this slice is a nonzero scalar
multiple of a unitary, and hence $\pi\simeq\rho$.
\end{proof}

No measurable choice of spectral decompositions is involved in this
lemma; it is a statement about each pair of representations. Also,
$\theta$ need not be pure. Its normality in both representations is
the essential hypothesis.

\section{A natural-averaging criterion for essential constancy}\label{sec:rigidity}

Let $D$ be a separable unital $C^*$-algebra. Suppose that for each
$\pi\in\Irr_d(D)$ there are positive, trace-preserving, trace-norm
contractive linear maps
\[
 \mathcal A_n^\pi:\Schat_1(\Hil_d)\longrightarrow\Schat_1(\Hil_d)
 \qquad(n\in\N)
\]
with the following properties:
\begin{enumerate}[label=(\roman*),leftmargin=*]
\item\label{hyp:borel} For each $d,n$, the map
$(\pi,T)\mapsto\mathcal A_n^\pi(T)$ is Borel in trace norm.
\item\label{hyp:natural} If $U\pi(a)U^*=\rho(a)$ for all $a\in D$, then
\[
 \mathcal A_n^\rho(UTU^*)=U\mathcal A_n^\pi(T)U^*.
\]
\item\label{hyp:mix} For every $\pi$ and unit vectors $\xi,\eta$,
\begin{equation}\label{eq:coalescence}
 \norm{\mathcal A_n^\pi(P_\xi-P_\eta)}_1\longrightarrow0.
\end{equation}
\end{enumerate}

\begin{theorem}\label{thm:collapse}
Under these hypotheses, let a countable property-$(T)$ group $\Lambda$
act by Borel automorphisms on a standard Borel probability space
$(X,\mu)$, preserving $\mu$ and acting ergodically. If
$x\mapsto\pi_x\in\Irr(D)$ is Borel and
\begin{equation}\label{eq:orbit-hom}
 \pi_{gx}\simeq\pi_x\qquad(g\in\Lambda,\ x\in X),
\end{equation}
then $\pi_x$ belongs to one unitary-equivalence class for almost every
$x\in X$.
\end{theorem}

The pointwise condition \eqref{eq:orbit-hom} is not an additional
uniformity assumption: it is exactly what a Borel homomorphism from the
orbit relation supplies, since every pair $(x,gx)$ belongs to that relation.

\begin{proof}
 The dimension is an invariant Borel function with countable range.
 Ergodicity therefore gives an invariant conull Borel set on which the
 dimension is one fixed $d$. Work on this set and put $\Hil=\Hil_d$.
 For each $g\in\Lambda$, evaluate the selector of
 Lemma~\ref{lem:intertwiners} at $(\pi_x,\pi_{gx})$. Since $\Lambda$
 is countable, this gives a jointly Borel field of unitary intertwiners
\begin{equation}\label{eq:U}
 U(g,x)\pi_x(a)U(g,x)^*=\pi_{gx}(a).
\end{equation}
Their composition defects are scalar:
\begin{equation}\label{eq:projective}
 U(g,hx)U(h,x)=c(g,h,x)U(gh,x),\qquad c(g,h,x)\in\T.
\end{equation}
We do not assert that the scalar cocycle can be trivialized.

Conjugation removes the scalars. The operators
\[
 \beta(g,x)T=U(g,x)TU(g,x)^*\qquad(T\in\Schat_2(\Hil))
\]
form a Borel unitary cocycle:
$\beta(g,hx)\beta(h,x)=\beta(gh,x)$.
For completeness, conjugation is jointly continuous from
$U(\Hil)\times\Schat_2(\Hil)$ to $\Schat_2(\Hil)$. This follows by first
verifying the assertion on rank-one operators and then using finite-rank approximation and the
isometry property. We obtain a unitary representation $V$ of $\Lambda$
on the fixed Hilbert space
\[
 \mathscr H=L^2(X,\mu;\Schat_2(\Hil))
\]
by
\begin{equation}\label{eq:V}
 (V_gF)(x)=\beta(g,g^{-1}x)F(g^{-1}x).
\end{equation}
The cocycle identity gives multiplication of these operators, and
probability preservation gives unitarity.

Fix the first basis vector $e_0$ of $\Hil$ and put
\[
 \varrho_n(x)=\mathcal A_n^{\pi_x}(P_{e_0}).
\]
These are Borel positive trace-one operators. Naturality gives
\[
 U(g,x)\varrho_n(x)U(g,x)^*
 =\mathcal A_n^{\pi_{gx}}(P_{U(g,x)e_0}).
\]
Thus \eqref{eq:coalescence}, applied in the single representation
$\pi_{gx}$, implies
\begin{equation}\label{eq:covariance-error}
 \norm{\varrho_n(gx)-U(g,x)\varrho_n(x)U(g,x)^*}_1\longrightarrow0
\end{equation}
for every $g,x$. The norm in \eqref{eq:covariance-error} is at most two.

Set $F_n(x)=\varrho_n(x)^{1/2}$. Lemma~\ref{lem:sqrt} gives Borel
Hilbert--Schmidt fields, and
\[
 \norm{F_n}_{\mathscr H}^2=\int_X\Tr(\varrho_n(x))\,d\mu(x)=1.
\]
Changing variables in \eqref{eq:V}, using \eqref{eq:PS}, and applying
dominated convergence, we obtain
\begin{align}
 \norm{V_gF_n-F_n}_{\mathscr H}^2
 &\le\int_X\norm{\varrho_n(gx)-U(g,x)\varrho_n(x)U(g,x)^*}_1\,d\mu(x)
 \longrightarrow0.\label{eq:almost}
\end{align}
This convergence holds uniformly on each finite subset of $\Lambda$.
Only the vectors $F_n$ vary with $n$; $V$ and $\mathscr H$ are fixed.
Property~$(T)$ supplies a nonzero invariant vector $F\in\mathscr H$.
We use here the invariant-vector characterization of property~$(T)$;
see \cite[Section~1.1 and Proposition~1.2.1]{BHV2008}.

 Choose a Borel representative of $F$. Countability of $\Lambda$ and
 probability preservation allow removal of one invariant null set so that
\[
 F(gx)=U(g,x)F(x)U(g,x)^*
\]
 holds simultaneously for all $g$. The invariant function
 $x\mapsto\norm{F(x)}_2^2$ is equal almost everywhere to a constant
 $c$. Since $F$ is nonzero and
 $\int\norm{F(x)}_2^2\,d\mu(x)=\norm F_{\mathscr H}^2>0$, one has
 $c>0$. On an invariant conull Borel set define
\[
 \varrho(x)=c^{-1}F(x)^*F(x).
\]
This is a Borel positive trace-one field with
\begin{equation}\label{eq:density-covariance}
 \varrho(gx)=U(g,x)\varrho(x)U(g,x)^*.
\end{equation}
Measurability follows, for example, from
\[
 \norm{S^*S-T^*T}_1\le(\norm S_2+\norm T_2)\norm{S-T}_2.
\]

Define states on $D$ by
\[
 \theta_x(a)=\Tr(\varrho(x)\pi_x(a)).
\]
Their evaluations are Borel: approximate the trace by finite diagonal
sums, using Borel operator columns and trace-norm measurability.
Equations~\eqref{eq:U} and~\eqref{eq:density-covariance} give
$\theta_{gx}(a)=\theta_x(a)$. Ergodicity, applied to a countable
norm-dense subset of $D$, yields a single state $\theta$ with
$\theta_x=\theta$ almost everywhere. It is normal in every $\pi_x$
on that conull set. Lemma~\ref{lem:common-state} now proves that all
these irreducible representations are equivalent.
\end{proof}

\begin{corollary}\label{cor:positive-restrictions}
Under the hypotheses of Theorem~\ref{thm:collapse}, let $S\subseteq X$
be Borel with $\mu(S)>0$. Every Borel homomorphism
$R_\Lambda^X\restr S\to E_D^{\mathrm{irr}}$ is essentially constant
modulo $E_D^{\mathrm{irr}}$ on $S$. If $\mu$ is nonatomic, then
\[
 R_\Lambda^X\restr S\not\le_B E_D^{\mathrm{irr}}.
\]
\end{corollary}

\begin{proof}
The saturation $Y=\Lambda S$ is Borel, invariant, and conull. Enumerate
 $\Lambda=\{g_n:n\ge0\}$ with $g_0=e$, and let $n(x)$ be the least
 index for which $g_{n(x)}x\in S$. This is Borel on $Y$. If $f$ is the
 given homomorphism, then
\[
 \widetilde f(x)=f(g_{n(x)}x)
 \]
 is an orbit homomorphism on $Y$. It agrees with $f$ on $S$, because
 $g_0=e$ makes $n(x)=0$ there. If $x,y\in Y$ are in the same orbit,
 then the two selected points $g_{n(x)}x,g_{n(y)}y\in S$ are in the
 same orbit, which verifies the homomorphism assertion.
 Apply Theorem~\ref{thm:collapse}. If $f$ were a reduction, its
 essential constancy would, by equivalence reflection, place a conull
 subset of $S$ in a single $\Lambda$-orbit. Such an orbit is countable,
 contradicting nonatomicity.
\end{proof}

\section{Averaging for nuclear algebras and amenable groups}\label{sec:averages}

\subsection{Positive approximate diagonals}

Write $D\widehat\otimes_{\mathrm{proj}}D$ for the Banach projective
tensor product, with norm $\norm{\cdot}_{\mathrm{proj}}$. Its module
operations and multiplication map are
\[
 a\cdot(x\otimes y)=ax\otimes y,\qquad
 (x\otimes y)\cdot a=x\otimes ya,\qquad
 m(x\otimes y)=xy.
\]
Haagerup's theorem \cite[Theorem~3.1]{Haagerup1983}, in its
approximate-diagonal form \cite[Lemma~3.1]{CSSWW2012}, says that for a
unital nuclear $D$, a finite $F\subseteq D$, and $\epsilon>0$, there is
\[
 r=\sum_{j=1}^k\lambda_jx_j^*\otimes x_j,
 \qquad \lambda_j\ge0,\quad\sum_j\lambda_j=1,\quad\norm{x_j}\le1,
\]
with $\norm{m(r)-1}<\epsilon$ and
$\norm{a\cdot r-r\cdot a}_{\mathrm{proj}}<\epsilon$ for $a\in F$.
This positive form, rather than an arbitrary bounded approximate
diagonal, supplies positivity of the trace-class maps below.

\begin{lemma}\label{lem:normalized-diagonal}
Every separable unital nuclear $C^*$-algebra $D$ has a sequence
\[
 d_n=\sum_{j=1}^{k_n}b_{nj}^*\otimes b_{nj}
 \quad\text{with}\quad \sum_jb_{nj}^*b_{nj}=1
\]
such that $\sup_n\norm{d_n}_{\mathrm{proj}}<\infty$ and
\begin{equation}\label{eq:diagonal-centrality}
 \norm{a\cdot d_n-d_n\cdot a}_{\mathrm{proj}}\longrightarrow0
 \qquad(a\in D).
\end{equation}
\end{lemma}

\begin{proof}
Apply the preceding result to the first $n$ elements of a fixed
countable norm-dense subset of $D$, with errors tending to zero and
less than $1/2$. Absorb $\sqrt{\lambda_j}$ into $x_j$ and write the
resulting tensors as
\[
 r_n=\sum_jc_{nj}^*\otimes c_{nj},\qquad
 s_n=m(r_n)=\sum_jc_{nj}^*c_{nj}.
\]
 We have $\norm{r_n}_{\mathrm{proj}}\le1$ and
 $\norm{s_n-1}<1/2$ for every $n$, with $s_n\to1$ in norm.
 Boundedness of $r_n$ extends approximate centrality from the dense
 subset to every $a\in D$. Each $s_n$ is positive and invertible.
Put $q_n=s_n^{-1/2}$ and $b_{nj}=c_{nj}q_n$. Then
\[
 d_n=q_n\cdot r_n\cdot q_n,\qquad m(d_n)=q_ns_nq_n=1.
\]
As $q_n\to1$,
\[
 \norm{d_n-r_n}_{\mathrm{proj}}
 \le\norm{q_n-1}(\norm{q_n}+1)\norm{r_n}_{\mathrm{proj}}
 \longrightarrow0.
\]
In particular $d_n$ is bounded, and
\[
 \norm{a\cdot d_n-d_n\cdot a}_{\mathrm{proj}}
 \le\norm{a\cdot r_n-r_n\cdot a}_{\mathrm{proj}}
       +2\norm a\norm{d_n-r_n}_{\mathrm{proj}}\longrightarrow0.
\]
\end{proof}

\begin{proposition}\label{prop:nuclear}
Every separable unital nuclear $C^*$-algebra satisfies the hypotheses
of Theorem~\ref{thm:collapse}. For $\pi\in\Irr_d(D)$ the maps can be
chosen as
\begin{equation}\label{eq:nuclear-average}
 \mathcal A_n^\pi(T)=\sum_{j=1}^{k_n}
    \pi(b_{nj})T\pi(b_{nj})^*,
\end{equation}
where $b_{nj}$ are as in Lemma~\ref{lem:normalized-diagonal}.
In fact $\norm{\mathcal A_n^\pi(T)}_1\to0$ for every $T$ with
$\Tr(T)=0$.
\end{proposition}

\begin{proof}
The maps are completely positive and preserve trace, since
$\sum_j\pi(b_{nj})^*\pi(b_{nj})=I$. Their adjoints on
$\BH(\Hil_d)$ are the unital completely positive contractions
$S\mapsto\sum_j\pi(b_{nj})^*S\pi(b_{nj})$; hence the maps in
\eqref{eq:nuclear-average} are contractions in trace norm.
They are natural under unitary intertwiners. For fixed $n$, they are
jointly continuous in $\pi$ and $T$, for pointwise strong-* convergence
of representations and trace-norm convergence of $T$. Indeed finite-rank
approximation reduces multiplication to convergence on finitely many
vectors. In particular the required joint Borel dependence holds.

Fix $\pi$, $T\in\Schat_1(\Hil_d)$, and $a\in D$. The bounded
linear map
\[
 Q_{\pi,T}:D\widehat\otimes_{\mathrm{proj}}D\longrightarrow
 \Schat_1(\Hil_d),\qquad Q_{\pi,T}(x\otimes y)=\pi(y)T\pi(x)
\]
has norm at most $\norm T_1$. The order of its factors gives
\[
 Q_{\pi,T}(a\cdot d_n-d_n\cdot a)
 =\mathcal A_n^\pi\bigl(T\pi(a)-\pi(a)T\bigr).
\]
Consequently
\begin{equation}\label{eq:nuclear-input-commutators}
 \norm{\mathcal A_n^\pi(T\pi(a)-\pi(a)T)}_1
 \le\norm T_1\norm{a\cdot d_n-d_n\cdot a}_{\mathrm{proj}}
 \longrightarrow0.
\end{equation}

 Let $W$ be the trace-norm closed linear span of these input
 commutators. Under $\Schat_1(\Hil_d)^*=\BH(\Hil_d)$, cyclicity of
 the trace shows that $W^\perp=\pi(D)'=\C I$. Hahn--Banach therefore
 gives $W=\ker\Tr$: indeed $(\ker\Tr)^\perp=\C I$, and two closed
 subspaces of a Banach space with the same annihilator are equal.
 Equation~\eqref{eq:nuclear-input-commutators}
and contractivity show that $\mathcal A_n^\pi(T)\to0$ in trace norm
for every $T\in W$. Applying this to $P_\xi-P_\eta$ proves
\eqref{eq:coalescence}.
\end{proof}

\begin{remark}\label{rem:nuclear-orientation}
The tensor factors must be applied in the displayed order. The maps
$T\mapsto\sum_j\pi(b_{nj})^*T\pi(b_{nj})$ need not preserve trace:
their trace involves $\sum_jb_{nj}b_{nj}^*$, which need not tend to
$1$. We require neither a symmetric approximate diagonal nor an
average of unitary conjugations. Thus the argument applies also to
nuclear algebras without tracial states, including the Cuntz algebras.
\end{remark}

\subsection{Finite-dimensional averaging}

\begin{proposition}\label{prop:AF}
Every separable unital AF algebra satisfies the hypotheses of
Theorem~\ref{thm:collapse}.
\end{proposition}

\begin{proof}
Write $D=\overline{\bigcup_nD_n}$ with increasing unital
finite-dimensional subalgebras. For $\pi\in\Irr_d(D)$ define
\begin{equation}\label{eq:AF-average}
 \mathcal A_n^\pi(T)=\int_{U(D_n)}\pi(u)T\pi(u)^*\,dm_n(u),
\end{equation}
where $m_n$ is normalized Haar measure. These maps are positive,
trace preserving, trace-norm contractive, and natural under intertwiners.
If $D_n=\bigoplus_rM_{k_r}$ has matrix units $e_{ij}^{(r)}$, then
\begin{equation}\label{eq:matrix-average}
 \mathcal A_n^\pi(T)
 =\sum_r\frac1{k_r}\sum_{i,j=1}^{k_r}
   \pi(e_{ij}^{(r)})T\pi(e_{ji}^{(r)}).
\end{equation}
The central phases kill off-block terms, and the ordinary unitary
average on each matrix block gives the displayed expression. It proves
Borel dependence without parameterized integration. Multiplication of a
trace-class operator by uniformly bounded strongly-* convergent
operators is trace-norm continuous, by finite-rank approximation.

Fix an irreducible $\pi$ and unit vectors $\xi,\eta$. We verify that the
vectors $\pi(v)\xi$, for $v\in\bigcup_nU(D_n)$, are dense in the unit
sphere. Choose a unitary $W\in\BH(\Hil_\pi)$ with $W\xi=\eta$ and a
bounded selfadjoint $h$ with $W=e^{ih}$. Irreducibility and the
bicommutant theorem give $\pi(D)''=\BH(\Hil_\pi)$. By the selfadjoint
form of Kaplansky density there is a uniformly bounded net of
selfadjoint $h_\lambda\in\pi(D)$ converging strongly to $h$. Continuous
functional calculus, uniformly approximated by polynomials on a common
bounded interval, gives $e^{ih_\lambda}\xi\to W\xi$.

For a chosen $\lambda$, take a selfadjoint lift $a_\lambda\in D$ of
$h_\lambda$. Since $\bigcup_nD_n$ is norm dense in $D$, approximate
$a_\lambda$ in norm by a selfadjoint $a\in D_m$. The elementary bound
$\|e^{ia}-e^{ia_\lambda}\|\le\|a-a_\lambda\|$ then shows that
$\pi(e^{ia})\xi$ is as close to $\eta$ as desired, and
$e^{ia}\in U(D_m)$. This proves the density assertion without assuming
that $\pi$ is faithful.

Given $\epsilon>0$, choose $v\in U(D_m)$ with
$\norm{\pi(v)\xi-\eta}<\epsilon$. For $n\ge m$, Haar invariance
gives $\mathcal A_n^\pi(P_{\pi(v)\xi})=\mathcal A_n^\pi(P_\xi)$.
Consequently
\[
 \norm{\mathcal A_n^\pi(P_\eta-P_\xi)}_1
 \le\norm{P_\eta-P_{\pi(v)\xi}}_1
 \le2\norm{\eta-\pi(v)\xi}<2\epsilon.
\]
This proves \eqref{eq:coalescence}.
\end{proof}

For the CAR algebra, $D_n=M_{2^n}$, so the density used in the proof
of Theorem~\ref{thm:collapse} has the particularly explicit form
\[
 \varrho_n(x)=2^{-n}\sum_{i,j=1}^{2^n}
 \pi_x(e_{ij})P_{e_0}\pi_x(e_{ji}).
\]
Its trace is one, since
$2^{-n}\sum_{i,j}e_{ji}e_{ij}=1$. No assumption that the input pure
states are product states is made anywhere in this construction.

\subsection{F\o lner averaging}

\begin{proposition}\label{prop:amenable}
If $H$ is a countable amenable group, then $C^*(H)$ satisfies the
hypotheses of Theorem~\ref{thm:collapse}.
\end{proposition}

\begin{proof}
Choose nonempty finite right F\o lner sets $F_n\subseteq H$. For an
irreducible representation $\pi$ put
\begin{equation}\label{eq:Folner-average}
 \mathcal A_n^\pi(T)
 =\frac1{|F_n|}\sum_{h\in F_n}\pi(h)T\pi(h)^*.
\end{equation}
Borel dependence, naturality, positivity, trace preservation and
contractivity follow directly from the finite sum.

Fix $\pi$ and let $W$ be the trace-norm closed linear span of
\[
 \{T-\pi(h)T\pi(h)^*:T\in\Schat_1(\Hil),\ h\in H\}.
\]
Under the duality $\Schat_1(\Hil)^*=\BH(\Hil)$, its annihilator is
$\pi(H)'=\C I$. Indeed, moving $\pi(h)$ across the trace identifies
the annihilator equations with commutation. The Hahn--Banach theorem
therefore gives
\begin{equation}\label{eq:trace-zero}
 W=\{T\in\Schat_1(\Hil):\Tr(T)=0\}.
\end{equation}
For a generating difference, right F\o lner invariance gives
\[
 \norm{\mathcal A_n^\pi(T-\pi(h)T\pi(h)^*)}_1
 \le\frac{|F_n\mathbin\triangle F_nh|}{|F_n|}\norm T_1
 \longrightarrow0.
\]
The side of translation is important: the second average involves
the products $kh$, with $k\in F_n$. The convergence extends to
finite linear combinations, then to $W$ by contractivity.
Since $P_\xi-P_\eta$ has trace zero, \eqref{eq:trace-zero} proves
\eqref{eq:coalescence}.
\end{proof}

Theorem~\ref{intro:C} is now Theorem~\ref{thm:collapse} and
Corollary~\ref{cor:positive-restrictions}, applied using
Proposition~\ref{prop:nuclear}, with Proposition~\ref{prop:gns} when
pure-state presentations are desired. For an amenable group $H$, apply
this algebraic case to the nuclear algebra $C^*(H)$ and use
\eqref{eq:group-dual}. Propositions~\ref{prop:AF} and
\ref{prop:amenable} retain the finite-dimensional and F\o lner formulas
as direct special-case proofs.
For a nonunital separable nuclear algebra, pass to its nuclear
unitization: its nondegenerate irreducible representations extend
unitally and form the invariant Borel subset obtained by removing the
augmentation class. Restricting the conclusion proves the nonunital case.

\section{The simple crossed product and Dixmier's problem}\label{sec:crossed}

\subsection{A free minimal probability-preserving action}

Use the groups and completion in \eqref{eq:example}. The subgroups
$\Gamma_n$ are normal and of finite index, and
$\bigcap_n\Gamma_n=\{e\}$: for a matrix in the intersection, every
off-diagonal entry and every diagonal entry minus $1$ is divisible by
all $2^n$, hence is zero. Thus the canonical map $\Gamma\to K$ is
injective. Its image
is dense, since each nonempty finite-level cylinder contains a coset
representative from $\Gamma$. In particular $K$ is an infinite compact
metrizable group. Denote its normalized Haar measure by $m$.

\begin{lemma}\label{lem:compact-action}
The left translation action $\Gamma\curvearrowright K$ is free,
minimal, and ergodic with respect to the invariant nonatomic probability
$m$. Moreover, $\Gamma$ has property~$(T)$.
\end{lemma}

\begin{proof}
Freeness follows by cancellation in $K$, and minimality follows from
density of $\Gamma$. Haar measure is invariant. If an $L^2(K,m)$
function is fixed by $\Gamma$, strong continuity of the translation
representation makes it fixed by $K$. Averaging over $K$ shows that
 it is constant, proving ergodicity. Every point has Haar measure zero:
 translation invariance gives all points the same mass, and a positive
 common mass would force the infinite set $K$ to have infinite total
 mass. A Borel probability measure on a compact metric space is nonatomic
 exactly when every singleton is null. Property~$(T)$ for $\SL_3(\Z)$ is the
classical Kazhdan theorem; see \cite[Theorem~1.4.15, Theorem~1.7.1 and
Example~1.7.4]{BHV2008}.
\end{proof}

\subsection{Regular representations attached to points}

Set $B=C(K)\rtimes_r\Gamma$, with action
$\alpha_s(f)(z)=f(s^{-1}z)$ and canonical unitaries $u_s$. For $x\in K$
define a representation on $\ell^2(\Gamma)$ by
\begin{equation}\label{eq:point-representation}
 \pi_x(f)\delta_t=f(tx)\delta_t,\qquad
 \pi_x(u_s)\delta_t=\delta_{st}.
\end{equation}
These are the regular representations induced from evaluations at $x$.
In particular, they factor through the reduced crossed product. Its
norm on finite sums is the supremum of the norms of these
representations: the direct sum of all point evaluations is faithful on
$C(K)$, and its regular induction is faithful on the reduced crossed
product by the defining induced-representation construction.

Let $\E:B\to C(K)$ be the canonical conditional expectation,
$\E(\sum_s f_su_s)=f_e$. It is faithful. One can see this directly
from
\[
 \E(b^*b)(x)=\norm{\pi_x(b)\delta_e}^2.
\]
If these quantities all vanish, use right translation on $\ell^2(\Gamma)$
and replace $x$ by $tx$ to obtain vanishing on every basis vector
$\delta_t$. Hence every $\pi_x(b)$ is zero, and so $b=0$.

\begin{samepage}
\begin{proposition}\label{prop:point-representations}
Each $\pi_x$ is irreducible, and
\begin{equation}\label{eq:orbit-equivalence}
 \pi_x\simeq\pi_y\quad\Longleftrightarrow\quad y\in\Gamma x.
\end{equation}
The map
\begin{equation}\label{eq:point-state}
 i:K\longrightarrow P(B),\qquad i(x)=\varphi_x=\operatorname{ev}_x\circ\E,
\end{equation}
is continuous and injective and reduces $R_\Gamma^K$ to $E_B^{\mathrm p}$.
\end{proposition}
\end{samepage}

\begin{proof}
If an operator $T$ commutes with $\pi_x(C(K))$, its $(r,t)$ matrix
entry satisfies
\[
 (f(rx)-f(tx))T_{r,t}=0\qquad(f\in C(K)).
\]
Freeness makes $rx\ne tx$ when $r\ne t$, so continuous functions
force all off-diagonal entries to vanish. A diagonal operator commuting
also with the left translations is scalar. Thus $\pi_x$ is irreducible.

If $y=sx$, the unitary $W_s\delta_t=\delta_{ts^{-1}}$ intertwines
$\pi_x$ with $\pi_y$. Conversely, if $W$ is such an intertwiner, then
$W\delta_e$ is a nonzero common eigenvector for $\pi_y(C(K))$ with
eigencharacter $f\mapsto f(x)$. Any nonzero coordinate at $\delta_t$
therefore implies $f(ty)=f(x)$ for every $f\in C(K)$, whence $ty=x$.
This proves \eqref{eq:orbit-equivalence}.

The vector $\delta_e$ is cyclic for $\pi_x$, since its translates
are the Hilbert basis. Its vector state is $\varphi_x$, so this state
is pure and its GNS representation is $\pi_x$. Evaluations of $i$ at
$b\in B$ are the continuous functions $\E(b)(x)$, and restriction
to $C(K)$ recovers $x$. This proves continuity, injectivity and the
reduction assertion.
\end{proof}

The same finite-sum approximation shows that $x\mapsto\pi_x(b)$ is
strongly continuous for every $b\in B$: first verify this on finitely
supported vectors and finite crossed-product sums, then use uniform
operator-norm bounds. We will use this observation for the group
representations in Section~\ref{sec:groups}.

\subsection{Simplicity}

\begin{proposition}\label{prop:simple}
The algebra $B$ is simple, separable, unital, exact, nonnuclear and
antiliminary.
\end{proposition}

\begin{proof}
Simplicity is a special case of the minimal topologically free
crossed-product theorem of Archbold--Spielberg
\cite[Theorem~1 and the following Corollary]{ArchboldSpielberg1994}.
We include the argument in the present setting.

Let $J$ be a closed two-sided ideal with $J\cap C(K)=0$, and let
$q:B\to B/J$ be the quotient map. It is isometric on $C(K)$. We claim
that $\norm{\E(b)}\le\norm{q(b)}$ for $b\in B$. Fix $\epsilon>0$
and a finite sum $c=\sum_{s\in F}f_su_s$ with
$\norm{b-c}<\epsilon$. Choose $x$ with
$|\E(b)(x)|=\norm{\E(b)}$. Freeness supplies a neighborhood $V$ of
$x$ such that $V\cap sV=\varnothing$ for $s\in F\setminus\{e\}$.
Choose $h\in C(K)$, $0\le h\le1$, supported in $V$, with $h(x)=1$.
Then $hch=h\E(c)h$, and
\[
 \norm{hbh-h\E(b)h}<2\epsilon.
\]
It follows that
\begin{align*}
 \norm{\E(b)}
 &=\norm{h\E(b)h}
 =\norm{q(h\E(b)h)}\\
 &\le\norm{q(hbh)}+2\epsilon
 \le\norm{q(b)}+2\epsilon.
\end{align*}
Let $\epsilon$ tend to zero. If $b\in J$, apply the resulting
inequality to $b^*b$ and use faithfulness of $\E$ to get $b=0$.
Thus every nonzero ideal intersects $C(K)$ nontrivially.

For a nonzero ideal $J$, the intersection $J\cap C(K)$ corresponds
to a nonempty invariant open subset of $K$. Minimality implies that
this subset is $K$, so $1\in J$ and $J=B$.

The algebra is separable and unital because $K$ is compact metrizable
and $\Gamma$ is countable. Guentner--Higson--Weinberger prove, for every
countable linear group, exactness of the reduced group algebra
\cite[Theorem~6 and the discussion following it]{GHW2005}.
For a discrete group $\Gamma$, Kirchberg--Wassermann
\cite[definition on p.~170 and Theorem~5.2]{KirchbergWassermann1999}
state that exactness of
$C_r^*(\Gamma)$ is equivalent to exactness of the group: the reduced
crossed-product functor preserves every short exact sequence of
$\Gamma$-$C^*$-algebras. We apply this formulation directly to verify
exactness of $B$. Given an
exact sequence $0\to I\to D\to D/I\to0$, exactness of $C(K)$ gives
\[
 0\longrightarrow C(K)\otimes_{\min}I
 \longrightarrow C(K)\otimes_{\min}D
 \longrightarrow C(K)\otimes_{\min}(D/I)\longrightarrow0.
\]
Equip the first tensor factor with the translation action and the second
with the trivial action. Taking reduced crossed products preserves this
sequence because $\Gamma$ is exact. There is a canonical identification
\[
 (C(K)\otimes_{\min}D)\rtimes_r\Gamma
 \cong (C(K)\rtimes_r\Gamma)\otimes_{\min}D.
\]
Indeed, since the action on $D$ is trivial, after faithful representations
of $C(K)$ and $D$ are fixed, the regular covariant representation of
$C(K)\otimes_{\min}D$ is the
tensor product of the regular representation of the first dynamical system
with the representation of $D$. Hence the two sides have the same reduced
norm on the common algebraic crossed product.
The analogous identifications with $I$ and $D/I$ show that tensoring
$B$ with the original exact sequence remains exact. Thus $B$ is exact.

 Let $m:C(K)\to\C$ be integration against Haar measure. It is a
 $\Gamma$-equivariant unital completely positive map, where $\Gamma$ acts
 trivially on $\C$. We spell out the resulting reduced-crossed-product
 map. Let $(\sigma_m,H_m,\xi_m)$ be its GNS representation. Invariance
 supplies a unitary representation $v$ of $\Gamma$ such that
 \[
  v_s\sigma_m(f)v_s^*=\sigma_m(\alpha_s(f)),\qquad v_s\xi_m=\xi_m.
 \]
 The covariant representation
 $(\sigma_m\otimes I,v\otimes\lambda)$ on
 $H_m\otimes\ell^2(\Gamma)$ factors through the reduced crossed
 product by the regular-representation construction (equivalently, by
 Fell absorption). Compress it by the isometry
 $V:\ell^2(\Gamma)\to H_m\otimes\ell^2(\Gamma)$,
 $V\delta_t=\xi_m\otimes\delta_t$. This gives the unital completely
 positive map
 \[
  \Phi_m:B\longrightarrow C_r^*(\Gamma),\qquad
  \Phi_m\!\left(\sum_s f_su_s\right)=\sum_s m(f_s)\lambda_s,
 \]
 since
 $V^*(\sigma_m(f_s)v_s\otimes\lambda_s)V=m(f_s)\lambda_s$.
 The map fixes the canonical copy of $C_r^*(\Gamma)$, so it is an
 idempotent norm-one projection and hence a conditional expectation.
 If $B$ were nuclear, then $C_r^*(\Gamma)$ would be nuclear: compose
 completely positive matrix factorizations for the inclusion
 $C_r^*(\Gamma)\subseteq B$ with $\Phi_m$ on the return map.
 Here the completely positive approximation property is equivalent
 to uniqueness of the $C^*$-tensor norm by the Choi--Effros--Kirchberg
 theorem \cite[Theorem~3.8.7]{BrownOzawa2008}.
 Lance's criterion \cite[Theorem~4.2]{Lance1973} would make $\Gamma$
 amenable. To recall why the last possibility is excluded, amenability
 supplies almost invariant unit vectors for the left regular
 representation on $\ell^2(\Gamma)$, while property~$(T)$ turns such a
 family into a nonzero invariant vector. An invariant vector in
 $\ell^2(\Gamma)$ is constant on $\Gamma$ and must vanish when $\Gamma$
 is infinite. Hence the infinite property-$(T)$ group $\SL_3(\Z)$ is
 nonamenable, a contradiction.

Finally, $B$ is infinite-dimensional because it
contains $C(K)$. In a simple infinite-dimensional unital algebra every
irreducible representation is faithful and infinite-dimensional.
The intersection of its image with the compact operators is an ideal
in that image. If it were nonzero, simplicity would force the identity
to be compact. This is impossible in infinite dimension. Thus no
irreducible image contains a nonzero compact operator; in particular
$B$ is antiliminary and not type~I.
\end{proof}

The CAR algebra is also simple, separable, unital and
infinite-dimensional, hence antiliminary by the same argument.

\subsection{Strict comparison and essential constancy}

We use the following classical lower bound, in its pure-state form:
if $D$ is separable and non-type-I, then
\begin{equation}\label{eq:Glimm-lower}
 E_{\CAR}^{\mathrm p}\le_B E_D^{\mathrm p}.
\end{equation}
The reduction in precisely this direction is stated in
\cite[Proposition~3]{Farah2011}, using the nonsmoothness alternative
of Glimm's theorem and Elliott's AF comparison. We use that explicit
Borel reduction; nonsmoothness alone would give only an $E_0$ lower bound.

\begin{theorem}\label{thm:simple-separation}
For $A=\CAR$ and $B$ above, $\defdual A\sdef\defdual B$.
Moreover, the probability $\nu=i_*m$ on $P(B)$ satisfies
\begin{enumerate}[label=(\roman*),leftmargin=*]
\item $\nu([\varphi]_{E_B^{\mathrm p}})=0$ for every $\varphi\in P(B)$;
\item for every Borel homomorphism $f:E_B^{\mathrm p}\to E_A^{\mathrm p}$
there is $\psi\in P(A)$ with
\[
 \nu\{\varphi:f(\varphi)E_A^{\mathrm p}\psi\}=1;
\]
\item for every Borel $D\subseteq P(B)$ with $\nu(D)>0$, the analogous
essential-constancy assertion holds for homomorphisms on
$E_B^{\mathrm p}\restr D$, with respect to $\nu\restr D$.
\end{enumerate}
\end{theorem}

\begin{proof}
By Proposition~\ref{prop:point-representations}, the preimage under $i$
of any $E_B^{\mathrm p}$-class is empty or a single $\Gamma$-orbit.
It is therefore countable and Haar null, proving~(i).
For~(ii), the map $f\circ i$ is a homomorphism from $R_\Gamma^K$
to $E_A^{\mathrm p}$. Apply Theorem~\ref{thm:collapse},
Propositions~\ref{prop:gns} and~\ref{prop:AF}, and
Lemma~\ref{lem:compact-action}. For~(iii), use
Corollary~\ref{cor:positive-restrictions} on $i^{-1}(D)$.
The set whose measure occurs in~(ii) is Borel, since $f$ and
$E_A^{\mathrm p}$ are Borel by hypothesis and Proposition~\ref{prop:gns}.

If a Borel-definable injection $\defdual B\to\defdual A$ existed,
Proposition~\ref{prop:gns} would give a Borel reduction between these
pure-state presentations. By~(ii), a conull set would be contained in
one source class, contradicting~(i). Conversely,
Proposition~\ref{prop:simple} and \eqref{eq:Glimm-lower} give the
injection $\defdual A\to\defdual B$.
\end{proof}

This proves Theorem~\ref{intro:A} with the simplicity hypothesis present in
Elliott's account of Dixmier's question. The probability $\nu$ is
specified by the compact family of point states; it is not asserted
to be quasi-invariant under every inner automorphism of $B$. The
essential constancy concerns representation classes, not equality of
the output pure states. Since the reverse Borel reduction does not exist,
there is in particular no Borel bijection between the representation-code
spaces of $A$ and $B$ whose forward and inverse maps preserve unitary
equivalence, which is the isomorphism formulation used by Farah.

\section{Unitary duals of the free group and the direct sum of symmetric groups}\label{sec:groups}

\subsection{Pullback from a free-group algebra}

\begin{lemma}\label{lem:free-quotient}
Every separable unital $C^*$-algebra $D$ is a quotient of $C^*(F_\infty)$.
For any unital surjection $q:C^*(F_\infty)\to D$, pullback gives a
Borel-definable injection $\defdual D\to\defdual {F_\infty}$ and
preserves the dimension of each representation.
\end{lemma}

\begin{proof}
Choose a countable generating family of selfadjoint contractions $a_n$
of $D$. The elements
\[
 v_n=a_n+i(1-a_n^2)^{1/2}
\]
are unitaries and generate $D$, because $a_n=(v_n+v_n^*)/2$.
Sending the free generators to $v_n$ defines the desired surjection.
The representations $\pi\circ q$ and $\pi$ have the same represented
image, so irreducibility is preserved. Surjectivity of $q$ implies
that a unitary intertwines $\pi\circ q$ and $\rho\circ q$ exactly
when it intertwines $\pi$ and $\rho$. The pullback is Borel, and its
Hilbert space is unchanged. Finally use \eqref{eq:group-dual}.
\end{proof}

Fix such a $q:C^*(F_\infty)\to B$ and let $\sigma_x$ be the
representation of $F_\infty$ obtained from $\pi_x\circ q$. After one
fixed identification of $\ell^2(\Gamma)$ with $\Hil_\infty$, the map
\begin{equation}\label{eq:j}
 j:K\to\Irr_\infty(F_\infty),\qquad x\mapsto\sigma_x
\end{equation}
is continuous. It is injective: equality of $\sigma_x$ and $\sigma_y$
implies equality on $q(C^*(F_\infty))=B$, and evaluation at $\delta_e$
on $C(K)$ recovers $x=y$. Moreover,
\begin{equation}\label{eq:group-orbit}
 \sigma_x\simeq\sigma_y\quad\Longleftrightarrow\quad y\in\Gamma x.
\end{equation}

\begin{theorem}\label{thm:group-separation}
Let $H$ be any countable amenable group. Then
\[
 \defdual H\sdef\defdual {F_\infty}.
\]
There is a probability $\eta=j_*m$ on $\Irr_\infty(F_\infty)$ such
that every equivalence class is $\eta$-null and every Borel homomorphism
from irreducible-representation equivalence of $F_\infty$ to that of
$H$ is essentially constant modulo equivalence for $\eta$. The latter
assertion holds on every $\eta$-positive Borel restriction as well.
\end{theorem}

\begin{proof}
Choose a group surjection $p:F_\infty\to H$. Pullback
$\rho\mapsto\rho\circ p$ preserves irreducibility and preserves and
reflects unitary equivalence. This gives the forward Borel-definable
injection, including all dimension strata.

Let $f$ be a Borel homomorphism in the reverse direction. Its
composition with $j$ is an orbit homomorphism from $R_\Gamma^K$ to
irreducible equivalence for $H$. Apply
Theorem~\ref{thm:collapse} and Proposition~\ref{prop:amenable}.
Equation~\eqref{eq:group-orbit} shows that every source class has
$\eta$-measure zero. Thus $f$ cannot reflect equivalence, proving
that no reverse injection exists. For positive restrictions apply
Corollary~\ref{cor:positive-restrictions} to the preimage under $j$.
\end{proof}

This argument uses a property-$(T)$ action as a family of irreducible
representations of $F_\infty$. It does not apply property~$(T)$ to the
free group itself, and it makes no universality assertion about the
unitary dual of $\Gamma=\SL_3(\Z)$.

\subsection{The group \texorpdfstring{$\bigoplus_{\N}\Sym(3)$}{direct sum of Sym(3)}}

The group $G_0$ is the increasing union of its finite products
$G_{0,n}=\Sym(3)^n$. These subgroups are right F\o lner sets: for
any fixed $h\in G_0$, one has $G_{0,n}h=G_{0,n}$ for all sufficiently
large $n$. Thus $G_0$ is amenable. Its full group algebra is
\begin{equation}\label{eq:S3-algebra}
 C^*(G_0)=\bigotimes_{n\in\N}
 \bigl(\C\oplus\C\oplus M_2(\C)\bigr),
\end{equation}
with the unital inductive-limit interpretation. Selecting the $M_2$
summand at each stage gives a unital surjection onto $\CAR$.
In particular $C^*(G_0)$ is non-type-I, and it has infinite-dimensional
irreducible representations.

\begin{corollary}\label{cor:Thomas}
For $G_0=\bigoplus_{\N}\Sym(3)$,
\[
 \defdual {G_0}\sdef\defdual {F_\infty},\qquad
 \defdual {G_0}^{\,\infty}\sdef\defdual {F_\infty}^{\,\infty}.
\]
Equivalently, in Thomas's notation for infinite-dimensional irreducibles,
\[
 \mathord{\approx}_{G_0}<_B\mathord{\approx}_{F_\infty}.
\]
\end{corollary}

\begin{proof}
The assertion about the full unitary dual follows from
Theorem~\ref{thm:group-separation}.
The forward injection in that theorem preserves dimension, so restricts
to the infinite-dimensional parts. The measure $\eta$ is already
supported on the infinite-dimensional part of the unitary dual of the
free group,
and its essential-constancy property rules out a reverse injection
even into the full unitary dual of $G_0$.
\end{proof}

This proves Theorem~\ref{intro:B}. In particular the comparison is stronger than
nonexistence of a Borel-definable bijection: it excludes an injection in
the indicated direction and supplies a fixed measure witnessing
essential constancy of all Borel-definable maps in that direction.

\section{The Cuntz and CAR spectra}\label{sec:groupoids}

We now prove the comparison needed for the positive nuclear result.
Write $E_D=E_D^{\mathrm{irr}}$ in the operator-algebraic comparisons
below. Let
$\mathcal O_2$ be generated by isometries $s_0,s_1$ with orthogonal
ranges summing to the identity.

\begin{theorem}\label{thm:cuntz-car}
There is a Borel reduction $E_{\mathcal O_2}\le_B E_{\CAR}$.
\end{theorem}

The proof transports representations on the aperiodic part of the
Cuntz groupoid through a fixed Borel groupoid isomorphism. We give the
operator construction explicitly, including its Borel dependence and
its effect on all bounded intertwiners. The remaining periodic part
has a smooth classification and is inserted separately.

\subsection{Two hyperfinite Borel groupoids}

Put $X=2^{\N}$ and let $T:X\to X$ be the left shift. The Cuntz
groupoid and its orbit relation are
\begin{align*}
 \mathcal G_2&=\{(x,m-n,y):m,n\ge0,\ T^m x=T^n y\},\\
 R_t&=\{(x,y):\exists m,n\ge0\ (T^m x=T^n y)\}.
\end{align*}
The arrow $(x,k,y)$ has source $y$ and range $x$, and lags add under
composition. For finite binary words $\alpha,\beta$, including the
empty word, the compact open bisections
\[
 Z(\alpha,\beta)
 =\{(\alpha z,|\alpha|-|\beta|,\beta z):z\in X\}
\]
form a countable basis. A bisection is a set of arrows on which source
and range are both injective. In Renault's model,
$C^*(\mathcal G_2)=\mathcal O_2$, the characteristic function of
$Z(\alpha,\beta)$ is $s_\alpha s_\beta^*$; see
\cite[Chapter~III, Definition~2.1, Proposition~2.3 and
Remark~2.8(ii), pp.~140--145]{Renault1980}. The diagonal is $C(X)$:
the cylinder projections $s_\alpha s_\alpha^*$ correspond to
$1_{[\alpha]}$, whose linear span is uniformly dense in $C(X)$.

Let $P$ be the countable set of eventually periodic sequences and set
$X_a=X\setminus P$. These are invariant Borel sets. On $X_a$ the lag
between two related points is unique, since two different lags would
make their common tail periodic. Thus
\[
 \mathcal G_2|X_a\longrightarrow R_t|X_a,
 \qquad (x,k,y)\longmapsto(x,y)
\]
is a Borel groupoid isomorphism. For each $k\in\Z$, the set
\[
 A_k=\{(x,y)\in R_t|X_a:\exists m,n\ge0\
            (m-n=k\ \&\ T^mx=T^ny)\}
\]
is Borel. The sets $A_k$ partition $R_t|X_a$ by uniqueness of the lag,
so the integer-valued function taking value $k$ on $A_k$ is Borel.
Hence the inverse $(x,y)\mapsto(x,k(x,y),y)$ is Borel.

The stabilized CAR algebra has the AF groupoid model on
$Y=2^{\N}\times\N$ with relation
\[
 ((x,i),(y,j))\in R_s\quad\Longleftrightarrow\quad xE_0y,
\]
where $E_0$ is eventual equality. Its basic matrix bisections are
\[
 B^{(n)}_{\alpha\beta;ij}
 =\{((\alpha z,i),(\beta z,j)):z\in2^{\N}\},
 \qquad |\alpha|=|\beta|=n.
\]
Their characteristic functions identify with
$e^{(n)}_{\alpha\beta}\otimes e_{ij}$ in $\CAR\otimes\KH$.
Here $e^{(n)}_{\alpha\beta}$ are the matrix units of $M_{2^n}$,
and the connecting maps satisfy
\begin{equation}\label{eq:af-inclusion}
 e^{(n)}_{\alpha\beta}\otimes e_{ij}
 =\sum_{b=0}^1 e^{(n+1)}_{\alpha b,\beta b}\otimes e_{ij}.
\end{equation}
These relations also give the AF groupoid identification directly.

Both groupoids used here are amenable: $\mathcal G_2$ is the groupoid of
the local homeomorphism $T$ and is amenable by
\cite[Lemma~3.5]{SimsWilliams2016}, while $R_s$ is an AF groupoid (the
amenability of AF equivalence relations is recalled in
\cite[Remark~3.1(i)]{Renault2003}). Hence their full and reduced groupoid
$C^*$-algebras coincide. Throughout this section
$C^*(G)$ denotes this common completion for the groupoids in the
application.

\begin{lemma}\label{lem:hyperfinite-isomorphism}
There is a Borel bijection $h:X_a\to Y$ such that
$xR_ty$ if and only if $h(x)R_sh(y)$.
\end{lemma}
\begin{proof}
 The tail relation generated by a countable-to-one Borel map is
 hyperfinite. More precisely, apply
 \cite[Corollary~8.2]{DJK1994} to $T$.
Hyperfiniteness passes to $X_a$. For $R_s$, take the finite
subrelations that allow changes in the first $n$ binary coordinates
and in indices $0,\ldots,n$, leaving points with larger indices
as singleton classes. Their increasing union is $R_s$.

 Both relations are aperiodic and nonsmooth. Every $R_s$-class is
 infinite because of the $\N$-coordinate. Every $R_t|X_a$-class is
 infinite because the points $0^nx$ are distinct for $x\in X_a$;
 equality for two different $n$ would make $x$ eventually periodic.
 Nonsmoothness of $R_s$ follows by restricting to
 $2^{\N}\times\{0\}$. For $R_t|X_a$, the subrelation $E_0|X_a$ is
 nonsmooth: if it were smooth, adjoining the countable invariant set
 $P$ would make $E_0$ smooth. A subrelation of a countable smooth
 Borel equivalence relation is smooth, so $R_t|X_a$ cannot be smooth.

 Both relations are compressible in the precise Borel sense. The map
 $c_t(x)=0x$ is a Borel bijection from $X_a$ onto $0X_a$, its graph
 lies in $R_t|X_a$, and the complement $1X_a$ is a complete section.
 Likewise $c_s(x,i)=(x,i+1)$ is a Borel bijection onto
 $Y\setminus(2^{\N}\times\{0\})$, and the omitted set is a complete
 section. If an invariant probability existed, invariance under
 $c_t$ or $c_s$ would make the omitted complete section null; its
 countable saturation is the whole space, a contradiction. Thus the
 set of ergodic invariant probability measures is empty for both
 relations.

 Theorem~2$'$ of \cite{DJK1994}, proved there as Theorem~9.1 on p.~215, states that aperiodic nonsmooth
 hyperfinite Borel equivalence relations are classified up to Borel
 isomorphism by the cardinality of their sets of ergodic invariant
 probability measures. All its hypotheses have now been verified,
 and the common cardinality here is zero. Hence the two relations are
 Borel isomorphic.
Fix such an $h$ for the rest of the proof.
\end{proof}

The map $h$ is only a Borel groupoid coordinate change; it is not asserted
to be a homeomorphism for the two ample topologies. This is sufficient
below because $h$ is fixed once and for all, and every transported set is
handled by the Borel bisection calculus rather than by a $C^*$-algebra
homomorphism induced by $h$. Only its existence is used; no canonical,
continuous, or effective choice of $h$ is required.

\subsection{Borel spectral calculus for bisections}

For a representation $\pi$ of $\mathcal O_2$, let $Q_\pi$ be the
projection-valued measure obtained by restricting to $C(X)$.
For the cylinder $[\alpha]$,
\[
 Q_\pi([\alpha])=\pi(s_\alpha s_\alpha^*).
\]
We use nondegenerate representations throughout. For a representation
of $\CAR\otimes\KH$, the diagonal $C_0(Y)$ similarly has a
projection-valued measure, also denoted $Q_\pi$.

\begin{lemma}[Borel bisection calculus]\label{lem:borel-bisections}
Let $G$ be a second-countable Hausdorff ample groupoid with a countable
basis of compact open bisections, and let $\pi$ be a nondegenerate
representation of the full groupoid algebra $C^*(G)$. The statement also
applies to representations of $C_r^*(G)$ by composition with the canonical
quotient map. For each fixed Borel set
$D\subseteq G^{(0)}$, the map $\pi\mapsto Q_\pi(D)$ is Borel in the
strong operator topology. For each fixed Borel bisection $V\subseteq G$,
there is a partial isometry $W_\pi(V)$, Borel in $\pi$, which agrees with
$\pi(1_V)$ when $V$ is compact open and satisfies
\begin{align}
 W_\pi(V)^*&=W_\pi(V^{-1}),&
 W_\pi(V)W_\pi(U)&=W_\pi(VU),\label{eq:bisection-products}\\
 W_\pi(D)&=Q_\pi(D),&
 W_\pi(V)^*W_\pi(V)&=Q_\pi(s(V)),\qquad
 W_\pi(V)W_\pi(V)^*=Q_\pi(r(V)).\notag
\end{align}
If $V=\bigsqcup_nV_n$ is a Borel bisection partitioned into Borel
sub-bisections, then
\begin{equation}\label{eq:bisection-additivity}
 W_\pi(V)=\operatorname{SOT}\!\sum_n W_\pi(V_n).
\end{equation}
Every bounded intertwiner of representations intertwines all these
projections and partial isometries. The conclusions remain valid after
restriction to a fixed invariant Borel subset of $G^{(0)}$.
\end{lemma}

\begin{proof}
On compact open unit sets, $Q_\pi(D)=\pi(1_D)$ is Borel in $\pi$.
The class of Borel unit sets with this property is closed under relative
complements and disjoint countable unions, using strong convergence of
orthogonal sums, and hence contains every Borel set. The same monotone
class argument shows that an intertwiner of the diagonal representations
intertwines their entire spectral measures.

Enumerate a compact-open bisection basis as $(C_j)$ and put $v_j=1_{C_j}$.
For a Borel bisection $V$, define a disjoint chart partition
\[
 V_j=V\cap\Bigl(C_j\setminus\bigcup_{k<j}C_k\Bigr),
 \qquad D_j=s(V_j).
\]
The sets $D_j$ are Borel by Lusin--Souslin. Set
\begin{equation}\label{eq:borel-bisection-sum}
 W_\pi(V)=\operatorname{SOT}\!\sum_j\pi(v_j)Q_\pi(D_j).
\end{equation}
The summands have pairwise orthogonal initial projections and pairwise
orthogonal final projections. Thus their finite sums and their adjoints
are contractions and converge strongly; in particular the sum converges
strong-* and is Borel in $\pi$.

For a compact open bisection $C$, covariance first on compact open unit
sets and then on all Borel unit sets gives
\begin{align}
 \pi(1_C)Q_\pi(E)\pi(1_C)^*
   &=Q_\pi(CEC^{-1}) &&(E\subseteq s(C)),\label{eq:covariance-left}\\
 Q_\pi(D)\pi(1_C)
   &=\pi(1_C)Q_\pi(C^{-1}DC) &&(D\subseteq r(C)).
   \label{eq:covariance-right}
\end{align}
If $C$ and $C'$ are two charts and $R\subseteq C\cap C'$ is Borel,
then
\begin{equation}\label{eq:chart-overlap}
 \pi(1_C)Q_\pi(s(R))=\pi(1_{C'})Q_\pi(s(R)).
\end{equation}
Indeed, $C\cap C'$ is itself a compact open bisection, since $G$ is
Hausdorff. Both $1_C1_{s(C\cap C')}$ and $1_{C'}1_{s(C\cap C')}$ equal
$1_{C\cap C'}$ in the convolution algebra. Applying $\pi$ and then
multiplying by $Q_\pi(s(R))$ proves the identity.

We now verify independence of the chart partition. Suppose
$V=\bigsqcup_jV_j=\bigsqcup_kV'_k$, with $V_j\subseteq C_j$ and
$V'_k\subseteq C'_k$. The sets $R_{jk}=V_j\cap V'_k$ form a common
refinement. Since
$s(V_j)=\bigsqcup_k s(R_{jk})$, strong countable additivity of $Q_\pi$
and \eqref{eq:chart-overlap} give
\[
 \pi(1_{C_j})Q_\pi(s(V_j))
 =\operatorname{SOT}\!\sum_k
   \pi(1_{C'_k})Q_\pi(s(R_{jk})).
\]
Summing over $j$ shows that the two definitions of $W_\pi(V)$ coincide.
The same common-refinement argument proves
\eqref{eq:bisection-additivity}; because the $V_n$ are disjoint pieces
of a bisection, their source projections and their range projections are
pairwise orthogonal.

It remains to write out the product. Choose chart partitions
$V=\bigsqcup_jV_j$ and $U=\bigsqcup_kU_k$ with
$V_j\subseteq C_j$, $U_k\subseteq C'_k$, and put
$D_j=s(V_j)$ and $E_k=s(U_k)$. The source of the composable part of
$V_jU_k$ is
\[
 E_{jk}=E_k\cap (C'_k)^{-1}(D_j\cap r(C'_k))C'_k.
\]
Using \eqref{eq:covariance-right} and the compact-open convolution
relations, we obtain
\begin{align*}
 W_\pi(V_j)W_\pi(U_k)
 &=\pi(1_{C_j})Q_\pi(D_j)\pi(1_{C'_k})Q_\pi(E_k)\\
 &=\pi(1_{C_j})\pi(1_{C'_k})Q_\pi(E_{jk})\\
 &=\pi(1_{C_jC'_k})Q_\pi(E_{jk})
  =W_\pi(V_jU_k).
\end{align*}
The nonempty pieces $V_jU_k$ are disjoint and partition $VU$: a product
$vu$ determines $u$ by injectivity of the source map on $U$, and then
determines $v$. Since $VU$ is a bisection, their source and range sets
are pairwise disjoint. The double sums are therefore bounded orthogonal
strong sums, and summing the displayed equality over $(j,k)$ proves
$W_\pi(V)W_\pi(U)=W_\pi(VU)$. Taking adjoints in the strong-* sums proves
the inverse identity. The unit and initial- and final-projection formulas
follow by applying the product identity to $V^{-1}V$ and $VV^{-1}$.

Finally, an intertwiner commutes across each spectral projection and each
represented chart operator, hence across their products and bounded
strong limits. Restricting to an invariant Borel unit set merely replaces
each bisection by its restriction, so the same proof applies.
\end{proof}

For an invariant Borel set $U\subseteq G^{(0)}$, say that a
representation $\pi$ is \emph{supported on $U$} when $Q_\pi(U)=I$.
This is a Borel condition on each representation stratum by
Lemma~\ref{lem:borel-bisections}.

\subsection{Transport of the aperiodic representations}

Compose the unique-lag isomorphism
$\mathcal G_2|X_a\to R_t|X_a$ with the isomorphism induced by $h$ to
obtain a Borel groupoid isomorphism
\[
 \phi:\mathcal G_2|X_a\longrightarrow R_s.
\]
We construct the required transport directly from the AF matrix units. If
$\widetilde B^{(n)}_{\alpha\beta;ij}$ denotes the inverse image under
$\phi$ of $B^{(n)}_{\alpha\beta;ij}$, set
\[
 T^{(n)}_{\alpha\beta;ij}(\pi)
 =W_\pi(\widetilde B^{(n)}_{\alpha\beta;ij}).
\]
The inverse images are Borel bisections. The adjoint, product, and strong
additivity formulas of Lemma~\ref{lem:borel-bisections} give
\begin{align*}
 (T^{(n)}_{\alpha\beta;ij})^*
   &=T^{(n)}_{\beta\alpha;ji},\\
 T^{(n)}_{\alpha\beta;ij}T^{(n)}_{\gamma\delta;kl}
   &=\delta_{\beta\gamma}\delta_{jk}
     T^{(n)}_{\alpha\delta;il},\\
 T^{(n)}_{\alpha\beta;ij}
   &=\sum_{b=0}^1T^{(n+1)}_{\alpha b,\beta b;ij}.
\end{align*}
The last equality follows because the two bisections on the right are a
disjoint partition of the bisection on the left. For fixed $n,m$, with
$i,j\le m$, the first two identities therefore define a $*$-representation
of $M_{2^n}\otimes M_{m+1}$; the third identity gives compatibility in
$n$, and the corner inclusions give compatibility in $m$. Since the union
of these finite-dimensional corners is norm dense in $\CAR\otimes\KH$,
there is a unique representation $F(\pi)$ satisfying
\begin{equation}\label{eq:transport-af}
 F(\pi)(e^{(n)}_{\alpha\beta}\otimes e_{ij})
  =W_\pi(\widetilde B^{(n)}_{\alpha\beta;ij}).
\end{equation}
For $\pi$ supported on $X_a$, it is nondegenerate, because
\[
 \sum_{i\ge0}F(\pi)(1\otimes e_{ii})
 =\sum_{i\ge0}Q_\pi(h^{-1}(X\times\{i\}))=Q_\pi(X_a)=I
\]
in the strong topology. Each matrix coefficient in
\eqref{eq:transport-af} is Borel in $\pi$ by
Lemma~\ref{lem:borel-bisections}. Evaluation on the rational AF
$*$-algebra is consequently Borel, and fixed norm approximations give a
Borel representation code on all of $\CAR\otimes\KH$.

On a basic diagonal cylinder, the construction gives
$Q_{F(\pi)}([\alpha]\times\{i\})
=Q_\pi(h^{-1}([\alpha]\times\{i\}))$. The cylinders form a countable
algebra generating the Borel subsets of $Y$; uniqueness of projection-valued
measures, equivalently the usual monotone-class argument, therefore yields
\begin{equation}\label{eq:transport-spectral}
 Q_{F(\pi)}(D)=Q_\pi(h^{-1}(D))
 \qquad(D\subseteq Y\text{ Borel}).
\end{equation}

\begin{lemma}\label{lem:transport-intertwiners}
For representations $\pi,\sigma$ supported on $X_a$,
\[
 \operatorname{Hom}_{\mathcal O_2}(\pi,\sigma)
 =\operatorname{Hom}_{\CAR\otimes\KH}(F(\pi),F(\sigma)).
\]
In particular $F$ preserves irreducibility and preserves and reflects
unitary equivalence.
\end{lemma}
\begin{proof}
If $S$ intertwines $\pi$ and $\sigma$, Lemma~\ref{lem:borel-bisections}
shows that it intertwines every pair of bisection operators occurring in
\eqref{eq:transport-af}. It therefore intertwines $F(\pi)$ and
$F(\sigma)$ on the dense AF $*$-subalgebra, and hence on the whole target.

For the converse we first recover every source bisection operator. If
$V\subseteq\mathcal G_2|X_a$ is a Borel bisection, then
\begin{equation}\label{eq:recover-af-bisection}
 W_{F(\pi)}(\phi(V))=W_\pi(V).
\end{equation}
To see this, partition $\phi(V)$ by the first basic matrix bisection
$B^{(n)}_{\alpha\beta;ij}$ containing each arrow. The defining
Borel-bisection sum writes the left side as a strong sum of operators
from \eqref{eq:transport-af}, restricted by the source projections of
the pieces. Formula \eqref{eq:transport-spectral} transports those source
projections back under $\phi$. The product and chart-independence parts of
Lemma~\ref{lem:borel-bisections} identify the summands with the corresponding
pieces of $W_\pi(V)$, and strong additivity proves
\eqref{eq:recover-af-bisection}.

Now let $C\subseteq\mathcal G_2$ be compact open. Invariance of $X_a$ and
support of $\pi$ on $X_a$ give
\[
 W_\pi(C\cap\mathcal G_2|X_a)
 =\pi(1_C)Q_\pi(s(C)\cap X_a)=\pi(1_C).
\]
If $S$ intertwines $F(\pi)$ and $F(\sigma)$, then the target-side part of
Lemma~\ref{lem:borel-bisections} makes it intertwine all Borel-bisection
operators of these representations. Apply this to
$\phi(C\cap\mathcal G_2|X_a)$ and use
\eqref{eq:recover-af-bisection}: $S$ intertwines $\pi(1_C)$ and
$\sigma(1_C)$. Compact-open bisections generate a dense $*$-subalgebra of
$\mathcal O_2$, so $S$ intertwines $\pi$ and $\sigma$. This proves the
reverse inclusion of Hom spaces. The assertions about irreducibility and
unitary equivalence follow at once.
\end{proof}

\begin{lemma}[Borel destabilization]\label{lem:borel-destabilization}
Compression by $1\otimes e_{00}$ gives a Borel reduction
\[
 E_{\CAR\otimes\KH}\le_B E_{\CAR}.
\]
\end{lemma}
\begin{proof}
Let $\tau$ be a nonzero irreducible representation of
$\CAR\otimes\KH$ on a fixed Hilbert space and put
$p_\tau=\tau(1\otimes e_{00})$. This projection is nonzero, since the
matrix units make the ideal generated by $1\otimes e_{00}$ equal to the
whole algebra. Compression to $p_\tau\Hil$ gives a representation
$c_\tau$ of $\CAR$. It is irreducible: an operator in its commutant
amplifies, by the formula below, to an operator in $\tau(\CAR\otimes
\KH)'$.

The operators $\tau(1\otimes e_{i0})$ implement the unitary
\begin{equation}\label{eq:destabilization-unitary}
 p_\tau\Hil\otimes\ell^2(\N)\longrightarrow\Hil,
 \qquad \xi\otimes e_i\longmapsto\tau(1\otimes e_{i0})\xi.
\end{equation}
The range projections are orthogonal and sum strongly to $I$, by
nondegeneracy of $\tau$. Under \eqref{eq:destabilization-unitary},
$\tau$ is $c_\tau\otimes\id_{\KH}$. Its commutant is therefore
$c_\tau(\CAR)'\otimes I$; irreducibility of $\tau$ implies
$c_\tau(\CAR)'=\C I$, so $c_\tau$ is irreducible. A unitary intertwiner
between $c_\tau$ and $c_\sigma$ amplifies to one between $\tau$ and
$\sigma$. Conversely, conjugate an intertwiner between $\tau$ and
$\sigma$ by the unitaries in \eqref{eq:destabilization-unitary} and take
a matrix slice in the $\ell^2$ coordinate. Some slice is nonzero and
intertwines $c_\tau$ with $c_\sigma$; Schur's lemma makes it a scalar
multiple of a unitary. Thus compression preserves and reflects
equivalence as well as irreducibility.

It remains to put the compressed representations on the fixed code
space. The map $\tau\mapsto p_\tau$ is strong-operator Borel. Apply
 Gram--Schmidt to $(p_\tau e_n)_n$, always choosing the first nonzero
 residual vector. The resulting basis, and hence the isometry
$J_\tau:\ell^2(\N)\to p_\tau\Hil$, are Borel in $\tau$; the vectors form
an orthonormal basis of $p_\tau\Hil$ because $(p_\tau e_n)_n$ has dense
linear span in the range of $p_\tau$. The compressed
space is infinite-dimensional, because the unital simple
infinite-dimensional algebra $\CAR$ has no nonzero finite-dimensional
representation. Define
\[
 \widehat c_\tau(a)=J_\tau^*\tau(a\otimes e_{00})J_\tau.
\]
Its matrix coefficients are Borel in $\tau$, first on a countable dense
$*$-subalgebra and then on all of $\CAR$ by norm approximation.
Formula \eqref{eq:destabilization-unitary} shows that
$\tau\mapsto\widehat c_\tau$ is the required Borel reduction.
\end{proof}

We will use the following elementary tensor-product fact both to separate
the periodic and aperiodic sectors and in the nuclear upper bound.

\begin{lemma}\label{lem:nuclear-tensor}
Let $A,C$ be separable unital $C^*$-algebras and let $\eta$ be a
fixed irreducible representation of $C$. Then
$\pi\mapsto\pi\otimes\eta$ gives a Borel reduction
$E_A\le_B E_{A\otimes_{\min}C}$.
\end{lemma}
\begin{proof}
Let $H_\pi$ and $K$ be the representation spaces of $\pi$ and $\eta$.
Since both representations are irreducible, the bicommutant theorem gives
$\pi(A)''=\BH(H_\pi)$ and $\eta(C)''=\BH(K)$. Hence
\[
 (\pi(A)\otimes I)'=I\otimes\BH(K),\qquad
 (I\otimes\eta(C))'=\BH(H_\pi)\otimes I,
\]
and the intersection of these two commutants is $\C I$. The representation
$\pi\otimes\eta$ is therefore irreducible.

If $V:H_\pi\to H_\sigma$ intertwines $\pi$ and $\sigma$, then
$V\otimes I_K$ intertwines their tensor products, so input equivalence
implies output equivalence. Conversely, let
$U:H_\pi\otimes K\to H_\sigma\otimes K$ be a unitary intertwining
$\pi\otimes\eta$ and $\sigma\otimes\eta$. For basis vectors $e_k,e_l$ of
$K$, define
\[
 U_{kl}=(I\otimes\langle e_k,\cdot\rangle)
        U(I\otimes e_l):H_\pi\longrightarrow H_\sigma.
\]
Intertwining $a\otimes1$ gives $U_{kl}\pi(a)=\sigma(a)U_{kl}$ for every
$a\in A$. At least one $U_{kl}$ is nonzero, because otherwise every
matrix coefficient of $U$ would vanish. Schur's lemma makes that slice a
nonzero scalar multiple of a unitary, proving $\pi\simeq\sigma$.

Finally, after fixed identifications of the tensor-product Hilbert spaces
on each dimension stratum, the matrix coefficients of
$(\pi\otimes\eta)(a\otimes c)=\pi(a)\otimes\eta(c)$ are Borel in $\pi$.
The assertion for every element of $A\otimes_{\min}C$ follows by fixed
norm approximation from the algebraic tensor product. Thus the assignment
is Borel and is a reduction.
\end{proof}

\begin{lemma}[Inequivalent tensor tags]\label{lem:tensor-tags}
Let $A,C$ be separable unital $C^*$-algebras, let $\pi,\sigma$ be
representations of $A$, and let $\eta_0,\eta_1$ be inequivalent
irreducible representations of $C$. Then
\[
 \pi\otimes\eta_0\not\simeq\sigma\otimes\eta_1
\]
as representations of $A\otimes_{\min}C$.
\end{lemma}
\begin{proof}
Suppose that a unitary $U:H_\pi\otimes K_0\to H_\sigma\otimes K_1$
intertwines the tensor-product representations. For
$\xi\in H_\pi$ and $\zeta\in H_\sigma$, define the slice
\[
 T_{\zeta,\xi}:K_0\to K_1,\qquad
 T_{\zeta,\xi}\eta
 =(\langle\zeta,\cdot\rangle\otimes I)
       U(\xi\otimes\eta).
\]
Intertwining the elements $1\otimes c$ gives
$T_{\zeta,\xi}\eta_0(c)=\eta_1(c)T_{\zeta,\xi}$. At least one slice is
nonzero, since otherwise every matrix coefficient of $U$ would vanish.
Schur's lemma would then give $\eta_0\simeq\eta_1$, a contradiction.
\end{proof}

\subsection{Periodic representations and disjoint coding}

If $Z$ is invariant, covariance shows that $Q_\pi(Z)$ commutes
with the two Cuntz generators. For irreducible $\pi$ it is therefore
$0$ or $I$. Thus every irreducible representation is supported
either on $X_a$ or on $P$, and this dichotomy is Borel.

\begin{lemma}\label{lem:periodic-smooth}
The restriction of $E_{\mathcal O_2}$ to the representations
supported on $P$ is smooth.
\end{lemma}
\begin{proof}
Enumerate the $R_t$-orbits in $P$ as $(O_j)$. Since $P$ is countable,
\[
 Q_\pi(P)=\operatorname{SOT}\!\sum_{y\in P}Q_\pi(\{y\}).
\]
For an invariant Borel set $O\subseteq P$, covariance for Borel
bisections gives
\[
 W_\pi(V)Q_\pi(O)=Q_\pi(O)W_\pi(V).
\]
Thus
$Q_\pi(O)$ commutes with the compact-open bisection operators and is
central in $\pi(\mathcal O_2)''$. If $\pi$ is irreducible and supported
on $P$, exactly one orbit, say $O_j$, has $Q_\pi(O_j)=I$.

Choose in each orbit a purely periodic point
$x_j=w_jw_j\cdots$, where $w_j$ has minimal positive length, and choose
once and for all an arrow $g_y$ from $x_j$ to each $y\in O_j$, with
$g_{x_j}$ the unit arrow. The isotropy group at $x_j$ is infinite cyclic,
generated by $u_j=(x_j,|w_j|,x_j)$.

Put $H_y=Q_\pi(\{y\})\Hil$ and
$p_\pi=Q_\pi(\{x_j\})$. Countable additivity gives
$\Hil=\bigoplus_{y\in O_j}H_y$. At least one fiber is nonzero; the
singleton-arrow partial isometries identify every $H_y$ with
$H_{x_j}$, so $p_\pi\ne0$. Let
$V_y=W_\pi(\{g_y\}):H_{x_j}\to H_y$. Let $U_\pi$ be the restriction
of $W_\pi(\{u_j\})$, equivalently of $\pi(s_{w_j})$, to
$H_{x_j}$. It is a unitary.

Every arrow from $y$ to $z$ is uniquely of the form
$g_zu_j^kg_y^{-1}$, and the bisection product formula represents it by
\[
 V_zU_\pi^kV_y^*:H_y\longrightarrow H_z.
\]
If $R\in\{U_\pi\}'$, the orthogonal strong sum
\[
 \widehat R=\operatorname{SOT}\!\sum_{y\in O_j}V_yRV_y^*
\]
has norm $\norm R$ and commutes with every singleton-arrow operator.
For a compact open bisection $C$, its restriction to $O_j$ is a
countable disjoint union of singleton arrows, and support on $O_j$ gives
\[
 \pi(1_C)=W_\pi(C\cap\mathcal G_2|O_j)
 =\operatorname{SOT}\!\sum_{\gamma\in C\cap\mathcal G_2|O_j}
      W_\pi(\{\gamma\}).
\]
Hence $\widehat R$ commutes with every compact-open generator and with
$\pi(\mathcal O_2)$. Irreducibility forces
$\{U_\pi\}'=\C I$. A normal operator on a Hilbert space of dimension
greater than one never has scalar commutant, so
$\dim H_{x_j}=1$. Write $U_\pi=\lambda_\pi I$ with
$\lambda_\pi\in\T$.

The pair $(j,\lambda_\pi)$ is a complete invariant. Necessity follows
because an intertwiner preserves every diagonal atom and the isotropy
loop. Conversely, suppose $\pi,\sigma$ have the same pair. Choose a
unitary $L:H_{x_j}^\pi\to H_{x_j}^\sigma$; the equality of the loop
scalars makes $LU_\pi=U_\sigma L$. Then
\[
 S=\operatorname{SOT}\!\sum_{y\in O_j}
       V_y^\sigma L(V_y^\pi)^*
\]
is a unitary. The displayed arrow formula shows that it intertwines every
singleton arrow and hence every compact-open generator.

Finally the invariant is Borel. The integer
$j(\pi)=\min\{j:Q_\pi(O_j)=I\}$ is Borel, since every spectral projection
$Q_\pi(O_j)$ is Borel by Lemma~\ref{lem:borel-bisections}. On the Borel
piece where $j(\pi)=j$, normalize the first nonzero column of
$p_\pi$ to obtain a Borel unit vector $v_\pi$. Our inner products are
linear in the second variable, so
\[
 \lambda_\pi=\langle v_\pi,\pi(s_{w_j})v_\pi\rangle.
\]
This is a Borel reduction to equality on $\N\times\T$.
\end{proof}

We record explicitly how to place this smooth piece into the CAR
spectrum. For $x\in2^{\N}$, the prefix-replacement matrix units act on
$\ell^2([x]_{E_0})$, giving a representation $\tau_x$ of $\CAR$.
Its diagonal spectral projections include every point projection
as a strong limit of cylinder projections. An operator in the
commutant is consequently diagonal, and the matrix units, acting
transitively on the tail class, force all its diagonal entries to
be equal. Thus $\tau_x$ is irreducible. Its atomic spectral support
is exactly $[x]_{E_0}$, whence
\begin{equation}\label{eq:tail-representations}
 \tau_x\simeq\tau_y\quad\Longleftrightarrow\quad xE_0y.
\end{equation}
For Borel coding, enumerate the countable group of finite binary
flips and use $g\mapsto x\mathbin\triangle g$ as the fixed
enumeration of the tail class. In these bases all matrix
coefficients of the prefix-replacement operators are Borel in $x$.

Choose a Borel injection of $\N\times\T$ into $2^{\N}$. Repeat
each coordinate infinitely often, on a fixed partition of $\N$
into infinite sets. Distinct codes are then not eventually equal.
Composing with $x\mapsto\tau_x$ and
Lemma~\ref{lem:periodic-smooth} gives a Borel reduction of the
periodic sector to the CAR spectrum.

To keep the two sectors disjoint, fix inequivalent irreducible CAR
representations $\rho_0=\tau_{0^\infty}$ and
$\rho_1=\tau_{1^\infty}$. Use the fixed isomorphism
$\CAR\otimes\CAR\cong\CAR$ obtained by interleaving tensor
factors. If $r_a(\pi)$ is the compressed representation from the
aperiodic construction and $r_p(\pi)$ the periodic code, send $\pi$
to
\[
 \begin{cases}
 r_a(\pi)\otimes\rho_0,&Q_\pi(X_a)=I,\\
 r_p(\pi)\otimes\rho_1,&Q_\pi(P)=I.
 \end{cases}
 \]
 Tensor products of irreducible representations are irreducible.
 Lemma~\ref{lem:nuclear-tensor} shows that each branch reflects
 equivalence. Lemma~\ref{lem:tensor-tags} shows that the two branches
 have disjoint equivalence-saturations, because
 $\rho_0\not\simeq\rho_1$. The
construction is Borel on the fixed Hilbert spaces. This proves
Theorem~\ref{thm:cuntz-car}.

\section{Fully faithful Fock induction}\label{sec:fock}

One construction serves both the nuclear upper bound and the universal
simple endpoint. We record its effect on all intertwiners, then identify
its quotient image explicitly. The coefficient representation need not
be faithful.

Let $C$ be a nonzero separable unital $C^*$-algebra. Fix a faithful
unital representation $\rho:C\to\BH(H_0)$ on a separable
infinite-dimensional Hilbert space such that
$\rho(C)\cap\KH(H_0)=\{0\}$. A countably infinite amplification
of any faithful separable unital representation has this property.
Let $\mathcal E=H_0\otimes C$ be the right Hilbert $C$-module with
\[
 \langle\xi\otimes a,\eta\otimes b\rangle_C
 =\langle\xi,\eta\rangle a^*b,
 \qquad
 \varphi(a)(\xi\otimes b)=\rho(a)\xi\otimes b.
\]
Put $\mathcal P_C=\mathcal T_{\mathcal E}$. Kumjian's construction gives
\[
 \mathcal T_{\mathcal E}=\mathcal O_{\mathcal E},\qquad
 \mathcal P_C\text{ separable, unital, simple, and purely infinite}
\]
by \cite[Proposition~2.1 and Theorem~2.8]{Kumjian2004}. If $C$ is
nuclear, then $\mathcal P_C$ is nuclear by
\cite[Theorem~7.2]{Katsura2004}, since the Toeplitz algebra of a
correspondence over a nuclear coefficient algebra is nuclear. This
statement has no UCT hypothesis.

Let $\iota:C\to\mathcal P_C$ be the coefficient inclusion. Fix an
orthonormal basis $(e_j)_{j\ge1}$ of $H_0$ and put
$t_j=t_{e_j\otimes1}$. The algebra is generated by $\iota(C)$
and these creation operators. Write
\[
 \mathcal F(H_0)=\C\Omega\oplus\bigoplus_{n\ge1}H_0^{\otimes n}
\]
for Hilbert-space Fock space, and let $\ell_j$ denote left creation
by $e_j$.

\begin{lemma}[Fock induction]\label{lem:nuclear-fock}
For every unital representation $\pi:C\to\BH(H_\pi)$ there is
a representation $\Pi_\pi$ of $\mathcal P_C$ on
$\mathcal F(H_0)\otimes H_\pi$ with
\begin{equation}\label{eq:fock-hom}
 \operatorname{Hom}_{\mathcal P_C}(\Pi_\pi,\Pi_\sigma)
 =I_{\mathcal F(H_0)}\otimes\operatorname{Hom}_C(\pi,\sigma).
\end{equation}
The assignment is Borel after fixed Hilbert-space identifications,
and it preserves irreducibility and preserves and reflects unitary
equivalence, including between different dimension strata.
\end{lemma}
\begin{proof}
Let
$\mathcal F(\mathcal E)=C\oplus\mathcal E\oplus\mathcal E^{\otimes2}
\oplus\cdots$ be the Fock $C$-module. Since $\pi$ is unital, the
internal tensor product
$\mathcal F(\mathcal E)\otimes_\pi H_\pi$ is nondegenerate, and every
adjointable Fock-module operator acts on it by tensoring with the
identity. The coefficient and creation operators therefore satisfy
\[
 \Pi_\pi(\iota(a))\Pi_\pi(t_\xi)
   =\Pi_\pi(t_{\varphi(a)\xi}),\qquad
 \Pi_\pi(t_\xi)^*\Pi_\pi(t_\eta)
   =\Pi_\pi(\iota(\langle\xi,\eta\rangle_C)).
\]
Thus they give the representation obtained by tensoring the
Fock-module representation of $\mathcal T_{\mathcal E}=\mathcal P_C$
with $H_\pi$. No faithfulness of $\pi$ is needed for the internal tensor
product or these relations.

 The identifications
 $\mathcal E^{\otimes_C n}\cong H_0^{\otimes n}\otimes C$
from \cite[Remark~2.2, p.~281]{Kumjian2004}
give the asserted Hilbert space and the formulas
\begin{align}
 \Pi_\pi(\iota(a))
 &=\pi(a)\ \oplus\!
   \bigoplus_{n\ge1}
   \bigl(\rho(a)\otimes I_{H_0}^{\otimes(n-1)}\otimes I_{H_\pi}\bigr),
     \label{eq:fock-coefficients}\\
 \Pi_\pi(t_j)&=\ell_j\otimes I_{H_\pi}.
     \label{eq:fock-creations}
\end{align}
 The projection onto the vacuum subspace is recovered as
\begin{equation}\label{eq:fock-vacuum}
 P_\pi=I-\operatorname{SOT}\!\lim_{m\to\infty}
              \sum_{j=1}^m\Pi_\pi(t_jt_j^*).
 \end{equation}
Indeed the ranges of the left creation operators $\ell_j$ are
orthogonal, and their closed sum is exactly the orthogonal complement
of the vacuum in $\mathcal F(H_0)$. The partial sums in
\eqref{eq:fock-vacuum} are therefore increasing projections converging
strongly to the nonvacuum projection.

 If $S$ intertwines $\Pi_\pi$ and $\Pi_\sigma$, it intertwines
 every finite partial sum in \eqref{eq:fock-vacuum}. Boundedness of $S$
 lets the intertwining relation pass to the strong limit, so $S$
 intertwines the vacuum projections. Its restriction to the vacuum has the form
$S(\Omega\otimes\xi)=\Omega\otimes R\xi$. Intertwining all
products of creation operators gives
\[
 S(e_{j_1}\otimes\cdots\otimes e_{j_n}\otimes\xi)
 =e_{j_1}\otimes\cdots\otimes e_{j_n}\otimes R\xi.
\]
The displayed vectors and the vacuum span a dense subspace. Hence
$S=I\otimes R$, and the vacuum compression of the coefficient
relation gives $R\pi(a)=\sigma(a)R$. Conversely this condition
and \eqref{eq:fock-coefficients}--\eqref{eq:fock-creations} show
that $I\otimes R$ is an intertwiner. This proves
\eqref{eq:fock-hom}. The commutant assertion gives irreducibility,
and $I\otimes R$ is unitary exactly when $R$ is unitary.

 For every $d\in\N_{\ge1}\cup\{\infty\}$ choose once and for all a
 unitary
 \[
  J_d:\mathcal F(H_0)\otimes\Hil_d\longrightarrow\ell^2(\N),
 \]
 and code the output as
 $\widehat\Pi_\pi=J_d\Pi_\pi J_d^*$ for $\pi\in\Rep_d(C)$.
 If $\pi$ and $\sigma$ belong to possibly different strata, conjugating
 an intertwiner of $\widehat\Pi_\pi$ and $\widehat\Pi_\sigma$ by the
 corresponding $J_d$'s reduces it to the Hom computation above. Thus
 \eqref{eq:fock-hom} controls equivalence also across different input
 dimensions.

 The only
 variable block in \eqref{eq:fock-coefficients} is $\pi(a)$;
the creation operators are fixed. Evaluations of $*$-polynomials
in a countable set of generators are Borel. Fixed norm
approximations prove Borelness on every element of $\mathcal P_C$, and
there are only countably many input dimension strata.
\end{proof}

Although all irreducible $\Pi_\pi$ are faithful, their vacuum
compressions recover the original representations and their kernels.
The projection in \eqref{eq:fock-vacuum} lies in the strong closure;
it is not claimed to belong to $\mathcal P_C$. Thus the construction does
not discard the primitive-ideal information in the input spectrum.

It is useful to check the image directly, without appealing to a
general image theorem.

\begin{proposition}[The Fock sector]\label{prop:sector}
For $\sigma\in\Irr(\mathcal P_C)$ define
\[
 S_j=\sigma(t_j),\qquad
 P_\sigma=I-\SOT\!\lim_m\sum_{j=1}^m S_jS_j^*.
\]
The equivalence-saturation of the induced representations is exactly
\[
 \mathcal V_C=\{\sigma\in\Irr(\mathcal P_C):P_\sigma\ne0\}.
\]
This is an invariant Borel set. Vacuum compression, followed by Borel
basis selection, gives a Borel inverse on quotient classes. Thus
$Q_C\qemb Q_{\mathcal P_C}$, with a Borel lift in each direction on
the image.
\end{proposition}
\begin{proof}
The $S_j$ are isometries with orthogonal ranges. Put
$W=P_\sigma H_\sigma=\bigcap_j\ker S_j^*$.
For $a\in C$, the covariance relation gives the norm-convergent identity
\[
 \sigma(\iota(a))S_j
  =\sum_k\langle e_k,\rho(a)e_j\rangle S_k.
\]
It follows, applying the same observation to $a^*$, that
$\overline{\sum_jS_jH_\sigma}$ reduces the coefficient representation.
Thus $W$ reduces it, and
$\pi_W(a)=\sigma(\iota(a))|_W$ is a unital representation when $W\ne0$.

For a finite word $w=(j_1,\ldots,j_n)$ put
$S_w=S_{j_1}\cdots S_{j_n}$, and let the empty word give $I$.
The spaces $S_wW$ are pairwise orthogonal: cancellation of common
prefixes reduces the only nontrivial case to $W\perp S_jH_\sigma$.
Their closed sum $H_F$ reduces every $S_j$, every $S_j^*$, and the
coefficient representation. For the latter assertion, use the displayed
covariance formula on the first letter, and use invariance of $W$ for
the empty word. These operators generate $\sigma(\mathcal P_C)$.
Irreducibility and $W\ne0$ therefore imply $H_F=H_\sigma$.

The map
\[
 e_{j_1}\otimes\cdots\otimes e_{j_n}\otimes w
       \longmapsto S_{j_1}\cdots S_{j_n}w
\]
(with $\Omega\otimes w\mapsto w$) is consequently a unitary
$\mathcal F(H_0)\otimes W\to H_\sigma$. The covariance formula and
vacuum compression show that it intertwines $\Pi_{\pi_W}$ with
$\sigma$. Lemma~\ref{lem:nuclear-fock} implies that $\pi_W$
is irreducible. Conversely, every $\Pi_\pi$ has nonzero vacuum space,
so this proves the image assertion.

Finally, $\sigma\mapsto P_\sigma$ is Borel: it is the pointwise strong
limit of Borel projection-valued maps. Nonzeroness can be tested on a
fixed countable basis. Apply Gram--Schmidt to the sequence
$(P_\sigma v_n)_n$, choosing the first nonzero residual at each step.
The tests, vectors, and dimension strata are Borel. Compressing the
coefficient operators to the resulting basis gives a Borel code for
$\pi_W$. It is inverse to induction up to unitary equivalence.
\end{proof}

The vacuum projection is used only in the strong operator closure of a
representation. It is not asserted to be an element of $\mathcal P_C$.
Also, only the specified sector $\mathcal V_C$ is asserted to be
induced; arbitrary irreducible representations of $\mathcal P_C$ need
not have a nonzero vacuum projection.

\section{A common Borel-definable spectrum for nuclear algebras}\label{sec:nuclear-duals}

The comparison with CAR now combines the groupoid transport with Fock
induction and absorption. All assertions in this section concern Borel
lifts, and hence also give isomorphisms for the quotient measurable
structures.

\subsection{The nuclear upper bound}

\begin{theorem}\label{thm:nuclear-upper}
For every separable nuclear $C^*$-algebra $A$,
\[
 E_A\le_B E_{\mathcal O_2}\le_B E_{\CAR}.
\]
No simplicity, unitality, or UCT assumption is required.
\end{theorem}
\begin{proof}
 The zero algebra has empty spectrum. Otherwise put $A^\dagger=A$ when $A$
 is unital and $A^\dagger=\widetilde A$ when it is not. Extension
 $\pi^\dagger(a+\lambda1)=\pi(a)+\lambda I$ preserves all intertwiners
 and is Borel. For nonunital $A$, its image is the invariant Borel set
 of irreducible representations of $\widetilde A$ whose restriction to
 $A$ is nonzero; the omitted irreducible is the one-dimensional
 augmentation. Borelness follows by testing nonvanishing on a countable
 dense subset of $A$. Hence $E_A\le_B E_{A^\dagger}$. The algebra $A^\dagger$ is
nuclear. Lemma~\ref{lem:nuclear-fock} gives
$E_{A^\dagger}\le_B E_{\mathcal P_{A^\dagger}}$, where
$\mathcal P_{A^\dagger}$ is separable, unital, simple, and nuclear.

Kirchberg's $\mathcal O_2$-absorption theorem gives
\[
 \mathcal P_{A^\dagger}\otimes_{\min}\mathcal O_2\cong\mathcal O_2.
\]
For example, apply \cite[Corollary~6.14]{Gabe2020} to the unital
separable nuclear algebras $\mathcal P_{A^\dagger}$ and $\C$. Both are
unital, so the unital-or-stable hypothesis in the cited result is met.
They are simple, so their
ideal lattices are the same two-element lattice. The cited corollary
therefore identifies $\mathcal P_{A^\dagger}\otimes\mathcal O_2$ with
$\C\otimes\mathcal O_2$. This form of the theorem does not require the
 UCT. Fix one such isomorphism $\theta$ and, by taking the GNS
 representation of a pure state, one irreducible representation
 $\eta$ of $\mathcal O_2$. The map
\[
 \pi\longmapsto(\Pi_{\pi^\dagger}\otimes\eta)\circ\theta^{-1}
\]
is a Borel reduction into $E_{\mathcal O_2}$. Apply
Lemma~\ref{lem:nuclear-tensor} for
equivalence reflection and Theorem~\ref{thm:cuntz-car} for the
last reduction.
\end{proof}

Theorem~\ref{thm:nuclear-upper} gives a second proof of Theorem~\ref{intro:C}:
compose a homomorphism into a nuclear spectrum with the reduction to the
CAR spectrum and apply the CAR case of Proposition~\ref{prop:AF}. The direct
Haagerup--CSSWW proof remains logically independent of the groupoid
comparison, Kumjian's construction, and $\mathcal O_2$-absorption.

\subsection{A common Borel-definable spectrum}

\begin{theorem}\label{thm:nuclear-isomorphism}
For every separable nuclear non-type-I $C^*$-algebra $A$,
\[
 \defdual A\isoBD\defdual{\CAR}\isoBD\defdual{\mathcal O_2}.
\]
Each displayed relation is a classwise Borel isomorphism. In particular,
the spectra of any two separable nuclear non-type-I $C^*$-algebras admit
a bijection with Borel lifts in both directions.
\end{theorem}
\begin{proof}
The Glimm--Elliott lower bound
\cite[Proposition~3]{Farah2011} gives $E_{\CAR}\le_B E_A$
for every separable non-type-I $A$. Together with
Theorem~\ref{thm:nuclear-upper}, this yields
$E_A\sim_B E_{\CAR}$. Apply the definable Cantor--Bernstein
principle, Proposition~\ref{prop:dual-cantor-bernstein}.
The algebra $\mathcal O_2$ is itself separable, nuclear, and
non-type-I, so it has the same definable spectrum.
\end{proof}

This proves Theorem~\ref{intro:D}. More explicitly, the result supplies Borel
maps $H:\Irr(A)\to\Irr(\CAR)$ and
$K:\Irr(\CAR)\to\Irr(A)$ such that
\[
 K(H(\pi))\simeq\pi,\qquad H(K(\sigma))\simeq\sigma.
\]
Their induced functions on classes are inverse bijections; the
maps on representation codes need not themselves be inverse
bijections. Since a Borel lift induces a measurable quotient map,
these Borel-definable bijections also identify the Mackey quotient
measurable spaces. The converse implication is not being used and need
not hold: an isomorphism of the Mackey quotient measurable spaces alone
need not have Borel lifts to the representation-code spaces.

\begin{corollary}\label{cor:nuclear-comparability}
There are no two separable nuclear $C^*$-algebras with incomparable
definable spectra. A type-I spectrum is strictly below the CAR spectrum,
and the non-type-I nuclear spectra are pairwise connected by Borel-definable
bijections.
Among type-I algebras, the degrees are ordered by the cardinalities
of their spectra. In the nonzero simple nuclear case there are
exactly two possibilities: the one-point spectrum and the CAR spectrum.
\end{corollary}
\begin{proof}
Glimm's theorem makes a separable type-I spectrum smooth. Its presenting
orbit relation is strongly idealistic by Lemma~\ref{lem:orbit-idealistic}.
Lemma~\ref{lem:smooth-strong-selector} therefore identifies it classwise
Borel with a standard Borel space. Consequently its type under
classwise Borel isomorphism is determined by its
cardinality, which is finite, countably infinite, or continuum
(including the empty spectrum of the zero algebra). Such degrees
are linearly ordered. Theorem~\ref{thm:nuclear-upper} places them
below the CAR spectrum, and nonsmoothness of the latter makes the
comparison strict. Theorem~\ref{thm:nuclear-isomorphism} treats
the remaining nuclear algebras. Finally, if $A$ is nonzero, simple, and
type~I, a faithful irreducible representation $\pi$ contains
$\KH(H_\pi)$. The nonzero ideal $\pi(A)\cap\KH(H_\pi)$ must equal
$\pi(A)$ by simplicity, so $\pi(A)=\KH(H_\pi)$. Thus $A$ is an algebra
of compact operators and has one-point spectrum.
\end{proof}

\begin{corollary}\label{cor:amenable-common}
If $H$ is a countable amenable group and $C^*(H)$ is non-type-I,
then $\defdual H\isoBD\defdual{\CAR}$. In particular,
\[
 \defdual{G_0}\isoBD\defdual{\CAR}
 \sdef\defdual B\bdef\defdual{F_\infty}.
\]
\end{corollary}
Thomas proved that any two countable amenable non-type-I groups have
unitary duals connected by Borel-definable bijections
\cite[Corollaries~6.2 and~1.7]{Thomas2015}. The CAR degree in this
comparison is already present in Thomas's Theorem~6.1 together with
Elliott's AF comparison
\cite[Theorem~4.2]{Thomas2015}, applied to the non-type-I AF algebra
$C^*(G_0)$. Corollary~\ref{cor:amenable-common} recovers these comparisons.
The extension
in Theorem~\ref{thm:nuclear-isomorphism} is from amenable groups and
Elliott's AF algebras to all separable nuclear $C^*$-algebras.
\begin{proof}
For amenable $H$, the full group algebra equals the reduced group
algebra \cite[Theorem~2.6.8]{BrownOzawa2008}, and the reduced group
algebra is nuclear \cite[Theorem~4.2]{Lance1973}. Use
\eqref{eq:group-dual} and
Theorem~\ref{thm:nuclear-isomorphism}. The algebra in
\eqref{eq:S3-algebra} is non-type-I. The remaining comparisons
are Theorem~\ref{intro:A} and pullback along the quotient of
$C^*(F_\infty)$ onto $B$.
\end{proof}

\begin{corollary}\label{cor:nuclear-minimality}
If $A$ is separable and nuclear and $D$ is any separable
non-type-I $C^*$-algebra, then
\[
 \defdual A\bdef\defdual D.
\]
\end{corollary}
\begin{proof}
Theorem~\ref{thm:nuclear-upper} and the Glimm--Elliott lower bound
give $E_A\le_B E_{\CAR}\le_B E_D$.
\end{proof}
This is the nuclear-algebra analogue
of Thomas's comparison of amenable group duals with arbitrary non-type-I
group duals \cite[Theorem~1.11 and Corollary~6.3]{Thomas2015}.
 
\section{Consequences of the definable comparison}\label{sec:consequences}

\subsection{Nuclear targets and representation universality}

The same compact probability action excludes every nuclear algebra as a
target, with a single witnessing measure independent of the target.
Combining Lemma~\ref{lem:free-quotient} with
Proposition~\ref{prop:nuclear} gives
\begin{corollary}\label{cor:nuclear-targets}
For every separable nuclear $C^*$-algebra $D$,
\[
 \defdual D\sdef\defdual {F_\infty}.
\]
Moreover, there is no Borel-definable injection
$\defdual B\to\defdual D$. Every Borel-definable map in this
direction is essentially constant modulo representation equivalence
for the probability $\nu$ in Theorem~\ref{thm:simple-separation},
also on each positive-measure Borel restriction. The probability $\eta$
in Theorem~\ref{thm:group-separation} has the analogous property for
maps from the unitary dual of the free group into the spectrum of $D$.
In particular no separable nuclear algebra has the maximal irreducible
representation degree among separable unital $C^*$-algebras.
\end{corollary}
\begin{proof}
For unital $D$, the forward injection comes from a quotient of
$C^*(F_\infty)$; for nonunital $D$, first include its spectrum in that of
its unitization. A homomorphism from the free-group or crossed-product
presentation, composed with $j$ or $i$, is an orbit homomorphism from
$R_\Gamma^K$. Theorem~\ref{intro:C} gives essential constancy and its
positive-restriction version. Every source equivalence class is null
for the specified probability, so a reduction is impossible.
\end{proof}

Proposition~\ref{prop:simple} shows that $B$ is exact and nonnuclear, in
addition to being simple. Consequently Theorems~\ref{intro:A} and~\ref{intro:D} answer the
nuclear-to-exact part of Farah's Problem~9 \cite{Farah2011} in the
Borel-liftable, existential sense: complexity strictly increases from the
common non-type-I nuclear degree to the spectrum of a simple exact algebra.
Equivalently, nuclearity in Theorem~\ref{intro:D} cannot be weakened to exactness.
The construction of the reduced crossed-product expectation in
Proposition~\ref{prop:simple} also makes the
nonnuclearity argument independent of the definable-spectrum comparison.

There is also a purely infinite, $\mathcal O_2$-absorbing form of the
separation. Let $\mathcal P_B$ be Kumjian's algebra from
Section~\ref{sec:fock}, formed with coefficient algebra $B$, and
put
\[
 B'=\mathcal P_B\otimes_{\min}\mathcal O_2.
\]
The algebra $B'$ is separable and unital. The algebra $\mathcal P_B$ is
simple by Kumjian's construction, while $\mathcal O_2$ is simple.
Therefore $B'$ is simple by Takesaki's tensor-product simplicity
theorem \cite[Corollary to Theorem~2]{Takesaki1964}. Since $B$ is exact,
$\mathcal P_B=\mathcal T_{\mathcal E}$ is exact by
\cite[Theorem~7.1]{Katsura2004}. The Cuntz algebra $\mathcal O_2$ is nuclear
and hence exact, and the class of exact $C^*$-algebras is closed under
minimal tensor products. Indeed, this follows directly by applying the
defining exactness property successively to the two factors and using
associativity of $\otimes_{\min}$. Thus $B'$ is exact.

For every $C^*$-algebra $X$, the tensor product $X\otimes\mathcal O_2$ is
$\mathcal O_2$-absorbing, simply because
\[
 (X\otimes\mathcal O_2)\otimes\mathcal O_2
 \cong X\otimes(\mathcal O_2\otimes\mathcal O_2)
 \cong X\otimes\mathcal O_2.
\]
Applying this with $X=\mathcal P_B$ shows that $B'$ absorbs
$\mathcal O_2$. Finally, $\mathcal P_B$ is simple and purely infinite
by \cite[Theorem~2.8]{Kumjian2004}, hence strongly purely infinite
by \cite[Remarks~3.2(iii)--(iv)]{KirchbergSierakowskiToAppear}. Since
$\mathcal O_2$ is exact, \cite[Theorem~1.3]{KirchbergSierakowskiToAppear}
makes $B'$ strongly purely infinite, and therefore purely infinite
by the same Remarks~3.2(iii).

Fock induction and Lemma~\ref{lem:nuclear-tensor} give
\[
 E_B\le_B E_{\mathcal P_B}\le_B E_{B'}.
\]
If $B'$ were nuclear, Theorem~\ref{thm:nuclear-upper} would yield
$E_B\le_B E_{\CAR}$, contradicting Theorem~\ref{intro:A}. Thus $B'$ is nonnuclear,
and the same argument excludes $E_{B'}\le_B E_{\CAR}$. Consequently
\begin{equation}\label{eq:o2-absorbing-separation}
 \defdual{\mathcal O_2}\isoBD\defdual{\CAR}
 \sdef\defdual{B'}.
\end{equation}
In particular, $\mathcal O_2$ and $B'$ are both simple, separable, unital,
purely infinite, $\mathcal O_2$-absorbing, and exact, while their definable
spectra have different degrees.

For any non-type-I separable nuclear $D$,
Theorem~\ref{thm:nuclear-isomorphism} gives the sharper chain
\[
 \defdual D\isoBD\defdual {\CAR}\sdef\defdual B
 \bdef\defdual {F_\infty}.
\]
For type-I nuclear $D$ the first comparison is a strict injection
into the CAR spectrum. Thus every separable nuclear spectrum lies below
the simple crossed-product spectrum as well as strictly below the
unitary dual of the free group.

This hierarchy lies above the common nonsmooth lower bound. Let
$E_{\ell^2}$ be the relation on $[0,1]^{\N}$ defined by
$xE_{\ell^2}y$ when $x-y\in\ell^2$. Farah's dichotomy gives a continuous
reduction $E_{\ell^2}\le_B E_D^{\mathrm p}$ for every separable
non-type-I $D$ \cite[Theorem~1]{Farah2011}; Kerr--Li--Pichot independently obtain
the corresponding nonclassifiability by turbulence
\cite[Theorems~2.5 and~2.8]{KLP2010}.
Combining this lower bound with the results above gives, for every
separable nuclear non-type-I $D$, the schematic comparison
\[
 E_{\ell^2}\ \le_B\ E_{\CAR}^{\mathrm{irr}}
 \ \sim_B\ E_D^{\mathrm{irr}}
 \ <_B\ E_B^{\mathrm{irr}}
 \ \le_B\ E_{F_\infty}^{\mathrm{irr}}.
\]
Thus the CAR degree and the crossed-product degree are two distinct
non-type-I degrees above the same classical lower bound; the last
comparison is not asserted to be strict.

Thomas proves that $F_2$ is representation universal
\cite[Theorem~1.10]{Thomas2015}; the paragraph preceding that theorem,
p.~1764, records the consequence for every countable nonabelian free group. He also proves that the
unitary duals of all countable amenable non-type-I groups are connected
by Borel-definable bijections
\cite[Corollaries~6.2 and~1.7]{Thomas2015}. Thus the strict
comparison extends from $F_\infty$ to $F_r$, $2\le r<\infty$, on the
infinite-dimensional parts. In the notation of Corollary~\ref{cor:Thomas},
\begin{equation}\label{eq:finite-free-group-comparison}
 \mathord{\approx}_{G_0}<_B\mathord{\approx}_{F_r}
 \qquad(2\le r<\infty).
\end{equation}
Indeed, representation universality gives the forward reduction and
$\mathord{\approx}_{F_\infty}\le_B\mathord{\approx}_{F_r}$; a reverse
reduction would contradict Corollary~\ref{cor:Thomas}. More generally,
the comparison holds with any representation-universal group on the
larger side. These extensions use the cited identifications; the proof
of Corollary~\ref{cor:Thomas} did not require them.

\subsection{Cocycles}

For brevity, write $E_D=E_D^{\mathrm{irr}}$ in this subsection.
Let $E_0$ be eventual equality on $2^{\N}$ and let $E_\infty$ be the
orbit relation of the Bernoulli shift of $F_2$ on $2^{F_2}$; equip both
spaces with the fair product probability measure. Fix $H=\ell^2(\N)$.
For either measured relation $E$ on $(X,\mu)$, put
\[
 M_E(A)=\int_X\#\{y\in X:(x,y)\in A\}\,d\mu(x)
 \qquad(A\subseteq E\text{ Borel}).
\]
Following \cite[Definitions~5.5--5.6]{Thomas2015}, a unitary cocycle is
a Borel map $\sigma:E\to\mathcal U(H)$ satisfying
$\sigma(x,z)=\sigma(y,z)\sigma(x,y)$ for all $xEyEz$ in an
$E$-invariant conull Borel set. Cocycles are identified when they agree
$M_E$-almost everywhere. Such a cocycle is \emph{irreducible} if every
essentially bounded Borel field $T:X\to\BH(H)$ satisfying
$\sigma(x,y)T(x)=T(y)\sigma(x,y)$ for $M_E$-almost every $(x,y)$ is
$\mu$-almost everywhere a constant scalar multiple of the identity.
On the standard Borel space of these irreducible cocycles, write
$\sigma\equiv_E\tau$ if there is a Borel map $v:X\to\mathcal U(H)$ with
\[
 \tau(x,y)=v(y)\sigma(x,y)v(x)^*
 \quad\text{for }M_E\text{-almost every }(x,y).
\]
Let $\mathscr C_0=\mathord{\equiv}_{E_0}$ and
$\mathscr C_\infty=\mathord{\equiv}_{E_\infty}$ for these measured models.
For countable amenable groups, Sutherland constructed cocycle embeddings
preserving all intertwiner spaces, and used them to obtain Borel functors
into representation categories of suitable amenable semidirect products
\cite[Corollary~2, Theorem~4 and Remark~(3), pp.~131--133]{Sutherland1983}.
Thomas's identifications in \cite[Theorems~5.8 and~6.1]{Thomas2015},
together with the $F_2$ case of \eqref{eq:finite-free-group-comparison}, give
\[
 \mathscr C_0<_B \mathscr C_\infty.
\]
This establishes the cocycle separation motivating his
Question~6.10. It does not identify $\mathscr C_0$ with $E_{\ell^2}$ or decide
the analogous question for ergodic quasi-invariant measures.

To compare Thomas's infinite-dimensional presentation with the full
spectrum, write $E_{G_0}^{\infty}$ for unitary equivalence on
$\Irr_\infty(G_0)$. Pullback along the unital surjection
$C^*(G_0)\to\CAR$ from \eqref{eq:S3-algebra} preserves dimensions, and
every irreducible CAR representation is infinite-dimensional. Hence
\[
 E_{\CAR}\le_B E_{G_0}^{\infty}\le_B E_{C^*(G_0)}\sim_B E_{\CAR}.
\]
The last comparison also follows from Elliott's comparison of all pure
states \cite[Corollary~3]{Elliott1977} and the GNS equivalence in
Proposition~\ref{prop:gns}.
Thus Thomas's Theorem~6.1 and Theorem~\ref{thm:nuclear-isomorphism} show that
$\mathscr C_0$ is Borel bireducible with the irreducible-representation relation
of every separable nuclear non-type-I $C^*$-algebra. Theorem~\ref{intro:C} can
therefore be restated without $C^*$-algebras: every Borel homomorphism
from the orbit relation of an ergodic probability-preserving Borel
action of a countable property-$(T)$ group on a standard Borel probability space
into $\mathscr C_0$ is essentially constant. If the source measure is nonatomic,
no Borel reduction to $\mathscr C_0$ exists.

There is a parallel consequence for the turbulent relation
$E_{\ell^2}$. Thomas's Theorem~6.12 asserts
$E_\infty\not\le_B E_{\ell^2}$; the target there is the $\ell^2$-coset
relation on $\mathbb R^{\N}$, not a cocycle relation. Its restriction
to $(0,1)^{\N}$ is Borel bireducible with it
\cite[Definition~6.6 and the following paragraph, p.~1778]{Thomas2015}.
The inclusions $(0,1)^{\N}\subseteq[0,1]^{\N}\subseteq\mathbb R^{\N}$
therefore give the same Borel degree for our convention; an affine
change of scale also gives that degree on $(-\pi/2,\pi/2)^{\N}$.
The following essential-constancy
statement already follows from Anantharaman--Delaroche's cohomology
theorem for measured property-$(T)$ groupoids
\cite[Theorem~4.12 and Corollary~5.16(i), arXiv version]{ADGroupoids}.
We give a second proof as a consequence of the nuclear-spectrum
rigidity theorem.

\begin{corollary}\label{cor:ell2-rigidity}
Let $R$ be the orbit relation of an ergodic probability-preserving
Borel action of a countable property-$(T)$ group on a standard Borel
probability space with nonatomic measure. Every
Borel homomorphism $R\to E_{\ell^2}$ is essentially constant modulo
$E_{\ell^2}$. Consequently $R\not\le_B E_{\ell^2}$ and
$E_\infty\not\le_B E_{\ell^2}$.
\end{corollary}
\begin{proof}
Farah's continuous reduction
$E_{\ell^2}\le_B E_{\CAR}^{\mathrm p}$
\cite[Theorem~6]{Farah2011}, restricted to $[0,1]^{\N}$,
followed by the Borel GNS map of
Proposition~\ref{prop:gns}, is a Borel reduction into
$E_{\CAR}^{\mathrm{irr}}$. Composing a homomorphism $R\to E_{\ell^2}$
with this reduction and applying the CAR case of Theorem~\ref{intro:C} shows that the
composite is essentially constant. Because the second map reflects
equivalence, the original homomorphism is essentially constant modulo
$E_{\ell^2}$.

If it were a reduction, a conull set would lie in a single $R$-class.
Orbit classes are countable and hence null for a nonatomic probability,
a contradiction. For the last assertion, take $R=R_\Gamma^K$ from
Section~\ref{sec:crossed}. By the Dougherty--Jackson--Kechris universality
theorem \cite[Proposition~1.8]{DJK1994}, every countable Borel equivalence
relation reduces to $E_\infty$; in particular, $R\le_B E_\infty$. Thus a
reduction
$E_\infty\le_B E_{\ell^2}$ would imply the already excluded reduction
$R\le_B E_{\ell^2}$.
\end{proof}

For completeness, the earlier cohomological implication is direct.
If $f:X\to\mathbb R^{\mathbb N}$ is a Borel homomorphism from the
orbit relation of $\Gamma\curvearrowright(X,\mu)$ to $E_{\ell^2}$,
put
\[
 b(g,x)=f(g\cdot x)-f(x)\in\ell^2.
\]
This is a Borel cocycle for the trivial Hilbert-space representation
of the transformation groupoid, with arrows $(g,x):x\to g\cdot x$.
Borelness as an $\ell^2$-valued map follows from the Borel inverse of
the inclusion $\ell^2\to\mathbb R^{\mathbb N}$ onto its image.
The groupoid has property~$(T)$ by
\cite[Corollary~5.16(i), arXiv version]{ADGroupoids}, and
\cite[Theorem~4.12, arXiv version]{ADGroupoids} gives a Borel map
$a:X\to\ell^2$ such that
$b(g,x)=a(g\cdot x)-a(x)$ almost everywhere.
Because $\Gamma$ is countable, these identities hold simultaneously
on an invariant conull Borel set. Each coordinate of $f-a$ is then
invariant and hence essentially constant by ergodicity. Thus
$f(x)\in c+\ell^2$ almost everywhere for some
$c\in\mathbb R^{\mathbb N}$.
This argument needs no freeness assumption or integrability of $a$;
nonatomicity is used only to exclude reductions. The proof above
deduces the same conclusion from the more general nuclear-spectrum
theorem.

These results also place the common nuclear non-type-I degree outside
the countable Borel degrees.
\begin{corollary}\label{cor:nuclear-incomparability}
Let $D$ be a separable nuclear non-type-I $C^*$-algebra, and let $R$
be the orbit relation of an ergodic probability-preserving Borel action
of a countable property-$(T)$ group on a standard Borel probability
space with nonatomic measure. Then $E_D$ is Borel incomparable with
both $R$ and $E_\infty$.
\end{corollary}
\begin{proof}
By Theorem~\ref{thm:nuclear-isomorphism}, it suffices to take $D=\CAR$.
Theorem~\ref{intro:C} gives $R\not\le_B E_{\CAR}$. Conversely,
$E_{\CAR}\not\le_B R$: every countable Borel equivalence relation is
classifiable by countable structures, whereas the CAR spectrum is not
\cite[Theorem~2.8]{KLP2010}. Since $R\le_B E_\infty$, the same arguments
give $E_\infty\not\le_B E_{\CAR}$ and $E_{\CAR}\not\le_B E_\infty$.
\end{proof}
The nonatomicity
assumption on $R$ is essential; a one-point probability action has a
one-class orbit relation, which reduces to the CAR spectrum.

\section{A universal simple spectrum and the endpoint reduction}\label{sec:universal}

Put
\[
 C_\infty=C^*(F_\infty),\qquad \mathcal U=\mathcal P_{C_\infty},
\]
using a fixed faithful unital representation
$\rho:C_\infty\to\BH(H_0)$ with
$\rho(C_\infty)\cap\KH(H_0)=\{0\}$, as in Section~\ref{sec:fock}.

\begin{lemma}\label{lem:universalquotient}
Every nonzero separable $C^*$-algebra $D$ has a measurable embedding
$Q_D\qemb Q_{C_\infty}$ with Borel lifts on its image.
\end{lemma}
\begin{proof}
For unital $D$, Lemma~\ref{lem:free-quotient} gives a surjection
$p:C_\infty\to D$. Pullback is Borel and preserves all intertwiners.
Its image is exactly the invariant Borel set of irreducible
representations annihilating $\ker p$, tested on a countable dense
subset of that ideal. Factoring through $p$ is Borel: fix preimages
of a countable dense subset of $D$, then use norm approximation.
For nonunital $D$, extend representations to its unitization and
remove the augmentation class as in Section~\ref{sec:duals}.
Restriction is the inverse on this invariant Borel image. Composing
with the unital pullback proves the assertion.
\end{proof}

\begin{theorem}\label{thm:universal}
The algebra $\mathcal U$ is separable, unital, simple, purely infinite,
and nonexact. Moreover,
\[
 Q_{\mathcal U}\bd Q_{C_\infty}.
\]
Every separable $C^*$-algebra spectrum measurably embeds into
$Q_{\mathcal U}$, with Borel lifts on its image.
\end{theorem}
\begin{proof}
The structural assertions except nonexactness follow from Kumjian's
theorem. The coefficient inclusion embeds $C_\infty$ into $\mathcal U$.
The canonical copy of $C^*(F_2)$ in $C_\infty$ is injective (the map
fixing the first two free generators and sending the others to $1$
is a left inverse). The nonexactness of $C^*(F_2)$ is stated explicitly in
\cite[Corollary~3.7.12, p.~98]{BrownOzawa2008}.
For precision, put $A=C^*(F_2)$ and let
$\lambda:A\to C_r^*(F_2)$ be the regular quotient, with kernel $I$.
Wassermann's ideal-slice example
\cite[Remark after Proposition~(2.7) and Remark~(2.8)(iii),
pp.~251--252]{Wassermann} gives
\[
 A\otimes_{\min}I\ \subsetneq\
 \ker(\mathrm{id}_A\otimes\lambda).
\]
Indeed, his slice-defined ideal consists exactly of the elements
whose image under $\mathrm{id}_A\otimes\lambda$ has every slice
zero. This is a failure of exactness for $A$. Since exactness
passes to $C^*$-subalgebras, $\mathcal U$ is nonexact.
This extra property is not needed for the spectrum reduction.

Proposition~\ref{prop:sector} gives an embedding from $Q_{C_\infty}$
into $Q_{\mathcal U}$ with Borel lifts on its image, and
Lemma~\ref{lem:universalquotient} gives one in the reverse direction.
The Cantor--Bernstein construction in Lemma~\ref{lem:cb} uses Borel
lifts on Borel invariant pieces in this case, so both the resulting
bijection and its inverse have Borel lifts.
Compose Lemma~\ref{lem:universalquotient} with Fock induction to obtain
the final assertion. The zero algebra has empty spectrum.
\end{proof}

\subsection{Equivalent forms of the measurable assertion}

The other endpoint comes from the Glimm--Elliott lower bound:
the Borel reduction in \eqref{eq:Glimm-lower} holds for every
separable non-type-I $D$. By Remark~\ref{rem:embedding},
\begin{equation}\label{eq:sandwich}
 Q_{\CAR}\qemb Q_D\qemb Q_{C_\infty}\bd Q_{\mathcal U}.
\end{equation}

The assertion $\Phi$ from the Introduction is statement (1) below.

\begin{theorem}\label{thm:equivalence}
The following statements are equivalent in ZFC.
\begin{enumerate}
\item The quotient Borel spectra of all simple separable non-type-I
$C^*$-algebras are isomorphic as measurable spaces.
\item The same assertion holds for all separable non-type-I
$C^*$-algebras, without simplicity or unitality assumptions.
\item $Q_{C^*(F_\infty)}\qiso Q_{\CAR}$.
\item $Q_{\mathcal U}\qiso Q_{\CAR}$.
\item There is a measurable embedding
$Q_{C^*(F_\infty)}\qemb Q_{\CAR}$.
\end{enumerate}
Statement (1) may equivalently be restricted to unital simple
separable non-type-I algebras.
\end{theorem}
\begin{proof}
Statement (2) implies (1), including its unital restriction. The latter
implies (4), since CAR and $\mathcal U$ satisfy all those hypotheses.
Here $\mathcal U$ is non-type-I because a nonzero simple unital type-I
algebra is a finite matrix algebra, whereas $\mathcal U$ is purely
infinite. Theorem~\ref{thm:universal} identifies (3) and (4), and (3)
implies (5).

Assume (5) and let $D$ be any separable non-type-I algebra.
The right-hand part of \eqref{eq:sandwich}, followed by (5), gives a
measurable embedding $Q_D\qemb Q_{\CAR}$. Its left-hand part gives
the reverse embedding. Lemma~\ref{lem:cb} yields
$Q_D\qiso Q_{\CAR}$. Thus (2) holds.
\end{proof}

\begin{remark}
In (5), ``measurable embedding'' cannot simply be replaced, on the
basis of this proof, by ``inverse-image measurable injection.''
Measurability of the image and inverse are needed for the
Cantor--Bernstein argument. No automatic image theorem for a bare
measurable injection has been established here.
\end{remark}

The algebra $\mathcal U$ realizes the full universal spectral degree
within the simple class. Thus a comparison with the free-group spectrum
can be stated as a comparison of two fixed simple algebras. The
remaining issue in the endpoint criterion is the ordinary measurable
map itself, rather than simplicity of the endpoint.

\subsection{The endpoint reduction without choice}
\label{sec:zf-endpoints}

In this subsection, work in ZF with countable codes for separable
$C^*$-algebras. We use the GCR formulation of type I: every
irreducible representation has image containing the compact
operators on its Hilbert space. This convention is explicit here;
we do not import an equivalence with every other characterization
of type I into bare ZF without an additional argument.

\begin{lemma}[Transfer of coded embeddings]\label{lem:zf-coded}
Fix real Borel codes for standard Borel presentations $X/E$ and
$Y/F$. The existence of coded Borel maps $f:X\to Y$, $g:S\to X$
and a coded Borel $F$-invariant set $S\subseteq Y$ such that
\[
 \begin{gathered}
 xEx'\ \Longleftrightarrow\ f(x)Ff(x'),\qquad f(X)\subseteq S,\\
 g(f(x))Ex\quad(x\in X),\qquad f(g(y))Fy\quad(y\in S)
 \end{gathered}
\]
is a $\boldsymbol\Sigma^1_2$ assertion. Such witnesses induce a
measurable embedding for the full sigma-generated quotient Borel
structures in ZF.
\end{lemma}
\begin{proof}
Use coded Borel realizations inside Baire space. Code a map by a
sequence of Borel partitions, one partition for each output
coordinate. Validity of the codes and the partition conditions,
range containment, invariance, and all the displayed identities
are jointly $\Pi^1_1$ conditions on these real codes. Equivalently,
after valid codes are fixed, a violation of any pointwise condition
is the nonemptiness of a coded Borel set. Existential quantification
over the witness codes gives $\Sigma^1_2$. This uses coordinate
partitions, not an arbitrary graph accompanied by a separate
$\forall x\exists y$ totality requirement.

The identities show that $S/F$ is exactly the quotient image and
that $g$ induces its inverse. A coded Borel map is measurable for
the full sigma-generated Borel structure even without countable
choice: the target subsets having Borel inverse image form a
$\sigma$-algebra containing all open sets. Relative Borel subsets
of the Borel set $S$ are ambient Borel. Apply this observation to
both maps and then restrict to invariant subsets.
\end{proof}

\begin{corollary}[ZF endpoint equivalence]\label{cor:zf-endpoints}
With the GCR convention above, all five assertions of
Theorem~\ref{thm:equivalence} are equivalent in ZF. The quotient
Borel structures here are the full sigma-generated structures.
\end{corollary}
\begin{proof}
We transfer the coded embeddings pointwise, for each algebra code;
we do not assert that a universal quantification over algebra codes
has complexity $\Sigma^1_2$.

Let $a$ code a separable non-GCR algebra $D$. Non-GCR has an
analytic formulation. Indeed, with $(d_j)$ norm dense in $D$ and
$p_0$ a fixed rank-one projection on $\ell^2$, it is equivalent to
the existence of an irreducible representation $\pi$ on $\ell^2$
and a rational $r>0$ such that
\[
                    \|\pi(d_j)-p_0\|\ge r\qquad(j\in\N).
\]
The represented algebra is norm closed, and an irreducible
algebra containing one nonzero compact operator contains all
compact operators. Irreducibility on a fixed separable carrier
is a coded Borel condition. Thus this formulation is analytic
in $a$ and is absolute between the universe and $L[a]$.

Inside $L[a]$, the ZFC Glimm--Elliott lower bound, the
image/inverse theorem of Remark~\ref{rem:embedding}, and
Lemma~\ref{lem:universalquotient} yield the strong coded
embeddings of Lemma~\ref{lem:zf-coded} in both parts of
\eqref{eq:sandwich} involving $D$. The pure-state equivalence
relations have uniform Borel codes: the norm-$<2$ criterion
after inner conjugacy can be tested on fixed countable dense
sets of unitaries and algebra elements. Hence the codes used
for the comparison denote the same relations in the larger
universe. Shoenfield absoluteness, in the ZF formulation
\cite[Corollary 5.6.11 and Theorem 5.8.1]{BlackadarFarah},
transfers the existence of the strong coded witnesses to the
universe. The upper embedding for arbitrary $D$ is treated
in the same way, without the non-GCR hypothesis.

For the fixed universal endpoint, choose once in $L$ the full
countable construction code for $\mathcal U$, including its
coefficient algebra and faithful essential representation.
Interpret the coded completion in the universe under discussion.
Its unitality is absolute; simplicity is a $\Pi^1_1$ condition
on the algebra code, and non-GCR is analytic as above.
These properties therefore transfer from $L$. The coded
embeddings between $Q_{\mathcal U}$ and $Q_{C_\infty}$ transfer
by Lemma~\ref{lem:zf-coded} and the same absoluteness theorem.
No transfer of nonexactness or pure infiniteness is needed.

All the resulting embeddings are measurable for the full Borel
structures by Lemma~\ref{lem:zf-coded}. The Cantor--Bernstein
proof in Lemma~\ref{lem:cb} uses an explicitly defined sequence
of measurable sets and requires no choice. The proof of
Theorem~\ref{thm:equivalence} now applies in ZF.
\end{proof}

\section{Orbit closures and local quotient comparisons}\label{sec:ch-local}

We return to ZFC. This section prepares the proof of Theorem~\ref{intro:E} by
constructing essentially hyperfinite pieces that retain any prescribed
countable family of quotient Borel tests. The difficulty is to make
these local comparisons compatible at countable limits. A fixed closed
cover of the orbit relation controls the locally closed orbits, and a
uniform fusion construction retains the remaining orbit closures.

An $F_\sigma$ Polish-group orbit quotient means a quotient $X/E$
of a continuous action of a Polish group on a Polish space, with $E$
an $F_\sigma$ subset of $X^2$ in that presentation.
For a quotient $Q=X/E$, let $\Sigma_Q$ be its quotient Borel
$\sigma$-algebra. A measurable subset of $Q$ is called smooth if the
corresponding invariant restriction of $E$ is smooth, and essentially
hyperfinite if that restriction is Borel reducible to $E_0$. On the
orbit quotients considered here, a smooth measurable subset has a
standard Borel presentation with lifts in both directions, by
Lemmas~\ref{lem:orbit-idealistic} and~\ref{lem:smooth-strong-selector}.
Countable invariant unions preserve each of
these properties, by disjointifying their presentations. Write
$Q_0=2^{\N}/E_0$.

\subsection{Topological refinements and smooth hulls}

\begin{lemma}[Smooth hulls continuous at countable limits]
\label{lem:continuous-smooth-hulls}
Let a Polish group $G$ act continuously on a Polish space $(X,\tau_0)$,
and fix closed sets $K_n\subseteq X^2$ such that
$E_G^X=\bigcup_n K_n$. For every finer Polish topology $\tau$
with the same Borel sets and continuous $G$-action, let $L(\tau)$
be the set of points whose orbits are locally closed in $(X,\tau)$.
Then:
\begin{enumerate}
\item $L(\tau)$ is invariant Borel and $E_G^X\restriction L(\tau)$
is smooth.
\item If an invariant $\tau$-continuous map $a:X\to2^\N$ has
a smooth fiber, that entire fiber is contained in $L(\tau)$.
\item If $\tau_m$ is an increasing sequence of such topologies and
$\tau=\bigvee_m\tau_m$, then $\tau$ has the same properties and
\[
                         L(\tau)=\bigcup_m L(\tau_m).
\]
\end{enumerate}
\end{lemma}
\begin{proof}
Fix a countable dense set $D\subseteq G$ and a countable
$\tau$-basis $(U_i)$. A point $x$ belongs to $L(\tau)$ exactly when
some $i,n$ satisfy
\begin{equation}\label{eq:locally-closed-witness}
 Gx\cap U_i\ne\varnothing,
 \qquad Gx\cap U_i\subseteq(K_n)_x.
\end{equation}
Indeed, these conditions imply
$\overline{Gx}^{\tau}\cap U_i\subseteq(K_n)_x\subseteq Gx$.
Translating this nonempty relatively open set shows that $Gx$ is
open in its closure. Conversely, a locally closed orbit is Polish;
the closed sections $(K_n)_x$ cover it, so Baire's theorem supplies
a witness of the displayed form. Both conditions in
\eqref{eq:locally-closed-witness} can be checked with $d\in D$
in place of $g\in G$. For the containment condition this follows
because its failure is an open condition on $g$. Thus $L(\tau)$
is Borel. The map $x\mapsto\overline{Gx}^{\tau}$ into the Effros
space of closed sets is Borel. Two locally closed orbits with the
same closure are both open dense subsets of that closure, hence
intersect. This map therefore classifies the orbits on $L(\tau)$.

For (2), take $x$ in a smooth fiber of $a$. The invariant closed
space $C=\overline{Gx}^{\tau}$ lies in that fiber and has a dense
orbit. Every invariant Borel subset of $C$ is meager or comeager:
a nonmeager such set is comeager on a nonempty open set, hence on
all its translates, whose union is dense; second countability
reduces this union to countably many translates. A countable
invariant separating family for the smooth relation on $C$ now
has one comeager membership cell, contained in a single orbit.
That orbit is $F_\sigma$ in $C$, so contains a nonempty relatively
open set. The dense orbit $Gx$ meets this set and equals that orbit.
Thus $x\in L(\tau)$.

The countable join in (3) is Polish
\cite[Theorem~5.1.3(b)]{BeckerKechris1996}. Explicitly, the diagonal
copy of $X$ in $\prod_m(X,\tau_m)$ is closed: two unequal coordinates
are separated by the common Hausdorff topology $\tau_0$. Its
subspace topology is precisely the join. This is also the argument
in \cite[proof of Lemma~4.2]{Sami1994}. Finite intersections
from the countably many countable bases give a countable base
of original Borel sets, so its Borel structure is unchanged.
The action remains continuous by checking inverse images of
the subbasic open sets coming from the individual topologies.
The original fixed $K_n$ remain closed in every finer topology.
A witness for membership in $L(\tau)$ can be shrunk to a basic
open neighborhood in the join. Such a neighborhood is a finite
intersection of opens from the $\tau_m$, and thus belongs to
one $\tau_m$. It gives the same witness with the same $K_n$
there. The reverse inclusion follows from the same witness
criterion under refinement. This proves (3).
\end{proof}

Every countable family of invariant Borel tests can be made clopen
in a finer compatible Polish $G$-topology
\cite[Corollary~5.1.6, p.~55]{BeckerKechris1996}.
The continuity assertion in part~(3) uses a fixed closed cover of
the relation throughout; it keeps locally closed orbits inside the
accumulated smooth core at limit stages.

\subsection{Uniform fusion and its countable limits}

\begin{lemma}[Uniform reductions on orbit closures]
\label{lem:uniform-orbit-closure}
Let $G$ act continuously on a Polish space $X$, with $F_\sigma$
orbit relation $E$, and let $L$ be its locally closed orbit locus.
The family
\[
 \mathcal D=\{\overline{Gx}:x\in X\setminus L\}
\]
is Borel in the Effros space of closed subsets of $X$. There is a
Borel injection $F:\mathcal D\times2^\N\to X\setminus L$ such that
\[
 F(C,a)\mathrel E F(C',b)
       \quad\Longleftrightarrow\quad C=C'\text{ and }aE_0b,
 \qquad \overline{G F(C,a)}=C.
\]
Its quotient image is measurable and its quotient inverse is
measurable. In particular this construction preserves every
invariant continuous coordinate on $X$.
\end{lemma}
\begin{proof}
Fix a compatible complete metric, a countable basis $(V_i)$,
and a countable dense subgroup $\Lambda\subseteq G$. Write
$E=\bigcup_n K_n$, with the $K_n$ closed, increasing, symmetric,
and containing the diagonal. The following is a uniform form of
the elementary fusion in
\cite[Theorem~3.4.5, pp.~38--40]{BeckerKechris1996}; we give the
parameter and orbit-closure details explicitly.

Consider the nonempty closed invariant sets $C$ on which the
action is topologically transitive and each $K_n\cap C^2$ is
nowhere dense. These conditions are Borel in $C$. Invariance can
be tested on $\Lambda$. Transitivity says that, whenever $C$ meets
$V_i,V_j$, some $\gamma\in\Lambda$ has
$C\cap V_i\cap\gamma^{-1}V_j\ne\varnothing$. Nowhere density
says that for any such pair $V_i,V_j$ there are basic
$V_k\subseteq V_i,V_l\subseteq V_j$ meeting $C$ with
$(V_k\times V_l)\cap K_n=\varnothing$. The permissible pairs
$(k,l)$ are fixed countable data, and every remaining test is an
Effros open-hit test. For the converse from relative nowhere
density to this ambient avoidance test, use closedness of $K_n$
to put an ambient basic rectangle around a point of $C^2\setminus K_n$.

This Borel family is exactly $\mathcal D$. Indeed, if
$C=\overline{Gx}$ with $x\notin L$, every orbit in $C$ is meager:
a nonmeager $F_\sigma$ orbit contains a relatively open set, which
the dense orbit $Gx$ meets, making $Gx$ locally closed.
Kuratowski--Ulam then makes $E\cap C^2$ meager. Conversely,
topological transitivity gives a dense orbit in $C$, and
meagerness of $E\cap C^2$ precludes its being locally closed.

Choose a dense-orbit point $r(C)$ Borelly. Here is a direct
selection. Put
\[
 W_j(C)=\bigcup_{\gamma\in\Lambda}\gamma^{-1}V_j
 \quad\text{if }C\cap V_j\ne\varnothing,
 \qquad W_j(C)=X\quad\text{otherwise}.
\]
Each is open dense relative to $C$. By first-index search choose
nested basic balls meeting $C$, with diameters tending to zero
and with each closure inside its predecessor and the next
$W_j(C)$. Their limit is a Borel function $r(C)$ and has dense
orbit in $C$. Only fixed basic-neighborhood inclusion tests and
Borel hit tests occur in this search: for fixed $j$, the set
$W_j(C)$ is either $X$ or the fixed open set
$\bigcup_{\gamma\in\Lambda}\gamma^{-1}V_j$, according to one
Effros hit test.

For a fixed $C$, abbreviate $r=r(C)$. Construct neighborhoods
$U_n$ of $r$ and elements $\gamma_n\in\Lambda$. For
$s\in2^n$ put $g_s=\gamma_0^{s(0)}\cdots
\gamma_{n-1}^{s(n-1)}$. Start with $U_0=X$. Choose
$\gamma_n r\in U_n$ such that
\[
                    (g_s r,g_t\gamma_n r)\notin K_n
                              \qquad(s,t\in2^n).
\]
Every orbit in $C$ is meager, as established above, so each
relevant closed vertical section of $K_n$ is nowhere dense in
$C$. The forbidden points therefore form a finite union of closed
nowhere dense subsets of $C$, and the dense set $\Lambda r$
supplies a choice.
Set $g_{s0}=g_s$, $g_{s1}=g_s\gamma_n$. Choose $U_{n+1}$ small
enough that
\[
 \overline{U_{n+1}},\ \gamma_n\overline{U_{n+1}}\subseteq U_n,
 \qquad\operatorname{diam}(g_{s\varepsilon}\overline{U_{n+1}})
                                                        <2^{-n},
\]
and every product
$g_{s0}\overline{U_{n+1}}\times g_{t1}\overline{U_{n+1}}$
avoids $K_n$. Also put each
$g_{s\varepsilon}\overline{U_{n+1}}$ inside $W_j(C)$ for
$j\leq n$. This last requirement is possible since all central
points have dense orbit $C$; choose a witnessing translate of
$V_j$ first, then shrink the neighborhood. All choices are Borel
in $C$ by first-index searches. The relevant inclusion, diameter,
and avoidance relations involve only countably many fixed group
elements and basic sets, so can be fixed as countable tables.

The nested closed sets $g_{a|n}\overline{U_n}$ have a unique
common point $F(C,a)$. They meet closed $C$, so this point lies
in $C$. It is the limit of the Borel functions $g_{a|n}r(C)$,
hence is jointly Borel, and is continuous in $a$ for fixed $C$.
If $a,b$ eventually agree, multiplication by
$g_{b|N}g_{a|N}^{-1}$ after their last difference takes one
limit to the other. If they differ infinitely often, the
avoidance requirements exclude every $K_n$. Containment of the
diagonal in the $K_n$ also gives injectivity. The $W_j(C)$
requirements make the output orbit dense in $C$; it is not
locally closed since $E\cap C^2$ is meager. Distinct $C$ give
inequivalent outputs.

The source relation $\Delta_{\mathcal D}\times E_0$ is countable
Borel. For each $y$, the set of $(C,a)$ with $y\mathrel E F(C,a)$
is either empty or one countable source class. Lusin--Novikov
therefore makes the saturation of the range, and the saturation
of the image of each invariant Borel source set, Borel. This
proves both assertions about the quotient embedding. An invariant
continuous function is constant on each orbit closure, proving
the final assertion.
\end{proof}

\begin{lemma}[Orbit closures at a countable inverse limit]
\label{lem:orbit-closure-inverse-limit}
Let $(\tau_n)$ be increasing compatible Polish $G$-action topologies
on $X$, with the same Borel sets, and suppose that the orbit
relation is $F_\sigma$ in $\tau_0$. Put $\tau=\bigvee_n\tau_n$.
Let $\mathcal D_n$ and $\mathcal D$ be the good-closure parameter
spaces of Lemma~\ref{lem:uniform-orbit-closure} for $\tau_n$ and
$\tau$, respectively. Coarsening the closure defines Borel maps
\[
 r_n:\mathcal D_{n+1}\to\mathcal D_n,
 \qquad r_n(C)=\overline C^{\tau_n}.
\]
There is a Borel isomorphism
\[
 \mathcal D\longrightarrow
 \varprojlim_n(\mathcal D_n,r_n),
 \qquad C\longmapsto(\overline C^{\tau_n})_n.
\]
For every compatible sequence $(C_n)$ its inverse is
$\bigcap_n C_n$. There is a Borel choice of a point $x$ with
$\overline{Gx}^{\tau_n}=C_n$ for all $n$.
\end{lemma}
\begin{proof}
The parameter spaces are Borel by
Lemma~\ref{lem:uniform-orbit-closure}. Coarsening is Borel by
open-hit tests. It maps good closures to good closures because
$L(\tau_n)\subseteq L(\tau_{n+1})$, by
Lemma~\ref{lem:continuous-smooth-hulls}.

Fix a compatible sequence. Then $C_{n+1}\subseteq C_n$ and
$C_{n+1}$ is $\tau_n$-dense in $C_n$; more generally
$\overline{C_m}^{\tau_i}=C_i$ for $m\geq i$. Fix complete metrics
$d_i$ for $\tau_i$ and countable bases $(U_{i,k})_k$. Whenever
$C_i\cap U_{i,k}\ne\varnothing$, put
\[
 W_{i,k}=\bigcup_{g\in G}g^{-1}U_{i,k}.
\]
This is invariant $\tau_i$-open and is $\tau_m$-dense relative
to every $C_m$ with $m\geq i$. Indeed, any $\tau_m$-dense orbit
in $C_m$ has $\tau_i$-closure $C_i$, so lies in $W_{i,k}$.

We give the fusion detail because finite compatibility alone
would not imply the conclusion. Recursively choose
$x_n\in C_n$ and $\tau_n$-open neighborhoods $V_n$ of $x_n$.
Choose $x_{n+1}$ in $C_{n+1}\cap V_n$ and in all the finitely
scheduled $W_{i,k}$ with $i,k\leq n+1$ whose hit test is nonempty.
Such a choice exists: $C_{n+1}$ meets $V_n$ by $\tau_n$-density,
and the finitely many hit sets are open dense in $C_{n+1}$ for
$\tau_{n+1}$. Shrink $V_{n+1}$ to lie in $V_n$ and those hit
sets, with $d_i$-diameter less than $2^{-n-1}$ for $i\leq n+1$.
Also require its $\tau_i$-closure to lie in each scheduled
$W_{i,k}$ with $k\leq n+1$, $i\leq n+1$; the finite requirement
is possible by using a smaller neighborhood in each $\tau_i$.
Start with the analogous finite requirements for $V_0$.

The points $x_n$ are Cauchy in each $d_i$. Their limits coincide
as points of $X$, since all topologies refine the Hausdorff
topology $\tau_0$. Call the common limit $x$. Each $C_i$ is
$\tau_i$-closed and contains all later $x_n$, so $x\in C_i$.
The closure buffers give $x\in W_{i,k}$ whenever $C_i$ meets
$U_{i,k}$. Thus $\overline{Gx}^{\tau_i}=C_i$ for all $i$.
Since $x\notin L(\tau_i)$ for every $i$, continuity of the smooth
hulls gives $x\notin L(\tau)$.

The orbit of $x$ is $\tau$-dense in $\bigcap_i C_i$. To see this,
a basic join-open neighborhood belongs to some $\tau_N$;
if it meets the intersection, it meets $C_N$ and hence $Gx$.
The intersection is $\tau$-closed, proving that it is exactly
$\overline{Gx}^{\tau}$. This also proves injectivity of the
displayed coarsening map: its coordinates determine the hit
tests for every basic join-open set. The map is Borel and is now
bijective between standard Borel spaces, so its inverse is Borel.
Composing that inverse with the Borel dense-orbit selector from
Lemma~\ref{lem:uniform-orbit-closure} proves the final assertion.
\end{proof}

\begin{theorem}[Retractions fixing countably many tests]
\label{thm:local-hyperfinite-retract}
Let $Q=X/E$ be a nonsmooth $F_\sigma$ Polish-group orbit quotient,
and let $\mathcal A\subseteq\Sigma_Q$ be countable. There are a
measurable piece $H\subseteq Q$ classwise Borel isomorphic to
$Q_0=2^\N/E_0$ and unital Boolean $\sigma$-homomorphisms
\[
 r:\Sigma_Q\longrightarrow\Sigma_H,\qquad
 j:\Sigma_H\longrightarrow\Sigma_Q
\]
such that $r(A)=A\cap H$, $rj=\mathrm{id}$, and $jr$ fixes every
member of $\sigma(\mathcal A)$. Moreover $j(B)=B$ for every smooth
measurable $B\subseteq H$. No additional set-theoretic axiom is used.
\end{theorem}
\begin{proof}
Make the membership map $p:X\to2^\N$ of $\mathcal A$ continuous
by refining the Polish topology of the $G$-space. Apply
Lemma~\ref{lem:uniform-orbit-closure} in this topology. Let $J$
be the quotient image of $F$, and set $H=(L/E)\sqcup J$.
By Lemma~\ref{lem:smooth-strong-selector}, the smooth orbit piece
$L/E$ is classwise Borel isomorphic to a standard Borel space $T_L$.
The construction of $F$ identifies $J$ classwise Borel with
$(\mathcal D\times2^\N)/(\Delta_{\mathcal D}\times E_0)$.
Nonsmoothness gives $\mathcal D\ne\varnothing$. The countable relation
\[
               \Delta_{T_L}\sqcup(\Delta_{\mathcal D}\times E_0)
\]
is Borel bireducible with $E_0$. In the forward direction, code each
parameter bit infinitely often and leave infinitely many other
coordinates for the $E_0$ variable; use a repeated tag to separate the
two summands. For the equality summand, only the repeated parameter
code is needed. In the reverse direction, fix a point of $\mathcal D$
and use its $E_0$ fiber. The inverse-incidence sections of each
reduction are classes of its countable source relation, so
Lusin--Novikov gives Borel saturated images and Borel inverse lifts.
The classwise version of Lemma~\ref{lem:cb} now identifies $H$ with
$Q_0$, including lifts in both directions.

For $C\in\mathcal D$ define a probability on $\Sigma_H$ by
\[
 \lambda_C(B)=m\{a\in2^\N:[F(C,a)]_E\in B\},
\]
where $m$ is fair Bernoulli probability. Each inverse image is
$E_0$-invariant, so $\lambda_C$ is zero--one. Each singleton has
measure zero, since its inverse image is at most one $E_0$-class.
The probability therefore vanishes on every smooth measurable
piece: otherwise the measure-one sides of a countable separating
family would intersect in at most one point of measure one.
For every fixed $B$, the function $C\mapsto\lambda_C(B)$ is Borel
by parameter integration.

For $q\notin H$, its orbit closure $C_q$ is a well-defined
invariant Borel coordinate in $\mathcal D$. Define
\[
 j(B)=B\ \cup\ \{q\in Q\setminus H:\lambda_{C_q}(B)=1\}.
\]
This is quotient-measurable. The zero--one probability identities
show that $j$ preserves complements and countable unions, while
$rj=\mathrm{id}$ is immediate. If $B$ is smooth, all displayed
probabilities vanish, so $j(B)=B$ literally.

For $A\in\sigma(\mathcal A)$, membership in $A$ depends only on
$p$. The function $p$ is constant on each $C_q$ and on all
$F(C_q,a)$. Hence $\lambda_{C_q}(A\cap H)=1_A(q)$, and
$jr(A)=A$ also outside $H$, completing the proof.
\end{proof}

\section{The synchronized construction under CH}\label{sec:ch-synchronized}

The local comparison in Theorem~\ref{thm:local-hyperfinite-retract}
fixes a countable family of tests, but an arbitrary sequence of such
comparisons need not have a compatible limit. We
construct both sides together, retaining actual orbit-closure
coordinates at cofinally many stages. Lemma~\ref{lem:orbit-closure-inverse-limit}
then gives the same full residual parameter space on both sides.

For a nonempty standard Borel space $T$, write
\[
 \mathsf P(T)=(T\times2^\N)/(\Delta_T\times E_0).
\]
An embedding of this space used below always has measurable image
and Borel lifts in both directions on its image. Call such an image
a \emph{reservoir over $T$} when its first coordinate is the
specified label of the surrounding quotient. Fair Bernoulli
probability on each $E_0$ fiber then gives a Borel family of
ordinary probabilities on the presentation, whose quotient
probabilities are zero--one and vanish on smooth pieces.

\begin{lemma}[Changing the parameter of a reservoir]
\label{lem:split-reservoir}
If $r:T'\to T$ is a Borel surjection with a Borel right inverse,
then $\mathsf P(T')$ and $\mathsf P(T)$ admit an isomorphism with
Borel lifts in both directions that preserves the $T$ coordinate,
using $r$ on the first space. Moreover, every measurable
essentially hyperfinite orbit piece carrying a quotient-Borel
label in $T$ embeds with Borel lifts into $\mathsf P(T)$ over $T$.
The space $\mathsf P(T)$ contains two disjoint reservoirs over $T$.
\end{lemma}
\begin{proof}
Let $s$ be a Borel right inverse of $r$. One embedding is
$(t,[a])\mapsto(s(t),[a])$. In the other direction use
$(t',[a])\mapsto(r(t'),[e(t',a)])$, where $e$ reduces
$\Delta_{T'}\times E_0$ to $E_0$: repeat each bit of a Borel
injective code for $t'$ on infinitely many coordinates and use
the remaining coordinates for $a$. The source relations are countable, so
Lusin--Novikov gives measurable saturated images and Borel
inverse lifts. The classwise version of Lemma~\ref{lem:cb} pastes these
lifts on invariant Borel branches. Each forward and inverse branch
preserves the common $T$ coordinate, so the resulting isomorphism does
as well.

For a labeled essentially hyperfinite piece, combine its label
with a Borel reduction of its presenting relation to $E_0$.
The resulting reduction into $\Delta_T\times E_0$ is a measurable
embedding with inverse lift by Remark~\ref{rem:embedding};
equivalently apply \cite[Proposition~2.8]{deRancourtMiller} to its
inverse incidence relation, whose nonempty sections are source
orbits. Finally, fix all even bits to $0$, or all even bits to
$1$, and leave the odd bits free. These give two disjoint classwise Borel
$E_0$ copies, uniformly in $T$.
\end{proof}

\begin{theorem}[Quotient-Borel isomorphism under CH]
\label{thm:ch-quotient-isomorphism}
Work in ZFC and assume CH. For $i=0,1$, let the Polish group
$G^i$ act continuously on the Polish space $X^i$, and suppose
its orbit relation $E^i$ is nonsmooth and $F_\sigma$ in $(X^i)^2$.
Then $X^0/E^0$ and $X^1/E^1$ are isomorphic for their ordinary
quotient Borel $\sigma$-algebras. In particular each such quotient is isomorphic
as a measurable space to $2^\N/E_0$. No Borel lift of the resulting
global bijection is asserted.
\end{theorem}

We prove the theorem by separating the two countable extension steps
from the final use of CH.

\subsection{Inductive data and compatibility}\label{sec:ch-data}

Fix the actions in Theorem~\ref{thm:ch-quotient-isomorphism} and
write the quotients as $Q^0=X^0/G^0$ and $Q^1=X^1/G^1$.
Keep the original closed covers of their orbit relations fixed.
All topology refinements below preserve their Borel structures
and continuous group actions.

At a stage $\alpha<\omega_1$ we maintain the following five requirements.
\begin{enumerate}[label=(I\arabic*),leftmargin=*]
\item\label{inv:ch-cores}
The data include measurable essentially
hyperfinite cores $H^i_\alpha\subseteq Q^i$ and an isomorphism
$h_\alpha:H^0_\alpha\to H^1_\alpha$ with Borel lifts. The cores
increase and these quotient-point maps extend one another exactly.
\item\label{inv:ch-labels}
There is a common nonempty standard Borel label space $T_\alpha$
and onto quotient-Borel maps
\[
 c^i_\alpha:Q^i\setminus H^i_\alpha\longrightarrow T_\alpha.
\]
For $\alpha<\beta$ there are compositional Borel surjections
$r_{\beta\alpha}:T_\beta\to T_\alpha$ satisfying
$c^i_\alpha=r_{\beta\alpha}c^i_\beta$ on the later residuals.
Each residual contains a reservoir over its full label space.
At a successor we use the immediately preceding stage's reservoir
before enlarging its core; no previously discarded reservoir is
required to survive.

\item\label{inv:ch-kernels}
On each core we keep probabilities $\lambda^i_{\alpha,t}$,
implemented by Borel probability kernels on its presentation,
whose quotient values are zero--one and smooth-null. They match
under $h_\alpha$, meaning
$\lambda^1_{\alpha,t}(h_\alpha(B))=\lambda^0_{\alpha,t}(B)$.
In fact the kernels below are joint Borel Bernoulli pushforwards
along maps $\kappa^i_\alpha:T_\alpha\times2^\N\to X^i$,
supported on the presentation of $H^i_\alpha$; for each $t$,
$\kappa^i_\alpha(t,\cdot)$ reduces $E_0$ to $E^i$.
Define sections
\[
 j^i_\alpha(B)=B\cup
 \{q\notin H^i_\alpha:
                   \lambda^i_{\alpha,c^i_\alpha(q)}(B)=1\}.
\]
The correspondence $j^0_\alpha(B)\leftrightarrow
j^1_\alpha(h_\alpha(B))$ is an isomorphism of their full ranges.
\item\label{inv:ch-coherence}
We require these isomorphisms to extend all earlier ones, and the
associated retractions to absorb all earlier retractions in both
orders. Explicitly, for $\alpha<\beta$ and
$B\in\Sigma_{H^0_\alpha}$ put $B^0=B$, $B^1=h_\alpha(B)$.
The required identities are
\begin{align}
 j^i_\beta\bigl(j^i_\alpha(B^i)\cap H^i_\beta\bigr)
     &=j^i_\alpha(B^i) \quad(i=0,1),
       \label{eq:ch-range-absorption}\\
 h_\beta\bigl(j^0_\alpha(B)\cap H^0_\beta\bigr)
     &=j^1_\alpha(h_\alpha(B))\cap H^1_\beta.
       \label{eq:ch-core-matching}
\end{align}
The first fixes every earlier full range; the second matches
its traces on the enlarged cores. Absorption in the opposite
order follows from agreement on the nested cores.

\item\label{inv:ch-topology}
Each side also retains increasing compatible Polish action
topologies on its whole original presentation. Successor stages
alternate the active side, with both sides active cofinally below
every nonzero limit ordinal. At an active stage the new label on
that side is its actual good orbit-closure coordinate on the
\emph{entire} new residual, with closures computed in the retained
ambient Polish action space before the new core is removed. The new
residual need not itself be Polish in the restricted topology.
The locally closed locus of the preceding residual is included in
the new core. Its
new topology makes all earlier cores clopen and all earlier
label coordinates continuous on the current residual. Only
countably many requirements occur at any stage. To specify the
latter condition, equip each $T_\alpha$, once it is introduced,
with a zero-dimensional Polish topology having its given Borel
sets. Apply \cite[Corollary~5.1.6]{BeckerKechris1996} to the
inverse images, on the current residual, of the countable clopen
bases of all earlier label spaces, together with the earlier
cores and the prescribed test. All are invariant Borel subsets
of the whole presentation: a residual is invariant Borel, so its
relative Borel subsets are ambient Borel subsets. Equivalently, extend
each label by a fixed value off its residual before taking inverse
images; the residual and its complement are made clopen at the same
time. No label identity outside its residual is required. The receiving
side retains its topology at a successor. At a limit use the countable
join of all earlier topologies on each side.
\end{enumerate}

A history through a countable stage means a family of these data
satisfying the compatibility identities and topology requirements
above. At a limit stage the same requirements are imposed relative
to every earlier stage. The two extension lemmas below are proved
in ZFC; CH is used only in the final scheduling argument.

\subsection{The successor extension}

The active successor uses the construction of
Theorem~\ref{thm:local-hyperfinite-retract}, with label coordinates
retained so that both sides remain synchronized. The proof below
verifies this additional compatibility.

\begin{lemma}[Successor extension]\label{lem:ch-successor-extension}
Suppose a history as in Section~\ref{sec:ch-data} has been constructed
through $\alpha<\omega_1$. Let $i\in\{0,1\}$ be the next active
side in the prescribed schedule. For any $A\in\Sigma_{Q^i}$,
the history extends to stage $\alpha+1$ with all the
stated data and topology requirements, with $h_{\alpha+1}$ extending
$h_\alpha$, and with \eqref{eq:ch-range-absorption} and
\eqref{eq:ch-core-matching} holding relative to every stage
$\beta\leq\alpha$. Moreover,
\[
 A=j^i_{\alpha+1}(A\cap H^i_{\alpha+1}).
\]
The new common label space is the good orbit-closure space of the
active residual. Its map to $T_\alpha$ is a Borel surjection with
a Borel right inverse. Both residual label maps remain onto and
both residuals retain reservoirs over the full new label space;
the new core kernels are matching Borel Bernoulli pushforwards.
\end{lemma}
\begin{proof}
\emph{An active successor.}
By symmetry, suppress the old stage index and suppose side $0$ is active,
with prescribed test $A\in\Sigma_{Q^0}$. Refine its retained
Polish action topology on the whole presenting space as in
\ref{inv:ch-topology}, additionally making the preimage of $A$ clopen. The preimage
of the residual $Q^0\setminus H^0$ is then clopen in this retained
Polish space. Apply the uniform orbit-closure construction there,
using the restricted retained topology, rather than choosing an
unrelated Polish topology on the residual. Let $L$ be its smooth locally
closed locus and $T'$ its good-closure parameter space. The old
label $c^0$ is constant on each new closure $C$, giving a Borel
map $r:T'\to T$.

This map has a Borel right inverse. Indeed, the old outside
reservoir supplies a Borel family $\nu_t$ of zero--one quotient
probabilities, concentrated on old label $t$, that vanish on $L$.
Push them through the new orbit-closure map on the complement
of $L$, which has full measure for every $\nu_t$. Each pushforward
is a zero--one probability on the standard Borel space $T'$.
For a countable Borel separating family in $T'$, intersect its
measure-one sides. Countable additivity gives an intersection of
measure one, and separation makes that intersection a singleton.
Thus the pushforward is $\delta_{s(t)}$. Parameter integration
and the Borel inverse of the Dirac embedding make $s$ Borel.
Since the pushforward is concentrated on
$r^{-1}\{t\}$, one has $rs=\mathrm{id}$.
This also proves $T'\ne\varnothing$. The pushforward is taken
before removing the new core; the old outside probabilities
need not persist afterwards.

In each new closure choose two disjoint canonical $E_0$ copies
$J_C^0,J_C^1$, jointly Borel in $C$, by the even-bit construction
inside Lemma~\ref{lem:uniform-orbit-closure}. Set
\[
 M=L\sqcup\bigcup_{C\in T'}J_C^0,
 \qquad H'^0=H^0\sqcup M.
\]
The increment $M$ is essentially hyperfinite. On the new residual,
let $c'^0$ be the restriction of the new orbit-closure map just
constructed in the retained Polish presentation. Removing $M$
restricts the domain of this map; it does not change the ambient
space in which these closures are computed. Each label still occurs
there, witnessed by $J_C^1$. Use Bernoulli probability on $J_C^0$
for $\lambda'^0_C$, and use the second copies for the new outside
reservoir. These probabilities preserve the old label $r(C)$.
The resulting section fixes $A$, since its membership is constant
on each $C$. The calculation below also proves that it fixes
every old section event.

\emph{The receiving successor.}
Split the old reservoir on side $1$ into two disjoint reservoirs
$V^0,V^1$ over $T$. By Lemma~\ref{lem:split-reservoir}, embed $M$
with Borel lifts into $V^0$, preserving its old label; denote the
embedding by $e$. Put
\[
 H'^1=H^1\sqcup e(M),\qquad h'=h\sqcup e.
\]
Transfer $\lambda'^0_C$ through $e$ to obtain $\lambda'^1_C$.
These remain Borel probability kernels on the presentations:
the source family is a joint Borel Bernoulli pushforward on the
new copies, and $e$ has a Borel lift there.

Since $r$ has a Borel right inverse, the untouched $V^1$ is
isomorphic with Borel lifts to $\mathsf P(T')$ over $T$. On this
reservoir define $c'^1$ to be its $T'$ coordinate. On all other
points of $Q^1\setminus H'^1$ put
$c'^1(q)=s(c^1(q))$. This is quotient-Borel, factors the old label
through $r$, and has a full outside reservoir over $T'$. It is
generally an artificial label; the next active topology on side
$1$ will make it continuous.

To check full-range coherence, take $B\in\Sigma_{H^0}$ and put
$D^0=j^0(B)$, $D^1=j^1(h(B))$. For $p\in M$, preservation of
the old label and matching of the old kernels give
\[
 \mathbf1_{D^1}(e(p))
   =\lambda^1_{c^0(p)}(h(B))
   =\lambda^0_{c^0(p)}(B)=\mathbf1_{D^0}(p).
\]
Moreover $\lambda'^0_C$ is supported on $J_C^0\subseteq M$,
where $c^0=r(C)$. The transferred kernel therefore satisfies
\begin{align*}
 \lambda'^1_C(D^1\cap H'^1)
   &=\lambda'^0_C(e^{-1}(D^1)\cap M)\\
   &=\lambda'^0_C(D^0\cap M)\\
   &=\lambda^0_{r(C)}(B)
    =\lambda^1_{r(C)}(h(B)).
\end{align*}
The same support calculation gives
$\lambda'^0_C(D^0\cap H'^0)=\lambda^0_{r(C)}(B)$.
For $q$ outside the new core on side $i$, substitute $C=c'^i(q)$
and use $r(c'^i(q))=c^i(q)$: the displayed value is exactly
$\mathbf1_{D^i}(q)$. This proves
\eqref{eq:ch-range-absorption}. Matching holds on the old cores
by $h$ and on the increments by the first display, proving
\eqref{eq:ch-core-matching}. These identities extend the whole
preceding range correspondence, and transitivity extends every
earlier one. The enlarged cores give \ref{inv:ch-cores}; the right inverse and
reservoirs give \ref{inv:ch-labels}; the transferred Bernoulli kernels
give \ref{inv:ch-kernels}. The two displayed identities verify
\ref{inv:ch-coherence}, and the initial topology refinement verifies
\ref{inv:ch-topology}.

\end{proof}

\subsection{The countable-limit extension}

\begin{lemma}[Countable-limit extension]\label{lem:ch-limit-extension}
Let $\eta<\omega_1$ be a nonzero limit ordinal. Suppose the data of
Section~\ref{sec:ch-data} have been constructed for all $\alpha<\eta$,
with their compatibility and topology requirements, and with each
side active cofinally below $\eta$. They extend to stage $\eta$ so that
\[
 H^i_\eta\supseteq\bigcup_{\alpha<\eta}H^i_\alpha,
 \qquad h_\eta\supseteq\bigcup_{\alpha<\eta}h_\alpha,
 \qquad T_\eta=\varprojlim_{\alpha<\eta}(T_\alpha,r_{\beta\alpha}).
\]
The topology on each side is the countable join of its earlier
topologies. The new label maps and all bonding maps to earlier
stages are onto; both residuals have reservoirs over all of $T_\eta$.
The core isomorphism has Borel lifts, the core kernels are matching
Borel Bernoulli pushforwards, and
\eqref{eq:ch-range-absorption} and \eqref{eq:ch-core-matching}
hold for every $\alpha<\eta$.
\end{lemma}
\begin{proof}
Put
$H^i_*=\bigcup_{\alpha<\eta}H^i_\alpha$ and
$h_*=\bigcup_{\alpha<\eta}h_\alpha$. By \ref{inv:ch-cores}, these cores are measurable
essentially hyperfinite pieces, and $h_*$ has Borel lifts in both
directions: disjointify the increments along a countable cofinal
sequence and use the earlier lifts on each piece.
Let
\[
 T_*=\varprojlim_{\alpha<\eta}(T_\alpha,r_{\beta\alpha}).
\]
This is standard Borel. It is nonempty and each projection is
onto: extend a prescribed earlier coordinate along a countable
cofinal sequence of surjective bonding maps by dependent choice.
This observation also proves inductively the surjectivity of
bonding maps introduced at previous limit stages. No Borel
section of these limit projections is needed.

We verify that \emph{every} path in $T_*$ occurs on both residuals.
Using \ref{inv:ch-labels} and \ref{inv:ch-topology}, fix a side $i$
and choose a cofinal sequence $\alpha_n<\eta$ of
its active stages. A path supplies actual closed sets
$C_n\in T_{\alpha_n}$ for that side's active topologies $\tau_n$.
At each such stage the residual on which $C_n$ was constructed
is clopen in the whole presentation $X^i$. Thus $C_n$ is also
a good orbit closure in $(X^i,\tau_n)$. The spaces
$T_{\alpha_n}$ need only be subfamilies of the full good-closure
spaces in Lemma~\ref{lem:orbit-closure-inverse-limit}; the
compatible sequence will be a path in that full inverse system.
They satisfy
\[
                 \overline{C_{n+1}}^{\tau_n}=C_n.
\]
Indeed, $C_{n+1}$ lies outside the previous core, where every
earlier label coordinate has been made continuous. Its earlier
actual closure coordinate is therefore constant on $C_{n+1}$,
with the value prescribed by the compositional bonding map:
evaluate it at a point in the remaining canonical reservoir
whose new label is $C_{n+1}$, and use
$c^i_{\alpha_n}=r_{\alpha_{n+1}\alpha_n}c^i_{\alpha_{n+1}}$.
Consequently every $y\in C_{n+1}$ has $\tau_n$-orbit closure
$C_n$, so $C_{n+1}\subseteq C_n$. Conversely, invariance of
$C_{n+1}$ gives $Gy\subseteq C_{n+1}$ and hence
$C_n\subseteq\overline{C_{n+1}}^{\tau_n}$. These two inclusions
prove the displayed equality. This argument includes all
intervening artificial label stages.

Lemma~\ref{lem:orbit-closure-inverse-limit} now realizes the actual
subpath by a point $q$. It lies outside every earlier core: a
sufficiently later active closure lies in that core's complement,
which is closed in the later topology, and the fusion point lies
in that closure. Every artificial coordinate is determined by
a later actual coordinate, so $q$ realizes the full common path.
The join orbit closure of $q$ is $\bigcap_nC_n$. Conversely, that
join closure determines all actual and hence all artificial
coordinates. The resulting identification of the residual
join-closure space with $T_*$ is Borel in both directions, by
Borel coarsening and Lusin--Souslin.

The join residual is closed Polish because every old core is
clopen in a later active topology. It has no locally closed
orbits. Indeed, the whole-space join closure of an orbit in this
closed residual remains inside the residual, so local closedness
there would imply local closedness in the whole presentation.
By Lemma~\ref{lem:continuous-smooth-hulls}, it would already be
locally closed at an earlier active stage and would have been
put in that stage's core. We may therefore apply
Lemma~\ref{lem:uniform-orbit-closure} on each residual with the
same parameter space $T_*$.

Choose two disjoint canonical copies of $\mathsf P(T_*)$ in each
residual. Add the first copies to the two union cores, match them
by their common $(t,[a])$ coordinates, and extend $h_*$ by this
matching. Use matched Bernoulli kernels on these first copies
for the new inside probabilities; the second copies give the
new outside reservoirs. Label the remaining residual by its
join orbit closure, identified with its path in $T_*$.
Every new inside kernel is concentrated on one entire old path.
Explicitly, fix $\alpha<\eta$, $B^0=B\in\Sigma_{H^0_\alpha}$,
$B^1=h_\alpha(B)$, and $D^i=j^i_\alpha(B^i)$. For $t\in T_*$,
every point of the newly added copy over $t$ has old label
$t_\alpha$, so
\[
 \lambda^i_{\eta,t}(D^i\cap H^i_\eta)
       =\lambda^i_{\alpha,t_\alpha}(B^i) \quad(i=0,1).
\]
The values on the right match between the two sides. Pointwise
membership in $D^i$ also matches on the new copies because they
are paired with the same path $t$; on the union cores it matches
by earlier coherence. For a point outside $H^i_\eta$ with path
$t$, the right-hand side is exactly its $D^i$-membership value.
These observations prove \eqref{eq:ch-range-absorption} and
\eqref{eq:ch-core-matching} for every $\alpha<\eta$ simultaneously,
which is \ref{inv:ch-coherence}. The matching fresh copies and
reservoirs verify \ref{inv:ch-cores}--\ref{inv:ch-kernels}, while
the joined topology verifies \ref{inv:ch-topology}.
The new cores include fresh residual pieces; they are not assumed
to equal the old support unions.

\end{proof}

\subsection{The global isomorphism}

\begin{proof}[Proof of Theorem~\ref{thm:ch-quotient-isomorphism}]
Initialize by taking two disjoint classwise Borel $E_0$ copies on each
side, supplied by Lemma~\ref{lem:uniform-orbit-closure}. One is
the core and one is the outside reservoir. Match the cores via
$E_0$, use matching Bernoulli probabilities inside, and take a
singleton label space. The initial constant labels need not be
actual orbit-closure coordinates.

Under CH each quotient $\sigma$-algebra has at most $\aleph_1$
elements. Schedule every event on each side as the prescribed
test at an active stage of that side. One may organize successor
stages in pairs, one active step per side; both sides then have
cofinally many active stages below every nonzero countable limit.
Apply Lemma~\ref{lem:ch-successor-extension} at each successor and
Lemma~\ref{lem:ch-limit-extension} at each nonzero limit to construct
the data for all $\alpha<\omega_1$.
Every test eventually belongs to its side's section range.
In particular every singleton eventually is fixed by a
retraction. If its point were outside that retraction's core,
its restriction to the core would be empty, and its image under
the retraction would be $j(\varnothing)=\varnothing$, a
contradiction. Thus
$h=\bigcup_{\alpha<\omega_1}h_\alpha$ is an everywhere bijection
$Q^0\to Q^1$.

If $A\in\Sigma_{Q^0}$ enters a range, let $B\in\Sigma_{Q^1}$ be
its matched full-range event. Every later core map preserves
membership in these two events. Every point eventually belongs
to a later core, so $h(A)=B$ exactly. Scheduling the target
events gives the same conclusion for $h^{-1}$. This proves both
directions of ordinary quotient-Borel measurability. The
uncountable union of the representative lifts is not claimed
to be Borel.
\end{proof}

The $F_\sigma$ hypothesis refers to the displayed continuous
Polish-group presentation. It supplies the countable closed
cover used in the uniform fusion, and remains valid under the
topology refinements in the proof. The theorem also applies to
a Borel orbit relation admitting such a presentation. Removing
this hypothesis for arbitrary Borel orbit relations would
require an additional argument; none is asserted here.

\begin{corollary}[Countable Borel quotients under CH]
\label{cor:ch-countable-quotients}
In ZFC+CH, all nonsmooth countable Borel equivalence relations have
isomorphic ordinary quotient Borel spaces. In particular, the quotient
spaces of $E_0$ and $E_\infty$ are measurably isomorphic, although
$E_0<_B E_\infty$.
\end{corollary}
\begin{proof}
Let $E$ be a nonsmooth countable Borel equivalence relation on a standard
Borel space $X$. By the Feldman--Moore theorem
\cite[Theorem~2.3]{KechrisCBER}, $E$ is generated by a Borel action of a
countable discrete group $G$. Give $X$ a compatible Polish topology for
which this action is continuous
\cite[Theorem~5.2.1]{BeckerKechris1996}. For each $g\in G$, the graph of
the homeomorphism $x\mapsto g\cdot x$ is closed in $X^2$. Hence
$E=\bigcup_{g\in G}\operatorname{graph}(g)$ is $F_\sigma$, and
Theorem~\ref{thm:ch-quotient-isomorphism} applies. The change of topology
preserves the Borel sets and therefore the quotient $\sigma$-algebra.
The strict comparison $E_0<_B E_\infty$ is classical; see
\cite[Theorem~5.10]{KechrisCBER}.
\end{proof}

Kechris asks whether Borel bireducibility of countable Borel equivalence
relations is equivalent to isomorphism of their quotient Borel spaces
\cite[Problem~2.34, p.~16; Section~16.1, p.~137]{KechrisCBER}.
Corollary~\ref{cor:ch-countable-quotients} gives a negative answer under
CH to the implication from quotient Borel isomorphism to Borel
bireducibility. It does not establish an unconditional counterexample
in ZFC or independence of this particular question.

\begin{corollary}[Positive consistency for spectra]
\label{cor:ch-dixmier-positive}
In ZFC+CH, the spectra of any two separable non-type-I
$C^*$-algebras are isomorphic for their ordinary quotient Borel
structures. In particular the Dixmier assertion $\Phi$ for simple
separable non-type-I algebras satisfies
\[
 \mathrm{ZFC}+\mathrm{CH}\vdash\Phi.
\]
Consequently
\[
 \operatorname{Con}(\mathrm{ZFC})
 \quad\Longleftrightarrow\quad
 \operatorname{Con}(\mathrm{ZFC}+\Phi).
\]
\end{corollary}
\begin{proof}
Pure-state equivalence is the $F_\sigma$ orbit relation of the
norm-Polish unitary group, as in
Lemma~\ref{lem:pure-state-fsigma}. Non-type-I
separability makes irreducible-representation equivalence nonsmooth by
\cite[Theorem~2, p.~592]{Glimm1961}; Proposition~\ref{prop:gns}
transfers this to pure-state equivalence. Apply
Theorem~\ref{thm:ch-quotient-isomorphism}. For the consistency
statement, G\"odel's constructibility argument
\cite[Theorems~7--10 and final proof paragraph, pp.~223--224]{Godel1939} gives
$\operatorname{Con}(\mathrm{ZFC})\Rightarrow
\operatorname{Con}(\mathrm{ZFC}+\mathrm{CH})$; then apply the theorem.
The reverse implication is
immediate.
\end{proof}

Theorem~\ref{thm:ch-quotient-isomorphism} and
Corollary~\ref{cor:ch-dixmier-positive} prove Theorem~\ref{intro:E}. In particular,
the isomorphism $Q_B\qiso Q_A$ exists under CH despite
Theorem~\ref{intro:A}. Its composition with the point-state map $i$ from
\eqref{eq:point-state} cannot admit even a Haar-measurable
representative lift. Such a lift would agree with a Borel map on a
conull Borel subset of $K$, contradicting
Corollary~\ref{cor:positive-restrictions} and the CAR case of
Proposition~\ref{prop:AF}. Thus the distinction between the two
categories persists even after permitting completion-measurable
lifts on this measured family.

\section{A negative solution in the Sacks model}\label{sec:sacks-negative}

We now work throughout in ZFC. We construct a model in which the
quotient measurable spectra of $C_\infty=C^*(F_\infty)$ and $\CAR$
are not isomorphic. The endpoint equivalence of
Theorem~\ref{thm:equivalence} then gives a negative answer within
the class of simple unital algebras. Our argument concerns arbitrary
measurable quotient maps, without a definability hypothesis on their
graphs.

We use real-coded Polish presentations and Borel relations, interpreted
on all the reals of each model under discussion. A smooth invariant
Borel set will always come with a code for its classifier, and, when
needed, its selector. Assertions that fixed Borel maps preserve
equivalence, take values in specified Borel domains, or satisfy a
specified pointwise identity are $\Pi^1_1$ when code validity is included.
They consequently persist in forcing extensions. We will specify
separately the places where a universal assertion over new branches
or an uncountable family of old codes is used.

\subsection{An orbit-class incidence obstruction}

Let $R$ be the orbit relation of the translation action of $\Gamma$
on the Cantor group $K$ from~\eqref{eq:example}. Write $E$ for the
pure-state equivalence relation of $C_\infty$, on $X$, and $F$ for
that of $\CAR$, on $Y$. Both are Borel orbit relations of continuous
Polish-group actions, by Lemma~\ref{lem:pure-state-fsigma}, and hence
strongly idealistic. The measured obstruction proved above gives
$R\not\le_B F$. These are fixed coded assertions; the same
presentations and obstruction are available in every ZFC forcing
extension.

\begin{lemma}\label{lem:sacks-incidence}
There is a jointly Borel family $(S_t)_{t\in2^\N}$ of $E$-invariant
smooth Borel subsets of $X$ with the following property. If
$h:X/E\qiso Y/F$ is an ordinary measurable isomorphism and
$P\subseteq2^\N$ is nonempty perfect, then
\[
 \{(t,y)\in P\times Y:[y]_F\in h(S_t/E)\}
\]
is not Borel.
\end{lemma}
\begin{proof}
Put $T=\operatorname{Emb}(K,2^\N)$, the Polish space of continuous
injections, and let $R^+$ on $T\times K$ be equality in the first
coordinate times $R$. For $t\in2^\N$ set
\[
 A_t=\{(f,z):f(\gamma z)=t\text{ for some }\gamma\in\Gamma\}.
\]
This incidence is $F_\sigma$. Projection to $f$ is a classifier
on $A_t$: in each fixed $f$-fiber, $A_t$ is empty or one whole
$R^+$-class.

We first embed $R^+$ into $E$. Choose a Borel injection
$b:T\to2^\N$ and put $D=C(2^\N)\otimes_{\min}B$, with $B$ as
in Theorem~\ref{intro:A}. The point pure states $\varphi_z$ of
Proposition~\ref{prop:point-representations} give pure states
\[
          \psi_{f,z}=\operatorname{ev}_{b(f)}\otimes\varphi_z
\]
of $D$. This is an injective Borel map, and its associated
irreducible representations are equivalent exactly when
$f=g$ and $zRw$: restriction to the central copy of $C(2^\N)$
determines $b(f)$, and the remaining comparison is that for
$\varphi_z,\varphi_w$. Pullback along a fixed surjection
$C_\infty\to D$ gives a Borel reduction $a:T\times K\to X$.

For any invariant Borel $A\subseteq T\times K$, the relation
\[
                 xEa(f,z),\qquad (f,z)\in A,
\]
has countable sections in $(f,z)$, each a single $R^+$-class.
Lusin--Novikov gives a Borel saturated image and a Borel inverse
quotient lift. The parameterized version applied to $A_t$ makes
$S_t=[a(A_t)]_E$ jointly Borel; the inverse lift followed by the
classifier $f$ makes each $S_t$ smooth. Fix also a Borel function
$c(t)$ assigning a section code to $S_t$.

Suppose the displayed image incidence were Borel on $P$.
Choose a continuous injection $f:K\to P$ and set
$V_f=[a(\{f\}\times K)]_E$. This is invariant Borel, so
$h(V_f/E)$ has an invariant Borel presentation $D_f\subseteq Y$.
The Borel relation
\[
 W(z,y)\quad\Longleftrightarrow\quad
 y\in D_f\ \text{ and }\ [y]_F\in h(S_{f(z)}/E)
\]
has as its section at $z$ exactly the one $F$-class
$h([a(f,z)]_E)$. Here intersecting with $D_f$ isolates the fixed
fiber; the global sets $S_t$ need not be disjoint.
Strongly idealistic whole-class uniformization
\cite[Proposition~2.8]{deRancourtMiller} supplies a Borel
uniformizer $u:K\to Y$. Then $zRw$ if and only if $u(z)Fu(w)$,
contrary to $R\not\le_B F$.
\end{proof}

\subsection{Local Borel graphs of an arbitrary measurable isomorphism}

\begin{lemma}\label{lem:sacks-local-graphs}
Let $M$ be a transitive model of ZFC containing Borel strongly
idealistic relations $E,F$ and an ordinary measurable isomorphism
$h:X/E\qiso Y/F$. In $M$ one can choose a set-indexed family of
matched invariant smooth Borel pieces $S,D$, and Borel maps
$f_S:S\to D$, $g_S:D\to S$, with the following properties.
\begin{enumerate}[label=(\roman*),leftmargin=*]
\item Every $M$-coded invariant smooth source piece occurs as $S$,
and every such target piece occurs as $D$.
\item The maps lift $h$ and $h^{-1}$ on those pieces and satisfy
\[
 g_S(f_S(x))Ex,\qquad f_S(g_S(y))Fy.
\]
Their saturated Borel graphs
$G_S=\{(x,y):x\in S,\ f_S(x)Fy\}$ are compatible in every
forcing extension.
\item If an extension $N$ has an isomorphism $H$ preserving the
interpretations of all old paired invariant Borel event codes of
$h$, then $H$ agrees with $G_S$ on each interpreted $S$.
Thus source coverage by the old $S$ makes $\bigcup_S G_S$ the
whole graph of $H$. Target coverage by the old $D$ makes this
union target-total, whether or not such an $H$ exists in $N$.
\end{enumerate}
\end{lemma}
\begin{proof}
The image $D$ of an invariant smooth $S$ is invariant Borel.
Transferring a countable family separating the points of $S/E$
gives a Borel classifier for $F$ on $D$. Both restricted relations
remain strongly idealistic, so Lemma~\ref{lem:smooth-strong-selector}
gives Borel selectors and standard Borel realizations on both sides.
Apply the final assertion of that lemma to the restrictions of $h$
and $h^{-1}$ to obtain the claimed Borel lifts. Thus this step uses
the two strong presentation hypotheses, not just standardness of the
quotient measurable spaces.
Starting with an old target piece and applying $h^{-1}$ ensures
the coverage assertion in (i).

The inverse identities, class preservation and reflection, and
compatibility between any two fixed graphs are universal coded
Borel assertions. Their truth is absolute to extensions. For (iii),
old event preservation puts $H([x])$ in $D/F$ whenever $x\in S$.
Transfer the old countable classifier events on $D$ back through
$h$. Their coded membership identities with $f_S$ remain true,
so $H([x])=[f_S(x)]_F$. Finally, every new $y$ in an interpreted
$D$ satisfies $f_S(g_S(y))Fy$, giving target-totality under the
stated coverage hypothesis. This last assertion uses the coded
inverse identity, not an unqualified transfer of old surjectivity.
\end{proof}

\subsection{Canonization for a fork of Sacks iterations}
\label{sec:sacks-fork}

We require a canonization statement for two Sacks iterations with one
common initial Sacks real. The second iteration in the product is computed
in the model containing that initial real; it is not recomputed after the
first factor has been added. This distinction also applies to the reduced
products used below.

The goal is Corollary~\ref{cor:sacks-old-smooth-cover}: every new point
lies in a smooth piece coded before the common root.
Lemma~\ref{lem:sacks-reading} reduces names to compact forks and lifts
refinements coded in the ground model to the original forcing. The pasting and fusion
lemmas then yield Theorem~\ref{thm:sacks-canonization}, from which
the coverage follows.

\subsubsection{The geometric presentation}

Fix countable ordinals $\alpha_0,\alpha_1$, and let
\[
 I=\{r\}\sqcup(\{0\}\times\alpha_0)
              \sqcup(\{1\}\times\alpha_1).
\]
The root $r$ precedes every other coordinate. Each arm has its ordinal
order, and different arms are incomparable. Write $I_{<i}$ for the strict
predecessors of $i$, $Z_J=(2^{\N})^J$, and $\pi_J$ for coordinate
projection. The family $\mathcal D(I)$ of downward-closed subsets consists
of $\varnothing$ and the sets
\[
 A_{\gamma_0,\gamma_1}
 =\{r\}\sqcup(\{0\}\times\gamma_0)
          \sqcup(\{1\}\times\gamma_1),
 \qquad \gamma_i\leq\alpha_i.
\]
It is countable and well-founded under strict inclusion. For
$A\in\mathcal D(I)$ put $b(A)=\min(I\setminus A)$. This set has at most
two elements, except that $b(\varnothing)=\{r\}$, and every coordinate
outside $A$ lies above an element of $b(A)$.

We use pruned perfect subtrees of $2^{<\N}$ as Sacks conditions. For such
a tree $T$, let $k_T:2^{\N}\to[T]$ be the homeomorphism obtained by
successive lexicographically ordered splitting choices. The map
$(T,u)\mapsto k_T(u)$ is continuous in the tree coding topology.

The geometric presentation below is the finite-fork case of Kanovei's
projection-keeping perfect sets on countable partial orders
\cite[Section~1, pp.~6--9]{Kanovei1997}; see also
\cite[Section~1, pp.~554--557]{Kanovei1999}. We use explicit tree fields
to retain the dependence of each arm on its own predecessors.
Kanovei's fusion and continuous-reading results supply the underlying
method \cite[Theorem~14, pp.~560--561; Theorems~24--25,
pp.~566--568]{Kanovei1999}. His Proposition~26, p.~569, identifies
the generic extensions with ordinary countable-support iterations
when the index order is an ordinal, and with products when it is
unordered; its proof is omitted there. It does not state the mixed
common-root, two-arm identification with both tails frozen in the
root extension. Lemma~\ref{lem:sacks-reading} gives a full proof of
that local presentation, including the lifting of geometric refinements.

\begin{definition}\label{def:sacks-geometric}
A \emph{geometric condition} $p$ on $I$ is a system of continuous
perfect-tree fields. Its root field is a perfect tree. At coordinate $i$,
its field $T_i^p$ is a continuous function from the compact allowed
predecessor space $\pi_{I_{<i}}[p]$ to perfect trees. Recursively, its
branch set is
\[
 [p]=\{x\in Z_I:
       x_i\in[T_i^p(x\restr I_{<i})]\text{ for every }i\in I\}.
\]
An arm field uses only the root and earlier coordinates in that arm.
We order geometric conditions by inclusion of their branch sets and
retain their field codes when making finite constructions. Denote this
forcing presentation by $\mathbb P_I$.
\end{definition}

Recursion through the coordinates gives a canonical homeomorphism
$h_p:Z_I\to[p]$, determined by
\[
 h_p(u)_i=k_{T_i^p(h_p(u)\restr I_{<i})}(u_i).
\]
Indeed, each coordinate function is continuous; the inverse recovers the
successive splitting choices. Compactness then gives a homeomorphism.
For every $A\in\mathcal D(I)$ it satisfies
\begin{equation}\label{eq:sacks-prefix}
 h_p(u)\restr A=h_p(v)\restr A
       \quad\Longleftrightarrow\quad u\restr A=v\restr A.
\end{equation}
In particular, $\pi_A[p]$ is compact and its projection fibers are
nonempty compact sets. Over each root point the two arms are independent.

The passage from an uncountable iteration to a countable geometric
condition requires some care. A projection onto an arbitrary countable
set of coordinates need not be a complete subiteration. We use instead
a local presentation of a stronger condition, in which every coordinate
tree depends continuously on a countable set of selected predecessors.
The same presentation must be obtained before the common Sacks real is
added: a code produced only in the intermediate extension would not give
the old smooth sets needed below.

\begin{lemma}[Continuous reading on a countable fork]
\label{lem:sacks-reading}
Work in a model $M$ of ZFC. For an ordinal $\lambda$, let
$\mathbb S_\lambda$ denote the countable-support iteration of full Sacks
forcing of length $\lambda$. Fix ordinals $\lambda_0,\lambda_1$, and put
\[
 \mathbb Q=\mathbb S*
       \bigl(\dot{\mathbb S}_{\lambda_0}
                         \times\dot{\mathbb S}_{\lambda_1}\bigr).
\]
The two factors are computed after the first Sacks real is added, and
their product is taken in that intermediate model.
For every $p\in\mathbb Q$ and every $\mathbb Q$-name $\dot x$ for an
element of $\omega^\omega$, there are $q\leq p$, countable sets
$D_i\subseteq\lambda_i$ in $M$, a compact geometric fork condition
\[
 C\subseteq(2^\omega)^D,
 \qquad
 D=\{*\}\cup(\{0\}\times D_0)\cup(\{1\}\times D_1),
\]
and a continuous map $f:C\to\omega^\omega$, all belonging to $M$, with
the following properties. The coordinate $*$ is the common root, and
each tail coordinate has as predecessors only $*$ and the earlier
coordinates in its own tail. The condition $q$ is supported on $D$ and
is represented by the continuous perfect-tree maps defining $C$.
Writing $\dot{\boldsymbol s}$ for the vector of generic Sacks reals,
\[
 q\Vdash\dot{\boldsymbol s}\mathbin{\upharpoonright}D\in C
 \quad\hbox{and}\quad
 \dot x=f(\dot{\boldsymbol s}\mathbin{\upharpoonright}D).
\]
Moreover, every compact geometric fork refinement $C'\subseteq C$ coded in
$M$ on these coordinates lifts to a condition $q'\leq q$ in the original
forcing $\mathbb Q$, with selected branch space $C'$.
The same conclusions hold when one tail is absent, and for names
for points of any fixed coded Polish space.
\end{lemma}

\begin{proof}
All constructions in this proof take place in $M$. Write
$\operatorname{Pred}(a)$ for the full original ancestor chain of a
coordinate $a$, and $D_{<a}=D\cap\operatorname{Pred}(a)$ for its
eventual selected predecessors. These are different sets. In
particular, restricting to $D_{<a}$ is not asserted to be a complete
forcing projection.

\smallskip\noindent
\emph{1. Ground-model supports and finite-front amalgamation.}
For a perfect tree $T$, let $A_n(T)$ be the downward closure of its
first $n+1$ splitting levels and their immediate successors, and let
$L_n(T)$ be the successors at the last of these levels. These are
finite. Say that a strengthening preserves the $n$-front at $a$ if
its predecessor condition forces equality of the two $A_n$-trees.
An address selects one member of $L_n$ at each of finitely many
protected coordinates. At this point addresses are forcing
restrictions, not restrictions of an already constructed geometric set.

We will use the following elementary amalgamation. Suppose that a
condition $v$ extends a condition $u$ restricted to one such address.
There is a strengthening $u[v]\leq u$ which preserves the protected
fronts and whose restriction to that address extends $v$.
Here and below, a restriction is used only when it is a condition.
To prove this, construct the names in ancestor order. At an
unprotected coordinate use the name from $v$ on the case where all
protected predecessor addresses match, and the name from $u$
otherwise. At a protected coordinate, on that same case replace
only its selected successor subtree by the corresponding tree of
$v$, retaining the other successor subtrees and their downward
closures. Off that case leave its tree unchanged. At the root this
is just the union of the selected replacement with the retained
sibling subtrees.

The induction assertion is that the new predecessor condition,
restricted to the matching predecessor address, extends the
corresponding predecessor of $v$. Consequently the replacement
name is forced to be a perfect subtree of the old tree wherever it
is used. The retained successors preserve $A_n$, and the assertion
continues to the next coordinate. At a limit take the union of
the already defined names; its support is contained in the union
of the two countable supports. This proves the assertion at limits
as well as successors. For the fork do this on each arm over the
same new root. The address test on an arm uses only that arm's
predecessors. Although a decision may be chosen using information
about both factors, no coordinate name acquires an argument from
the other arm. Later pastings can strengthen other cells sharing
one factor address; they cannot undo a statement already forced
on those cells. This is the forcing-name form of
\cite[Lemma~6.4.10, pp.~88--89]{Schuerz2021}; the product use is
explicit in \cite[Lemma~6.4.13, pp.~89--90]{Schuerz2021}.

We also need to normalize the support of every condition pasted
in. If a Sacks name $\dot e$ is forced to be a function from
$\omega$ to the ordinals, root fusion, deciding $\dot e(k)$ on
each of the finitely many current root-front cells at step $k$,
gives a stronger root tree and finite ground sets $b_k$ with
$\dot e(k)\in b_k$ forced for all $k$. The union of the $b_k$
is a countable ground bound on its range. The construction can
start beyond any prescribed protected root front. This is also
\cite[Lemma~2.3(ii)--(iii), pp.~275--276]{BaumgartnerLaver} in
the one-step case. Apply it to enumerations of the two named
tail supports, and put full trees at unused coordinates of the
resulting ground countable bounds. The coordinate names then
belong to the appropriate full original ancestor iterations, by
associativity of two-step forcing.

Do the same for every later deciding condition \emph{before}
pasting it: inside its chosen root cell capture the union of its
two tail supports, and then amalgamate with the retained root
cells. The capture strengthens the deciding condition, so its
decision survives. Each such normalization supplies a countable
set in $M$; a finite stage uses only finitely many of them.

\smallskip\noindent
\emph{2. The schedule and the maintained invariants.}
Starting with a normalized $p_0\leq p$, construct $p_n$ and
increasing finite protected sets $F_n$ containing $*$. Whenever a
normalization supplies a new countable set, give it an enumeration.
At the end of stage $n$, let $B_n$ be the union of the bounds
supplied during that stage, with a fixed enumeration. Let $B_{-1}$
contain the support bound of $p_0$. In choosing $F_n$, include the
first $n+1$ terms of each $B_j$, $-1\leq j<n$, as well as $F_{n-1}$
and the root. Finite or empty bounds cause no difficulty. Thus
every coordinate ever used enters some $F_n$. Additional
preassigned countably many coordinates may be included in $B_{-1}$.

At the end of stage $n$ require:
\begin{enumerate}[label=(F\arabic*),leftmargin=*]
\item\label{inv:reading-support}
$p_{n+1}\leq p_n$; every intermediate condition has a ground
countable support bound, which is registered in the schedule.
Each coordinate name uses only its original ancestors.
\item\label{inv:reading-fronts}
For every $a\in F_n$, the predecessor of $p_{n+1}$ forces
$A_n(p_{n+1}(a))=A_n(p_n(a))$. The same preservation holds
through all the intermediate pastings at this stage.
\item\label{inv:reading-tables}
For each $a\in F_n\setminus\{*\}$ there is a finite table
$h_{a,n}$. Its arguments are the $n$-front addresses on
$F_n\cap\operatorname{Pred}(a)$, and the corresponding
predecessor restriction of $p_{n+1}$ forces its value to be
$A_n(p_{n+1}(a))$. The root has the constant table given by its
own finite tree. These predecessor cells cover all possibilities
below the predecessor condition.
\item\label{inv:reading-output}
Each global $n$-front cell $\sigma$ of $p_{n+1}$ forces
$\dot x\restr n=\rho_{n,\sigma}$ for a fixed
$\rho_{n,\sigma}\in\omega^n$.
\end{enumerate}

Here is the finite stage construction. Process $F_n$ in an order
extending the ancestor order. The root finite tree is already
decided. When processing $a$, the finite trees at its protected
ancestors have been decided on their own predecessor cells.
There are therefore finitely many legal ancestor addresses.
Below each such address in the \emph{ancestor iteration}, decide
the finite object $A_n(p(a))$ by an arbitrary strengthening,
normalize its support, and use the preceding amalgamation on
that ancestor iteration. Leave the name at $a$ and all later
names unchanged. The operation preserves the protected ancestor
fronts. It also preserves all protected descendant fronts,
since their names have not changed and their predecessors have
only been strengthened. Induction through the finite list gives
the tables in \ref{inv:reading-tables}.

Next list the finitely many global addresses and decide
$\dot x\restr n$ below each, using full fork extensions,
normalizing and pasting as in step~1. Preserve $A_n$ at every
coordinate of $F_n$. This preserves the just recorded tables
and all previous decisions, giving \ref{inv:reading-output}.
Later stages preserve higher fronts on larger protected sets,
and hence preserve all earlier tables and addresses.

For clarity, a coordinate introduced while stage $n$ is being
performed is \emph{not} required to have a table at that stage.
Its original forcing name is used until its first stage
$n_a=\min\{n:a\in F_n\}$. Earlier tables depend only on their
earlier finite protected sets; their identities were forced in
the full ancestor forcing, including all then unprotected
coordinates. Strengthening at a newly introduced predecessor
therefore preserves these identities. From stage $n_a$ onward,
the coordinate has tables and protected fronts at every stage.
Since $n$ tends to infinity, no omitted early table is needed to
recover its entire tree. This is the support bookkeeping used
for a chain in \cite[Lemma~3.13, pp.~14--16]{SchrittesserSacks2025};
the ancestor-table operation is its Claim~3.15, pp.~15--16.

\smallskip\noindent
\emph{3. Fusion in the original forcing.}
Let $D=\bigcup_n F_n$. The root trees fuse. Below their fusion,
each tail sequence is forced to satisfy ordinary iteration
fusion: the finite protected sets increase and exhaust the
union of its supports, and every protected splitting level is
eventually preserved. Coordinatewise intersection gives a
condition $r\leq p_n$ for every $n$, supported on $D$; at a
coordinate entering at $n_a$ use the intersection from that
stage onward. Its earlier trees contain this intersection.
Perfectness follows from the preserved, unbounded sequence of
splitting levels. Countable support gives legality at limit
coordinates. This is the usual iteration fusion lemma
\cite[Lemma~6.4.7, p.~88]{Schuerz2021}; applying it separately
to the two tails over the fused root gives the fork assertion.
In particular, for $n\geq n_a$ the predecessor cells of $r$
force the recorded $h_{a,n}$ to equal $A_n(r(a))$.

\smallskip\noindent
\emph{4. A coordinate induction, including its limit step.}
We now prove, simultaneously, the following assertions for every
constructed downward-closed selected domain $J\subseteq D$:
\begin{enumerate}[label=(G\arabic*),leftmargin=*]
\item\label{inv:reading-compact}
The constructed fields give a compact space $C_J$ and a
triangular homeomorphism $H_J:(2^\omega)^J\to C_J$, compatible
with restrictions to downward-closed subsets. Each partial
branch extends to any later constructed domain.
\item\label{inv:reading-realization}
Evaluating those fields gives a condition $q_J$ of the original
prefix forcing, with full trees at unselected coordinates,
which extends the corresponding prefix of $r$ and forces its
selected generic vector into $C_J$.
\item\label{inv:reading-cylinders}
The image under $H_J$ of every nonempty basic product cylinder
is the selected branch space of a genuine refinement of $q_J$.
Consequently every nonempty relatively open subset of $C_J$
contains such a branch space.
\end{enumerate}

At the root use its fused perfect tree. Suppose all predecessors
of $a$ have been treated, so these assertions hold for
$J=D_{<a}$. At each $n\geq n_a$, the recorded predecessor
addresses give clopen subsets of $C_J$: the splitting nodes in
each of the finitely many tests vary continuously with the
already constructed predecessors. Their union is all of $C_J$.
Otherwise its nonempty clopen complement would, by
\ref{inv:reading-cylinders}, contain a genuine condition below
the predecessor of $r$ forcing that none of its exhaustive
front addresses occurs. This contradicts
\ref{inv:reading-tables}. Write $\sigma_n(u)$ for the resulting
address of $u\in C_J$ and put
\[
                 T_a(u)=\bigcup_{n\geq n_a}
                              h_{a,n}(\sigma_n(u)).
\]

These finite trees cohere at every $u$, including nongeneric
ones. Indeed, the intersection of two address cells containing
$u$ is nonempty and relatively open. A genuine refinement
inside it forces both table identities and the corresponding
front-preservation identities for $r(a)$. Hence the two ground
finite trees have the required nesting and the same earlier
splitting levels. Any failure of that finite combinatorial
statement would be forced by this same condition, a
contradiction. Thus $T_a(u)$ is a pruned perfect tree with
$A_n(T_a(u))=h_{a,n}(\sigma_n(u))$ for every $n\geq n_a$.
For each length $k$, any stage $n\geq\max(k,n_a)$ determines
membership of all nodes of length at most $k$. The value is
constant on a finite clopen partition of $C_J$. Therefore
$u\mapsto T_a(u)$ is continuous in the tree coding topology.

Adjoin the branches of this field and define the new component
of the parametrization by
\[
 H_{J\cup\{a\}}(v)_a=k_{T_a(H_J(v\restr J))}(v_a).
\]
The resulting space is closed in $C_J\times2^\omega$, hence
compact. The map is a continuous bijection from a compact space
to a Hausdorff space, hence a homeomorphism. For forcing
realization use the name $T_a(\dot{\boldsymbol s}\restr J)$,
with any fixed perfect tree as default outside $C_J$. Under
$q_J$, every recorded finite approximation agrees with that
of $r(a)$. The union of these approximations is the whole
tree, so this name equals $r(a)$ there. This proves legality
and the extension relation at this coordinate.

Here is the cylinder assertion at every successor and limit,
without invoking the final arbitrary-lifting conclusion. A
basic cylinder fixes finitely many initial splitting choices
$v_b\restr\ell_b=w_b$. At each such $b$, restrict its
constructed field to the subtree above the successor selected
by $w_b$; at other coordinates retain the field, restricted
to the new predecessor domain. This is a continuous
perfect-tree field: a finite splitting successor depends
continuously on the predecessor, and taking the corresponding
subtree is continuous. Coordinate induction gives legal
original names and a stronger condition. Its selected branch
space is exactly the image of the specified input cylinder
under the \emph{original} $H_J$, since the restrictions were
made using the original splitting choices. This proves
\ref{inv:reading-cylinders}, independently of arbitrary lifting.

At a limit domain $J$, take all fields already constructed and
set
\[
 C_J=\{u\in(2^\omega)^J:
       u\restr J'\in C_{J'}\text{ for every earlier prefix }J'\}.
\]
This is a closed inverse limit in the compact product. More
explicitly, for every input $v\in(2^\omega)^J$, transfinite
recursion using the displayed splitting formula produces a
unique compatible output $H_J(v)$. Conversely, successive
splitting choices recover $v$ from any compatible output.
Every coordinate of $H_J$ is continuous by its already proved
predecessor construction; continuity into a product is
coordinatewise. Thus $H_J$ is again a homeomorphism. This also
proves nonemptiness, surjectivity of the prefix projections,
and extension of partial branches, rather than merely
assuming those properties of the inverse limit.

For the forcing part at a limit, take the union of the already
defined coordinate names. Its support is countable. At every
original coordinate its restriction to the predecessors is
the condition already constructed there, so it forces a
perfect subtree of the corresponding tree of $r$; these are
exactly the coordinatewise requirements for an iteration
condition extending $r$. Unselected original coordinates have
full trees. Finally, a basic open set at a limit still fixes
only finitely many splitting choices, so the cylinder
construction in the preceding paragraph applies unchanged.
There is no requirement to control infinitely many new
coordinates at a single finite stage. At an original cut of
uncountable cofinality the selected predecessors are bounded;
between selected coordinates one just inserts full trees.
Construct the two arms separately and take their fiber
product over the root. All three assertions then hold for
the fork as well, since a finite input cylinder restricts
the root once and each arm separately.

This completes the simultaneous induction and yields
$q\leq r$ and $C=C_D$ in $M$. In particular, the open-cylinder
argument used to establish pointwise coherence was available
on the predecessors \emph{before} the field at $a$ was defined.

\smallskip\noindent
\emph{5. Reading the output on all branches.}
The global recorded front cells form finite clopen partitions
of $C$. For exhaustiveness use
\ref{inv:reading-cylinders} and \ref{inv:reading-output} just
as in the predecessor argument. If cells from two stages
intersect, their intersection contains a genuine refinement
of $q$. It forces both labels to be initial segments of the
same $\dot x$. They therefore agree on their common length.
For $u\in C$ define $f(u)$ to be the union of the labels of
its cells. Their lengths tend to infinity. Each finite
initial segment is constant on a finite clopen partition,
so $f$ is continuous, and the recorded forcing identities
give $q\Vdash\dot x=f(\dot{\boldsymbol s}\restr D)$.
Replacing each $D_i$ by its inherited countable order type
gives real codes for the fields and $f$ in $M$.

\smallskip\noindent
\emph{6. Lifting arbitrary geometric refinements.}
Let $C'\subseteq C$ be a geometric fork refinement coded in
$M$, with fields $T'_a$. Every allowed predecessor of $C'$
extends to a full branch of $C'$. Since its fields depend
only on predecessors, every branch of $T'_a(u)$ does too.
The inclusion $C'\subseteq C$ consequently implies
$[T'_a(u)]\subseteq[T_a(u)]$, and hence
$T'_a(u)\subseteq T_a(u)$ because both trees are pruned.
Evaluate $T'_a$ on the selected predecessor generics at
each original selected coordinate, with a fixed default
outside its coded domain; use full trees elsewhere.
The same coordinate induction, with the same countable
limit argument as in step~4, gives a legal $q'\leq q$
with selected branch space $C'$. These names belong to $M$.
The universal Borel assertions about these fixed compact codes
(perfectness, inclusion, and the field identities) are absolute
to forcing extensions, so they also apply to new generic branches.
This assertion is about ground-coded refinements; it does
not identify arbitrary countable projections as complete
subiterations, or insert new codes into the old forcing.

\smallskip\noindent
\emph{7. Polish targets and boundary cases.}
For a point name in a nonempty coded Polish space choose
a compatible complete metric and a countable family of
open balls. Replace the output decision at step~2 by a
decision of a ball containing the named point, of diameter
at most $2^{-n}$, with closure inside the ball recorded
on the preceding cell. Such a ball exists around every
point of that preceding open ball. On every branch the
recorded closed balls are nonempty and nested with
diameters tending to zero, by the same finite-cell
argument as in step~5. Completeness gives a unique limit,
continuously in the branch, and forces it to be the named
point. Empty or finite tails simply omit coordinates
from the construction. This proves all versions stated
in the lemma.
\end{proof}

\subsubsection{Reduced squares and frontwise amalgamation}

For countable tail lengths, Lemma~\ref{lem:sacks-reading} gives
the density of the geometric presentation in the actual fork forcing:
apply it to a constant real name and retain every coordinate, giving
unselected coordinates the constant full-tree field. Thus the geometric
conditions below represent the full fork forcing, with the original
ancestor dependencies.

Projection to $A\in\mathcal D(I)$ is a complete forcing projection.
For example, if $a\leq\pi_A[p]$, replacing the fields on $A$ by those of
$a$ and restricting the remaining fields to the new predecessor spaces
gives a condition with projection exactly $a$. Define $\mathbb R_A$ to
be the reduced square of $\mathbb P_I$ over this complete projection.
We use its dense exact-projection presentation
\begin{equation}\label{eq:sacks-reduced}
 \mathbb R_A
 =\{(p_0,p_1):\pi_A[p_0]=\pi_A[p_1]\},
\end{equation}
ordered componentwise. A pair represents the compact fiber product of
$[p_0]$ and $[p_1]$ over their common $A$-projection.

More formally, this reduced square first forces with the projection and
then takes two copies of the quotient of the \emph{original} forcing.
The exact-projection form is dense: decide the two names for original
quotient conditions and strengthen their common projection. Their fields
retain their original ancestor dependencies. Thus, even when $A$ contains
prefixes from both arms, an unshared left-arm tail is not made into a full
iteration over the unrelated right-arm prefix. We have
$\mathbb R_I=\mathbb P_I$ diagonally and
$\mathbb R_{\varnothing}=\mathbb P_I\times\mathbb P_I$.

The two generic vectors for $\mathbb R_A$ agree on $A$. At every
$i\notin A$ they are forced to differ: the two unshared perfect fibers
can be restricted to incompatible finite strings. Equivalently, after the
common projection they are coordinates in independent copies of the
original quotient forcing. We will include the corresponding dense sets
in the fusion.

For finite $F\subseteq I$, $n\geq1$, and $s\in(2^n)^F$, let $p[s]$ be
the geometric condition with branch set
\[
 h_p[\{u\in Z_I:u_i\restr n=s_i\text{ for }i\in F\}].
\]
These cells form a finite compact clopen partition of $[p]$. Their
$A$-projections depend only on $s\restr(F\cap A)$. Write
$p'\leq_{F,n}p$ if $p'$ preserves the first $n$ splitting choices at each
coordinate in $F$, on every remaining predecessor point, and its cells
satisfy $[p'[s]]\subseteq[p[s]]$. The set $F$ need not be downward closed.

The following pasting operations are tree-field versions of Kanovei's
splitting-system refinement. His Lemma~12 gives one-cell refinement,
and the proof of Theorem~20 explicitly performs two successive
refinements to install a pair with a common projection
\cite[Lemma~12, p.~558; proof of Theorem~20, p.~563]{Kanovei1999}.
We record the protected-front and ancestor-dependence details for the fork.

\begin{lemma}[One-cell pasting]\label{lem:sacks-one-paste}
If $q\leq p[s]$, there is $p'\leq_{F,n}p$ with $p'[s]=q$. At coordinate
$i$ its construction tests only the selected addresses in $F\cap I_{<i}$.
\end{lemma}
\begin{proof}
Construct the new fields recursively. If the selected predecessor
addresses do not match $s$, retain $T_i^p$. If they match and $i\notin F$,
use $T_i^q$. If they match and $i\in F$, replace the selected depth-$n$
front subtree of $T_i^p$ by $T_i^q$, retaining all the other front subtrees
and their predecessors.

The induction invariant is that every remaining predecessor point which
passes the address test belongs to $\pi_{I_{<i}}[q]$. Thus $T_i^q$ is
well-defined and perfect wherever it is used. The test is clopen: it
involves finitely many splitting choices, whose nodes vary continuously
with the predecessor point. The new field is therefore continuous. The
retained siblings preserve the protected splitting choices. On the
selected cell the resulting fields are exactly those of $q$.
All tests involve ancestors of $i$ only, so no cross-arm dependency is
introduced.
\end{proof}

\begin{lemma}[Paired pasting]\label{lem:sacks-pair-paste}
Let $A\in\mathcal D(I)$, $b(A)\subseteq F$, and let
$s,t\in(2^n)^F$ satisfy
\begin{equation}\label{eq:sacks-compatible-addresses}
 s_i=t_i\quad(i\in F\cap A),\qquad
 s_i\ne t_i\quad(i\in b(A)).
\end{equation}
If $(q_0,q_1)\leq(p[s],p[t])$ in $\mathbb R_A$, then some
$p'\leq_{F,n}p$ satisfies $p'[s]\leq q_0$ and $p'[t]\leq q_1$.
\end{lemma}
\begin{proof}
Paste $q_0$ into the $s$-cell. On $A$ the two addresses agree and the
projections of $q_0,q_1$ coincide. Their pruned fiber trees there are
therefore equal, so $q_1$ survives these prefix changes. At a coordinate
in $b(A)$ the $t$-front subtree differs from the replaced subtree. At any
later coordinate outside $A$, an ancestor in $b(A)$ has already failed
the $s$-address test, so the field on the $t$-cell is unchanged. Consequently
$q_1$ is still a refinement of that cell. Paste it second. The same
argument shows that the already installed $q_0$ refinement is preserved.
If $A=I$, the cells and refinements coincide and one pasting suffices.

This also explains synchronization at the root. If it is unshared, the
two root front subtrees separate. If it is shared, its projection is
changed identically before the independent arm operations. Selecting a
pair of legal conditions using ground-model information about both arms
does not insert dependencies on the other arm's generic coordinates.
\end{proof}

\subsubsection{A fusion of reduced-product generics}

\begin{lemma}\label{lem:sacks-fusion}
Let $N$ be a countable elementary submodel of a sufficiently large
$H(\theta)$, containing $I$, the forcing presentations and
$p_0\in\mathbb P_I$. There is $c\leq p_0$ such that, for every
$x,y\in[c]$, their coordinate-agreement set $A$ is downward closed and
$(x,y)$ is $\mathbb R_A$-generic over $N$ below $(p_0,p_0)$.
\end{lemma}
\begin{proof}
Because $I$ and $\mathcal D(I)$ are countable and belong to $N$, all their
elements belong to $N$. Enumerate $I$ and all requirements $(A,D)$, where
$A\in\mathcal D(I)$ and $D\in N$ is an open dense subset of
$\mathbb R_A$ below $(p_0,p_0)$.

Construct $p_k\in N$, increasing finite $F_k\subseteq I$, and strictly
increasing positive integers $n_k$. At stage $k$, let $F_k$ contain the
previous finite sets, the first $k$ coordinates, and $b(A)$ for each of
the first $k$ requirements. Enlarge $F_k,n_k$ so that every corresponding
cell of $p_k$ has diameter at most $2^{-k}$ in a fixed compatible product
metric. Uniform continuity of $h_{p_k}$ on $Z_I$ permits this. These
finite choices are made in $N$.

Process each of the first $k$ requirements $(A,D)$. For every pair of
current cells satisfying \eqref{eq:sacks-compatible-addresses}, their
$A$-projections agree. Choose inside $N$ a strengthening in $D$ and use
Lemma~\ref{lem:sacks-pair-paste} to install both refinements. There are
finitely many such pairs. Repeating this for the finitely many requirements
gives $p_{k+1}\in N$ with $p_{k+1}\leq_{F_k,n_k}p_k$. Previously arranged
requirements persist by strengthening. Every requirement is processed
infinitely often, which is essential for pairs whose first difference
occurs arbitrarily far along a splitting front.

For $u\in Z_I$, at stage $k$ select the cell addressed by
$(u_i\restr n_k)_{i\in F_k}$. Protected fronts make these compact cells
nested, and their diameters tend to zero. Their unique common point
$h(u)$ depends continuously on $u$. Distinct input vectors eventually
lie in different protected cells, so $h$ is injective. Conversely, each
point of $\bigcap_k[p_k]$ has coherent finite addresses, and is obtained
this way. Thus $h$ is a homeomorphism of $Z_I$ onto that intersection.
Projection of a cell to a downward-closed set depends only on its
restricted address; the same nested-cell argument proves
\eqref{eq:sacks-prefix} for $h$.

These facts also verify that the intersection is a geometric condition
$c$. For each coordinate, $h_i$ depends continuously only on the input
predecessors and the coordinate itself. Over a fixed predecessor point
its image is a perfect compact fiber. As the predecessor varies, these
fiber images vary continuously in the Hausdorff topology, by uniform
continuity on the compact coordinate-input space. Their pruned trees
therefore vary continuously in the tree coding topology. These are the
fields of $c$.

Now take $x=h(u)$ and $y=h(v)$. Initially define $A$ by \emph{longest
common prefixes}, not by all coordinate equalities. If the root differs,
put $A=\varnothing$; otherwise take, in each arm, the prefix preceding
the first disagreement. Property~\eqref{eq:sacks-prefix} gives the same
prefixes for $u,v$. Their inputs differ at every point of the finite set
$b(A)$, so sufficiently long front addresses distinguish them at all
those coordinates. Consequently, every sufficiently late processing of
any fixed requirement $(A,D)$ handles their pair of cells.

Those cell pairs belong to $N$ and to $\mathbb R_A$, and form a descending
sequence. At a sufficiently late processing of $D$, the next pair is
below a condition of $D$ selected in $N$. The filter generated by these
nested cell pairs is therefore generic over $N$, and its branch vectors
are $x,y$. This proves genericity below the original $(p_0,p_0)$; the
limit $c$ need not belong to $N$. Finally, the dense sets forcing the two
reals at every coordinate outside $A$ to differ are among the requirements.
Thus there are no accidental later equalities, and the full agreement
set is precisely $A$.
\end{proof}

The minimal-downset argument below adapts Kanovei's ordinal-iteration
canonization in the proof of Lemma~28: Step~2 identifies equivalence
on each fixed parameter fiber with equality on a minimal initial
segment, and Step~6 uses two copies
of the order sharing an initial segment
\cite[pp.~572--574]{Kanovei1999}. Here downsets can be incomparable;
the reduced-generic fusion and the intersection argument below handle
this additional feature.

\begin{theorem}[Finite-fork canonization]\label{thm:sacks-canonization}
Let $E$ be a Borel equivalence relation on the branch set of a geometric
condition $B\in\mathbb P_I$. There are $c\leq B$ and
$A\in\mathcal D(I)$ such that
\begin{equation}\label{eq:sacks-canonical}
 x\mathrel E y\quad\Longleftrightarrow\quad
       x\restr A=y\restr A\qquad(x,y\in[c]).
\end{equation}
In particular, $E\restr[c]$ is smooth.
\end{theorem}
\begin{proof}
Call $A\in\mathcal D(I)$ eligible if some condition $(B_0,B_1)$ of
$\mathbb R_A$, below $(B,B)$, forces its two generic vectors to be
$E$-equivalent. The set $I$ is eligible. Choose an inclusion-minimal
eligible $A$ and a witnessing pair $(B_0,B_1)$.

The condition $(B_0,B_0)$ also forces $E$ in $\mathbb R_A$. Indeed, over
the common projected generic, take two independent quotient generics
below $B_0$ and add a third below $B_1$ over their joint extension. Each
original vector paired with the third is generic below $(B_0,B_1)$.
Transitivity gives the assertion, and Borel absoluteness brings it back
to the original pair extension. These are products of the original
quotient posets; no iterands are recomputed over unrelated coordinates.

Take a countable elementary $N$ containing $E,A,B,B_0,B_1$, the minimality
assertion and the forcing statements, and apply
Lemma~\ref{lem:sacks-fusion} below $B_0$. If $x,y\in[c]$ have
agreement set exactly $A$, the forcing theorem and Borel absoluteness give
$xEy$. If their agreement set contains $A$, choose $z\in[c]$ agreeing
with both exactly on $A$; then $xEzEy$. To choose $z$, keep the common
$A$-prefix and choose fresh values at the finitely many boundary coordinates,
using their perfect fibers. The agreement normalization from the fusion
lemma excludes accidental later equalities. The case $A=I$ is immediate.

Suppose instead that the agreement set $B'$ of $x,y$ does not contain $A$,
and set $D=A\cap B'\subsetneq A$. Choose $x',y'\in[c]$ so that
\[
 \operatorname{agree}(x,x')=A,\qquad
 \operatorname{agree}(y,y')=A,\qquad
 \operatorname{agree}(x',y')=D.
\]
For an explicit construction, retain the $A$-prefixes of $x,y$ and use
fresh, pairwise distinct values at every omitted arm boundary where
prefixes still agree. If the root is omitted, make these choices already
there. This can be performed in the input product $Z_I$ and transferred
by the prefix-preserving parametrization of $c$. Perfect fibers and the
agreement normalization give exactly the three displayed sets.
The preceding case yields $xEx'$ and $yEy'$.

By minimality of $A$, no $\mathbb R_D$-condition below $(B,B)$ forces
$E$; hence the top condition there forces its negation. Reduced genericity
of $x',y'$ and Borel absoluteness give $\neg(x'Ey')$, and therefore
$\neg(xEy)$. This proves \eqref{eq:sacks-canonical}.
\end{proof}

\subsubsection{Coverage by old smooth pieces}

The ordinal-iteration predecessor of the next corollary is Kanovei's
Lemma~28 \cite[pp.~571--574]{Kanovei1999}. In an ordinal Sacks
extension of $M\models V=L$, it states that every class of an analytic
equivalence relation meets a closed partial selector coded in $M[r]$,
where $r$ is its real parameter. The local construction of these
partial selectors does not use $V=L$; constructibility is used to
choose least codes for a global selector in Theorem~27.
We give the corresponding Borel invariant-set
formulation for the two frozen tails, with all codes in the model before
the common root is added.

\begin{corollary}\label{cor:sacks-old-smooth-cover}
Let $M$ be a transitive model of ZFC and $E$ an $M$-coded Borel equivalence
relation on an $M$-coded standard Borel space $X$. Let $M[t][G_0,G_1]$
be obtained by one Sacks real followed by two mutually generic ordinal
Sacks tails, both posets computed in $M[t]$. Every point of $X$ in this
extension belongs to the interpretation of an $M$-coded invariant Borel
set on which $E$ has an $M$-coded Borel selector.
\end{corollary}
\begin{proof}
Fix a name $\dot x$ and an arbitrary starting condition. By
Lemma~\ref{lem:sacks-reading}, after strengthening, read $\dot x$ by a
continuous $f:[q]\to X$ on a compact countable fork; use a Polish
realization of $X$. All data are in $M$. Canonize the Borel pullback
\[
       u\mathrel R v\quad\Longleftrightarrow\quad f(u)\mathrel E f(v)
\]
by Theorem~\ref{thm:sacks-canonization}. Obtain $c\leq q$ and $A$
such that $f(u)Ef(v)$ is equivalent to $\pi_A(u)=\pi_A(v)$ on $[c]$.

The continuous projection $\pi_A:[c]\to\pi_A[c]$ has nonempty compact
fibers and compact base. Choose a Borel selector $s$ for these fibers.
The map $g=f\circ s$ is Borel and injective: equality of its values implies
$E$-equivalence, hence equality of the two base points. Distinct values
of $g$ are $E$-inequivalent. Lusin--Souslin therefore makes
\[
                     H=g[\pi_A[c]]
\]
a Borel $E$-antichain coded in $M$.

The Borel relation $\{(z,h):h\in H, zEh\}$ has vertical sections of
size at most one. By Lusin--Novikov its projection $S=[H]_E$ is Borel
and its unique-valued uniformization is a Borel selector for $E\restr S$.
Every $f(u)$, $u\in[c]$, is equivalent to $g(\pi_A(u))$, so
$f[[c]]\subseteq S$. The codes for $S$ and its selector belong to $M$.
Their selector identities and this containment are universal Borel
assertions and remain true in forcing extensions. The lifting rule in
Lemma~\ref{lem:sacks-reading} now gives a genuine strengthening forcing
$\dot x\in S$. Such strengthenings are dense below every starting
condition, proving the claim.
\end{proof}

\begin{remark}\label{rem:sacks-sources}
Kanovei's projection-keeping presentation and fusion apply to arbitrary
countable partial orders \cite[Section~1 and Theorem~15]{Kanovei1997}.
The continuous-reading, paired-pasting, and minimal-initial-segment
arguments above adapt \cite[Theorems~24--25; proof of Theorem~20;
proof of Lemma~28]{Kanovei1999}. In particular, neither paired pasting
nor the minimal-segment/reduced-square method is introduced here.

The present proof gives their common-root, two-arm form. It verifies
the presentation of the original frozen-tail forcing and its refinement
lifting, spells out simultaneous reduced genericity for every pair,
and treats incomparable downsets by intersection. Kanovei's
Proposition~26 states the ordinal and unordered cases, while his
Lemma~28 is proved in the ordinal setting. Although its proof already
uses a two-copy partial order, neither result states the full
canonization and coverage assertions for two frozen tails. For the
reduced-product treatment of finite products
and ordinal iterations see also \cite[Theorem~9.3, Lemma~9.4,
Theorem~10.3, Lemma~10.27 and Claim~10.28]{KSZ}; for frontwise pasting
and product continuous reading see
\cite[Lemmas~6.4.10 and~6.4.13]{Schuerz2021}.
The argument above uses arbitrary Borel equivalence relations, with no
assumption that the CAR presentation is countable or essentially
hyperfinite. It gives no canonization claim for a countably infinite
product of independent Sacks factors.
\end{remark}

\subsection{The delayed image of a new smooth event}

\begin{lemma}\label{lem:sacks-delay}
Let $t$ be Sacks generic over $M$, put $W=M[t]$, and let $N$ be
a further forcing extension. Suppose $M$ contains an isomorphism
$h:X/E\qiso Y/F$ for the spectral presentations above, $N$
contains an isomorphism $H$ preserving all old event-code pairs
of $h$, and the $M$-coded invariant smooth source pieces cover
$X^N$. Then no Borel code in $W$ presents $H(S_t/E)$ in $N$.
\end{lemma}
\begin{proof}
Fix in $M$ the local graphs of Lemma~\ref{lem:sacks-local-graphs}.
For each fixed old $S$, let $K_S(c,d)$ assert validity and
invariance of the two event codes and
\[
    \forall(x,y)\in G_S\quad
             (x\in B_c\ \Longleftrightarrow\ y\in B_d).
\]
This is $\Pi^1_1$: a violation of the incidence assertion is an
analytic existence statement, using evaluation witnesses for
valid Borel codes.

Suppose $d\in W$ were a code as in the conclusion. Source
coverage and Lemma~\ref{lem:sacks-local-graphs} identify the
union of the old local graphs with $H$ in $N$, so every
$K_S(c(t),d)$ holds there and hence in $W$. By the forcing
theorem choose a Sacks condition $P$ in the actual generic
forcing all these assertions. This is one formula quantified
over the old set of graphs; no intersection of countably many
forcing conditions is taken. Strengthen within that generic
so that a continuous ground function $d(u)$ on $[P]$ reads
the name for $d$.

In $W$ there is a nonempty perfect set $Q\subseteq[P]$ all of
whose points are Sacks generic over $M$ and generate $W$.
Indeed, for an old homeomorphism $e:2^\N\to[P]$, take
\[
 Q=e\bigl[\{t\mathbin{\oplus}z:z\in2^\N\}\bigr],
\]
where $\oplus$ denotes interleaving. Every member is new over
$M$, since its inverse image under $e$ recovers $t$. Every new
real in a one-Sacks extension is itself Sacks generic and
generates that extension
\cite[Lemma~2.9]{SchilhanSacks}. Thus for each old $S$,
\[
                  W\models\forall u\in Q\ K_S(c(u),d(u)).
\]
This universal assertion is still $\Pi^1_1$, with the real
code of $Q$ as parameter. It transfers to $N$, including the
new branches of $Q$. The transfer is made separately for
each member of the unchanged old indexing set.

In $N$, old source coverage now implies that $d(u)$ presents
$H(S_u/E)$ for every $u\in Q$. A Borel family of valid Borel
codes has jointly Borel evaluation: membership and
nonmembership each have analytic evaluation definitions, so
Souslin's theorem applies. This contradicts
Lemma~\ref{lem:sacks-incidence}.
\end{proof}

\subsection{Descent of a Borel event through proper forcing}

The proof below adapts the pinning argument of Larson--Zapletal
\cite[Proposition~2.1.5 and Theorem~2.6.2, pp.~38 and~65]{LZ}.
Their theorem treats analytic equivalence relations; here validity and
equality of Borel codes form a $\Pi^1_1$ assertion, so we give the
descent argument explicitly.

\begin{lemma}\label{lem:sacks-pinning}
Let $\mathbb P$ be proper and let $\dot d$ be a name forced
to code a Borel subset of a fixed coded Polish space. If
$\mathbb P\times\mathbb P$ forces that the two interpretations
of $\dot d$ code the same Borel set, then there is a ground
Borel code $d_0$ such that $\mathbb P\Vdash B_{\dot d}=B_{d_0}$.
\end{lemma}
\begin{proof}
Replace $\dot d$ by a normalized real name: for each digit fix
a maximal antichain deciding it, together with the integer
labels. Given $p\in\mathbb P$, choose a countable
$N_0\prec H(\theta)$ containing $p$ and these data, and an $N_0$-master
condition $q\le p$. Let $\pi:N_0\to\overline N$ be the
transitive collapse. In the ground universe choose a filter
$H$ generic for $\overline{\mathbb P}$ over $\overline N$
containing $\pi(p)$, and set $d=\pi(\dot d)^H$.

In any ambient generic extension with $q\in G$, the filter
$J=\pi[G\cap N_0]$ is generic over $\overline N$. Moreover,
\[
                      \pi(\dot d)^J=\dot d^G.
\]
For each digit, masterhood makes $G\cap N_0$ meet the dense
set below its deciding antichain. Elementarity supplies a
member of that antichain in $G\cap N_0$ with the same integer
label before and after collapse. This proves the displayed
identity without identifying an arbitrary collapsed name
with its ambient value.

Inside this extension, choose $L$ generic for the same frozen
poset $\overline{\mathbb P}$ over both countable models
$\overline N[H]$ and $\overline N[J]$. Then $(H,L)$ and
$(J,L)$ are product-generic over $\overline N$. The hypothesis
equates each of their first-coordinate Borel sets with the
one coded by $\pi(\dot d)^L$.

Validity and equality of Borel codes together form a
$\Pi^1_1$ assertion. They transfer upward from these
transitive models: a rank function for the relevant
well-founded countable tree remains a rank function in the
ambient universe. No equality of their internal $\omega_1$
with the ambient $\omega_1$ is needed. Thus $d$ is a valid
ground code and $q\Vdash B_{\dot d}=B_d$.

We have obtained such a local pin below every $p$. Any two
local pins code the same ground Borel set, by applying the
product hypothesis to their two conditions and then Borel
absoluteness. Fix one pin $d_0$; conditions forcing
$B_{\dot d}=B_{d_0}$ are dense. This proves the result.
Properness of $\mathbb P\times\mathbb P$ was not used.
\end{proof}

\subsection{Reflection along the Sacks iteration}
\label{subsec:sacks-predicate-reflection}

Work in a ground model $V$ of $\mathsf{ZFC}+\mathsf{GCH}$, put
$\kappa=\omega_2^V$, and let
$\langle\mathbb P_\alpha,\dot{\mathbb S}_\alpha:
\alpha<\kappa\rangle$ be the countable-support iteration in which
$\dot{\mathbb S}_\alpha$ denotes the full Sacks forcing of the
$\mathbb P_\alpha$-extension.  For a $\mathbb P_\kappa$-generic filter
$G$, write $G_\alpha=G\cap\mathbb P_\alpha$ and $V_\alpha=V[G_\alpha]$.
The use of the full Sacks forcing at each stage is part of this convention.

\begin{lemma}\label{lem:sacks-reflection-cardinals}
Every $\mathbb P_\alpha$, for $\alpha\leq\kappa$, is proper and preserves
cardinals.  For every $\alpha<\kappa$,
\[
 V_\alpha\models 2^{\aleph_0}=\aleph_1,
 \qquad
 \mathbb P_\alpha\Vdash |\dot{\mathbb S}_\alpha|=\aleph_1<\kappa.
\]
Moreover, $V_\kappa\models 2^{\aleph_0}=\aleph_2$.  If
$\lambda\leq\kappa$ has uncountable cofinality, then every real of
$V_\lambda$ belongs to some $V_\alpha$ with $\alpha<\lambda$.
\end{lemma}

\begin{proof}
Properness follows from the properness of Sacks forcing and the
countable-support iteration theorem.  The normalized conditions of
Baumgartner--Laver give a dense subset of $\mathbb P_\alpha$ of cardinality at most
$\mathfrak c^V=\aleph_1$ whenever $\alpha<\omega_2^V$; see
\cite[Lemma~3.1, p.~277]{BaumgartnerLaver}.
Their finite-front fusion argument also reads a name for a real from a
countable set of coordinates and a countable table of finite decisions;
compare \cite[Lemma~2.3 and Theorem~3.3, pp.~275--277]
{BaumgartnerLaver}, or the continuous-reading formulation in
\cite[Fact~2.6, p.~5]{SchilhanSacks}.
Below any condition, the reading data can be chosen in $V$.  For fixed
$\alpha<\kappa$, the number of possible such data is at most
\[
 |[\alpha]^{\leq\aleph_0}|\cdot\mathfrak c^V
 \leq (\aleph_1^{\aleph_0})^V=\aleph_1.
\]
Evaluating these ground-model data at the actual generic sequence accounts
for every real of $V_\alpha$.  Since $\omega_1$ is preserved, this proves
$\mathsf{CH}$ at every bounded stage, including stages of countable
cofinality.  In particular, full Sacks forcing has cardinality $\aleph_1$
in each such extension.

Baumgartner--Laver prove that the iteration is $\omega_2$-cc and preserves
cardinals \cite[Theorem~3.2, p.~277]{BaumgartnerLaver}.  Their
Theorem~3.3(a) gives the assertion about reals at limits of uncountable
cofinality, and Theorem~3.3(b) gives the final value of the continuum
\cite[p.~277]{BaumgartnerLaver}.
\end{proof}

We use the reflection theorem of Farah--Shelah for an arbitrary predicate
on the reals.  If $\dot A$ is a name for a set of reals, normalize it as a
collection of pairs $(p,\dot x)$, where the name $\dot x$ uses countable
antichains below $p$.  Let $\dot A\mathbin{\upharpoonright}\xi$ consist
of those pairs for which $p\in\mathbb P_\xi$ and $\dot x$ is a
$\mathbb P_\xi$-name, and put
\[
 A_\xi=(\dot A\mathbin{\upharpoonright}\xi)^{G_\xi}.
\]
Farah--Shelah's Proposition~4.6 says that, for a countable-support iteration
of proper forcings of size less than a regular $\kappa>\mathfrak c^V$,
there is a club $C\subseteq\kappa$ such that
\begin{equation}\label{eq:sacks-predicate-elementarity}
 (H(\aleph_1)^{V_\xi},\in,A_\xi)
 \prec (H(\aleph_1)^{V_\kappa},\in,\dot A^G)
 \quad
 (\xi\in C,\ \operatorname{cf}(\xi)=\omega_1).
\end{equation}
See \cite[\S4.3 and Proposition~4.6, p.~13 of arXiv version~2]
{FarahShelah}.  Lemma~\ref{lem:sacks-reflection-cardinals} verifies the
cardinality hypotheses for the iteration fixed above.  The restriction
of the name is essential: it makes $A_\xi$ an element of $V_\xi$.

\begin{proposition}\label{prop:sacks-reflect-quotient-isomorphism}
Let $E$ on $X$ and $F$ on $Y$ be Borel equivalence relations with fixed
ground-model presentation codes, interpreted in each forcing extension.
Suppose that $V_\kappa$ contains a bimeasurable bijection
\[
 H:X^{V_\kappa}/E^{V_\kappa}\longrightarrow
       Y^{V_\kappa}/F^{V_\kappa}
\]
for the quotient Borel structures.  Then, for club many
$\xi<\kappa$ of cofinality $\omega_1$, there is in $V_\xi$ a
bimeasurable bijection
\[
 h_\xi:X^{V_\xi}/E^{V_\xi}\longrightarrow
          Y^{V_\xi}/F^{V_\xi}
\]
with the following property.  If $a,b\in V_\xi$ are Borel codes and
$V_\xi$ satisfies that $B_a\subseteq X$ and $B_b\subseteq Y$ are
invariant and
\begin{equation}\label{eq:sacks-old-event-pair}
 h_\xi(B_a/E)=B_b/F,
\end{equation}
then their interpretations in $V_\kappa$ are invariant and satisfy
\begin{equation}\label{eq:sacks-new-event-pair}
 H(B_a^{V_\kappa}/E^{V_\kappa})
   =B_b^{V_\kappa}/F^{V_\kappa}.
\end{equation}
Thus event transport is preserved on new representatives as well as old
ones.  No definability assumption is imposed on $H$.
\end{proposition}

\begin{proof}
Use fixed real codings of the presentations and of pairs of
representatives.  In $V_\kappa$ let $A$ be the saturated graph
\[
 A=\{(x,y)\in X\times Y:H([x]_E)=[y]_F\},
\]
regarded as a set of reals, and choose a name $\dot A$ for it.
Apply \eqref{eq:sacks-predicate-elementarity}, with the presentation codes
as parameters, and fix a resulting $\xi$.  Put $M=V_\xi$ and let
$A_\xi\in M$ be the interpretation of the restricted name.

The assertions that a predicate is saturated under $E\times F$, is total
in both coordinates, and induces a one-to-one correspondence between
the classes are first-order assertions over $H(\aleph_1)$ with that
predicate.  For example, functionality on classes is expressed by
\[
 \forall x\in X\ \forall y,y'\in Y\quad
 \bigl(A(x,y)\wedge A(x,y')\ \Longrightarrow\ y\,F\,y'\bigr),
\]
and injectivity is expressed by the analogous formula with $E$.
Consequently $A_\xi$ induces a quotient bijection $h_\xi$ in $M$.
Here membership in $M$ follows from $A_\xi\in M$ by Separation and
Replacement in $M$, rather than merely from an external trace of $A$.

Validity of a Borel code, membership in its coded set, and invariance of
that set under a fixed coded Borel relation are also expressible over
$H(\aleph_1)$.  For valid invariant codes $a,b$, define the event-pair
formula
\begin{equation}\label{eq:sacks-event-pair-formula}
 \forall x\in X\ \forall y\in Y\quad
 \bigl(A(x,y)\ \Longrightarrow\
          (x\in B_a\ \Longleftrightarrow\ y\in B_b)\bigr).
\end{equation}
Once $A$ is the saturated graph of a quotient bijection,
bimeasurability is expressed by saying that every valid invariant code
on $X$ has a valid invariant partner on $Y$ satisfying
\eqref{eq:sacks-event-pair-formula}, and conversely for every code on $Y$.
These are first-order assertions with quantification over real codes;
they do not require a real code for $H$ itself.  Elementarity therefore
gives bimeasurability of $h_\xi$ inside $M$.

Finally, if $a,b\in M$ satisfy \eqref{eq:sacks-old-event-pair}, then
the validity, invariance, and event-pair assertions hold in the smaller
structure.  By elementarity they hold in the larger structure with the
same parameters.  The universal quantifiers in
\eqref{eq:sacks-event-pair-formula} now range over all final-model
representatives, which proves \eqref{eq:sacks-new-event-pair}.
\end{proof}

\subsection{The negative model and independence}

\begin{theorem}\label{thm:sacks-negative}
Start with a model of ZFC+GCH and form the countable-support
iteration of full Sacks forcing of length $\omega_2$.
In the resulting model,
\[
                 Q_{C_\infty}\not\qiso Q_{\CAR}.
\]
Consequently $Q_{\mathcal U}\not\qiso Q_{\CAR}$ for the
fixed simple unital algebra $\mathcal U$ of
Theorem~\ref{thm:universal}, and $\neg\Phi$ holds.
\end{theorem}
\begin{proof}
Suppose the final model $N$ had an ordinary measurable
isomorphism $H$ for the two pure-state presentations.
Apply Proposition~\ref{prop:sacks-reflect-quotient-isomorphism}
to its saturated graph. Obtain
an intermediate stage $M$, of uncountable cofinality, with
an isomorphism $h\in M$ such that $H$ preserves the full
interpretations of all old paired invariant Borel event
codes. Let $t$ be the next Sacks real and put $W=M[t]$.

Fix the family of old local graphs from
Lemma~\ref{lem:sacks-local-graphs}. Over $W$ the remaining
tail is a proper countable-support iteration of full Sacks forcing
computed in $W$, by \cite[Theorem~2.5, pp.~276--277]{BaumgartnerLaver}
(applied after the stage adding $t$ and with the tail coordinates
reindexed by their order type). By the forcing
theorem choose a condition in the actual tail generic
forcing that the name for $H$ has the stated properties
and that a real name $\dot d$ codes $H(S_t/E)$.
The family of old event pairs is a set parameter in $W$,
so one condition can force its universal preservation.
Restrict the tail below that condition and call it
$\mathbb P$.

For each old $S$, both coordinate interpretations of
$\dot d$ in a $\mathbb P\times\mathbb P$ extension
satisfy $K_S(c(t),\dot d)$. Indeed it holds in the
corresponding coordinate extension, and is $\Pi^1_1$.
Viewed from $M$, this product extension is exactly one
common Sacks step followed by two independent copies of
the tail poset computed in $W$, restricted below the chosen condition.
Choose $M$-names for that condition and the relevant tail names, and
strengthen within the actual root generic so that the required
statements are forced. This is a restriction of the forcing
$\mathbb S*(\dot{\mathbb S}_{\lambda}\times
\dot{\mathbb S}_{\lambda})$ of Lemma~\ref{lem:sacks-reading},
where $\lambda$ is the remaining ordinal length; the second factor
is not recomputed after the first factor. Apply
Corollary~\ref{cor:sacks-old-smooth-cover} in $M$.
Every target point belongs to an $M$-coded smooth piece.
The inverse identities of
Lemma~\ref{lem:sacks-local-graphs} make the old union of
graphs target-total there. Given $y$, choose an old $S$
and $x$ with $(x,y)\in G_S$. Then
\[
 y\in B_{\dot d_0}
 \ \Longleftrightarrow\ x\in S_t
 \ \Longleftrightarrow\ y\in B_{\dot d_1}.
\]
Thus the product forces equality of the two Borel sets.
Lemma~\ref{lem:sacks-pinning} supplies a code $d\in W$
for $H(S_t/E)$ in the actual final extension $N$.

The one-arm case of
Corollary~\ref{cor:sacks-old-smooth-cover}, applied over
$M$, also covers every source point in $N$ by an
$M$-coded invariant smooth set. All hypotheses of
Lemma~\ref{lem:sacks-delay} therefore hold, and that lemma
excludes the code $d\in W$. This is the contradiction.
The last conclusions follow from
Theorem~\ref{thm:equivalence}.
\end{proof}

\begin{corollary}\label{cor:sacks-independence}
The ordinary quotient-measurable assertion $\Phi$ is
independent of ZFC, relative to the consistency of ZFC:
\[
 \operatorname{Con}(\mathrm{ZFC})
 \ \Longleftrightarrow\ \operatorname{Con}(\mathrm{ZFC}+\Phi)
 \ \Longleftrightarrow\ \operatorname{Con}(\mathrm{ZFC}+\neg\Phi).
\]
CH gives the positive answer; the length-$\omega_2$
iterated Sacks model over a GCH ground gives the negative
answer. No large cardinal is required for either direction.
\end{corollary}
\begin{proof}
For the positive direction use
Corollary~\ref{cor:ch-dixmier-positive} and the relative
consistency of CH. For the negative direction first pass
to the constructible model, where GCH holds, and apply
Theorem~\ref{thm:sacks-negative} with the fixed coded
algebras. The reverse consistency implications are
immediate.
\end{proof}

The negative assertion here is a failure of measurable
isomorphism. Choice still gives a set-theoretic bijection
between the two quotient sets. The choiceless results
below exclude even an injection for a different fixed
pair, and are stronger in that specific sense. Neither
statement changes the ZFC negative answer in the
Borel-definable category.

\section{The measured obstruction in ZF and ZF+DC}
\label{sec:choiceless-obstruction}

This section supplies the common measured input for the negative
comparisons under determinacy and in the symmetric Solovay model.
We use the fixed algebras $A=\CAR$ and
$B=C(K)\rtimes_r\Gamma$ from \eqref{eq:example}. We isolate the
exact coded consequence of the ZFC CAR obstruction and transfer it
by descriptive complexity. The first assertion below is a theorem
of ZF; its full-Borel consequence and the subsequent structural
discussion use ZF+DC. No determinacy or model-specific uniformization
principle is used in this section.

\begin{lemma}[The coded measured obstruction in ZF]
\label{lem:zf-measured-obstruction}
Use the fixed codes for $K$, the translation relation $R_\Gamma^K$,
and pure-state equivalence $F_A$ of $A=\CAR$. Let $\mathcal B$ be
the countable algebra of finite unions of finite-level clopen cylinders,
and let $m_0$ be its canonical rational-valued Haar content. For compact
$C\subseteq K$ put
\[
       m_{\mathrm c}(C)=\inf\{m_0(U):U\in\mathcal B,\ C\subseteq U\}.
\]
In ZF there is no compact $C\subseteq K$ with $m_{\mathrm c}(C)>0$
and coded Borel reduction
\[
                       R_\Gamma^K|_C\le_B F_A.
\]
Consequently, in ZF+DC there is no Borel reduction
$R_\Gamma^K|_S\le_B F_A$ on any Borel set $S\subseteq K$ of
positive Haar measure, with the full sigma-generated Borel convention.
The first assertion uses only the compact-set content; it does not
assert in bare ZF that Haar content extends to a countably additive
probability on the full Borel $\sigma$-algebra.
\end{lemma}
\begin{proof}
The existence of such a compact set and coded reduction is
$\Sigma^1_2$ in the fixed data. To see this without a quantifier
over arbitrary measurable sets, code $C$ in the compact hyperspace
of $K$. The values $m_0(U)$ are fixed rational numbers, and the tests
$C\subseteq U$ are open hyperspace conditions. Hence $m_{\mathrm c}(C)>0$
is a Borel condition on the compact-set code. Code the proposed map
by coordinate Borel partitions, as in Lemma~\ref{lem:zf-coded}.
Validity, totality on $C$, range containment in $P(A)$, and
\[
 x,y\in C\quad\Longrightarrow\quad
 \bigl(xR_\Gamma^K y\ \Longleftrightarrow\ f(x)F_Af(y)\bigr)
\]
are jointly $\Pi^1_1$ conditions on the witness codes: after valid
codes are fixed, each failure is nonemptiness of a coded Borel set.
Existential quantification over the compact-set and map codes gives
the asserted $\Sigma^1_2$ complexity.

In ZFC, $m_{\mathrm c}(C)=m(C)$ for Haar probability $m$, and
such witnesses are excluded by
Corollary~\ref{cor:positive-restrictions}, Proposition~\ref{prop:AF},
and the Borel GNS map. Their nonexistence is therefore a
$\Pi^1_2$ theorem of ZFC in the fixed codes, and is a theorem of
ZF by \cite[Theorem~5.8.1]{BlackadarFarah}. The codes are
interpreted on all reals of the universe in question.

For the last assertion, DC implies countable choice for reals.
It follows that the real-coded Borel sets form a $\sigma$-algebra:
choose codes for a given countable family and attach them to one
union node. They therefore coincide with the full sigma-generated
Borel sets. Likewise, a Borel map between the fixed standard Borel
presentations has a real code, by choosing codes for its countably
many coordinate inverse images. In ZF+DC, Haar probability on the
compact metrizable space $K$ is inner regular and satisfies
$m(C)=m_{\mathrm c}(C)$ on compact sets. Thus there is a compact
$C\subseteq S$ of positive measure. Restricting the proposed map
to $C$ gives the forbidden coded witnesses.
\end{proof}

\subsection{Structural properties and the use of choice}

The proofs of Lemma~\ref{lem:compact-action} and
Propositions~\ref{prop:point-representations} and~\ref{prop:simple}
give the simplicity, separability, unitality and non-type-I property
of $B$ in ZF+DC; the exactness and nonnuclearity assertions are
not needed here. We use the same crossed product throughout.

Lemma~\ref{lem:zf-measured-obstruction} transfers a fixed coded
nonexistence statement from the ZFC CAR rigidity proof. Thus the
negative comparisons do not require a separate choiceless proof of
the general nuclear averaging theorem, Hahn--Banach, or groupoid
transport. DC supplies countable choice for Borel coding and the
measure regularity needed to pass from positive Borel sets to compact
sets. No well-ordering of a spectrum or choice of one representative
from each equivalence class is used in this transfer.

For precision, the separable property-$(T)$ input can also be
transferred without importing arbitrary choices from its usual
proof. For a fixed coded countable group, failure of this input says
that there exist a representation on one of the fixed separable
Hilbert spaces and a sequence of almost invariant unit vectors,
but no nonzero invariant vector. The representation identities and
almost invariance are Borel conditions on the codes; absence of an
invariant vector is a universal real quantifier. Failure is therefore
$\Sigma^1_2$, and the input itself is $\Pi^1_2$. Its ZFC proof for
$\mathrm{SL}_3(\mathbb Z)$ transfers to ZF by
\cite[Theorem~5.8.1]{BlackadarFarah}. The Hilbert spaces used in
Theorem~\ref{thm:collapse} are separable under DC. This is also a
separate verification of the property-$(T)$ input in
Lemma~\ref{lem:compact-action}; the negative comparisons themselves
use Lemma~\ref{lem:zf-measured-obstruction}.

\section{A negative answer under determinacy}\label{sec:determinacy}

We work in ZF+DC with the fixed algebras $A,B$ of
Section~\ref{sec:choiceless-obstruction}. Its measured obstruction
excludes arbitrary quotient injections once a representative lift
has a Borel version on a conull set. The axiom $\AD$ states that every game of length $\omega$ on the
integers is determined. It already supplies the local uniformization
needed here. The stronger axiom $\ADR$ allows real moves.

\begin{theorem}[Determinacy]\label{thm:determinacy}
In $\mathrm{ZF}+\mathrm{DC}+\AD$, the algebras $A$ and $B$ are
simple separable unital non-type-I $C^*$-algebras, and there is no
set-theoretic injection
\[
                         Q_B\longrightarrow Q_A.
\]
Consequently their spectra, with the quotient Borel structures, are
not isomorphic as measurable spaces. No definability, measurability,
or lifting assumption on the proposed injection is required.
\end{theorem}

Throughout this section the base theory is $\mathrm{ZF}+\mathrm{DC}$.
Determinacy is used only for
Lemma~\ref{lem:det-uniformization} and its consequences. All code
spaces below are standard Borel spaces with fixed Polish presentations.

\subsection{From arbitrary quotient maps to almost-everywhere Borel lifts}

We first isolate the regularity argument used later in the Solovay
model; it does not require determinacy.

\begin{lemma}[Borel versions under measure regularity]
\label{lem:regularity-borel-version}
In ZF+DC, suppose every set of reals is Lebesgue measurable. If $X,Y$
are nonempty standard Borel spaces, $\mu$ is a Borel probability on
$X$, and $u:X\to Y$ is any function, then $u$ agrees with a Borel
function on a conull Borel subset of $X$.
\end{lemma}
\begin{proof}
The atomless probability-space isomorphism, together with the
countable atomic part, makes every subset of $X$ measurable for the
completion of $\mu$. Identify $Y$ with a Borel subset of $2^{\N}$.
For each coordinate of $u$, choose a Borel version of its inverse
image of $1$, with discrepancy contained in a Borel null set. DC
permits these countably many choices. Outside the union of the null
sets, the resulting Borel code agrees with $u$ and belongs to $Y$.
\end{proof}

\begin{lemma}[Conull uniformization from determinacy]
\label{lem:det-uniformization}
Assume $\AD$. If $X,Y$ are nonempty standard Borel spaces,
$\mu$ is a Borel probability on $X$, and $R\subseteq X\times Y$
has nonempty vertical sections, then there are a conull Borel set
$S\subseteq X$ and a Borel map $f:S\to Y$ such that
$(x,f(x))\in R$ for every $x\in S$. The relation $R$ need not be Borel.
\end{lemma}
\begin{proof}
Choose Borel injections $j:X\to2^{\N}$ and $k:Y\to\N^{\N}$ with
Borel ranges and inverses, and put $\nu=j_*\mu$. Transport
$R$ by $(j,k)$ and give every point outside $j(X)$ a fixed dummy
output. This produces a total relation
$R'\subseteq2^{\N}\times\N^{\N}$. Let $m$ be fair-coin product
measure and put $\widetilde\nu=(\nu+m)/2$. Every nonempty cylinder
has positive $\widetilde\nu$-measure, and $\nu\ll\widetilde\nu$.
Thus it suffices to uniformize $R'$ on a
$\widetilde\nu$-conull Borel set. This also avoids null cylinders in
the unfolded game of \cite[Definition~5.1, p.~14]{MeasureGame},
whose displayed upper bound on player~I's cylinder weights is stated
for cylinders of positive measure. No change to the game rules is needed.

The measure-game theorem \cite[Theorem~5.7, pp.~16--17]{MeasureGame}
gives, for each $n\geq1$, a closed set $C_n\subseteq2^{\N}$ with
$\widetilde\nu(C_n)>1-2^{-n}$ and a continuous uniformizer $f_n$ of $R'$ on
$C_n$. Its hypothesis is the determinacy of the unfolded measure
game. The moves are rational numbers, bits, and natural-number witness
digits, so $\AD$ applies. Using DC, choose such a sequence
$(C_n,f_n)$. The sets
\[
 D_n=C_n\setminus\bigcup_{1\leq m<n}C_m
\]
are disjoint and Borel. Pasting $f_n|_{D_n}$ gives a Borel uniformizer
on the $\widetilde\nu$-conull, hence $\nu$-conull, Borel set
$\bigcup_nC_n$. Restrict it to $j(X)$ and
pull it back by $j$ and $k^{-1}$. Its values there lie in $k(Y)$ by
the definition of $R'$, proving the assertion.
\end{proof}

\begin{proposition}[Transfer of a measured obstruction]\label{prop:transfer}
Assume $\AD$. Let $E$ and $F$ be Borel equivalence relations on
standard Borel spaces $X,Y$, respectively, and let $\mu$ be a Borel
probability on $X$. If $X/E$ injects into $Y/F$, then for some conull
Borel $S\subseteq X$ there is a Borel reduction $E|_S\le_B F$.
\end{proposition}
\begin{proof}
For an injection $h:X/E\to Y/F$ form the set
\[
 R_h=\{(x,y)\in X\times Y:h([x]_E)=[y]_F\}.
\]
Its sections are nonempty; we do not assert that it is Borel.
Lemma~\ref{lem:det-uniformization} gives a Borel $f$ uniformizing it
on a conull Borel $S$. For $x,x'\in S$,
\[
 x E x'\quad\Longleftrightarrow\quad
 h([x]_E)=h([x']_E)\quad\Longleftrightarrow\quad f(x)Ff(x').
\]
Thus $f$ is the required reduction.
\end{proof}

\subsection{Applying the shared obstruction}

\begin{proof}[Proof of Theorem~\ref{thm:determinacy}]
The structural properties of $B$ were proved in
Propositions~\ref{prop:point-representations} and~\ref{prop:simple}. CAR is simple, separable,
unital and infinite-dimensional by its increasing full matrix algebras;
it is therefore non-type-I, since a simple unital type-I algebra is a
full finite matrix algebra.
Suppose $h:Q_B\to Q_A$ were injective. Compose $h$ with
$[x]_{E_\Gamma^K}\mapsto[\pi_x]$ to obtain an injection
$K/E_\Gamma^K\to Q_A$. Proposition~\ref{prop:transfer}, applied to
Haar measure and the pure-state presentation of $Q_A$, gives a conull
Borel $S\subseteq K$ and a Borel reduction
$R_\Gamma^K|_S\le_B F_A$. This contradicts
Lemma~\ref{lem:zf-measured-obstruction}. In particular there is no quotient
Borel isomorphism, since any such isomorphism would be an injection.
\end{proof}

\subsection{The determinacy hypothesis and its scope}

The proof uses conull Borel uniformization, not a global choice of
representatives. The cited measure-game theorem makes ordinary
$\AD$ sufficient. DC remains explicit: it is used in the
countable pasting argument and in the measured operator-algebraic
obstruction. The proof therefore applies in particular under
$\mathrm{ZF}+\mathrm{DC}+\ADR$, and under stronger determinacy
theories implying $\AD$ with the indicated choice hypothesis.
We make no assertion here about $\mathrm{ZF}+\AD$ without DC.

Determinacy is incompatible with full Choice. For a sufficient
relative-consistency upper bound, Larson--Sargsyan--Wilson
\cite[Remark~1.2]{LSW} report Woodin's construction of an inner model
of $\mathrm{ZF}+\mathrm{DC}+\ADR$ from a proper class of Woodin
limits of Woodin cardinals. This stronger theory suffices for the
present theorem; no optimal consistency strength is claimed. The
Sacks proof of ordinary measurable independence, by contrast, has no
large-cardinal hypothesis.

If the determinacy theorem is applied in an inner model, all its
spectra and functions are interpreted there. It does not exclude
additional quotient functions in an ambient universe satisfying
Choice. In ZFC the two quotient sets have cardinality continuum;
under CH their quotient Borel structures are also isomorphic.
Thus determinacy and CH give genuinely different answers to the
measurable problem, in their respective background theories.

Projective determinacy restricts the payoffs to projective sets. It
cannot be applied to the arbitrary relation associated with a
quotient injection without an additional definability hypothesis.
In fact it is compatible with the positive solution, as follows.

\begin{proposition}[Projective determinacy and the positive answer]
\label{prop:pd-positive}
If ZFC together with projective determinacy is consistent, then so
is ZFC together with projective determinacy, CH, and $\Phi$.
\end{proposition}
\begin{proof}
Work in a model of ZFC and projective determinacy, and let $\kappa$
be the cardinality of its reals. Force with countable partial
functions $\omega_1\to\kappa$, ordered by extension. The union of a
descending sequence is a condition, so this forcing is countably
closed. It adds no reals and preserves $\omega_1$. The generic
function is onto $\kappa$, giving CH in the extension. Projective
payoffs and their truth are unchanged because the reals are unchanged;
strategies for games on the integers are real-coded as well. Every
old winning strategy therefore remains winning, and projective
determinacy persists. Theorem~\ref{intro:E} now gives $\Phi$.
\end{proof}

The proposition establishes compatibility, not an implication from
projective determinacy alone. The argument for full determinacy works
with every quotient function; the restriction to projective payoffs
does not provide that conclusion.

\section{A negative comparison in the symmetric Solovay model}
\label{sec:solovay}

A second route to the same non-injection conclusion uses measure
regularity and uniformization only for relations whose sections are
single equivalence classes. This route applies in the symmetric
Solovay model, from one inaccessible cardinal.

Let $F_A$ be pure-state equivalence on $P(A)$, as in
Section~\ref{sec:duals}. It is $F_\sigma$ by
Lemma~\ref{lem:pure-state-fsigma}. In particular it is pinned in
the ZFC ground model, by \cite[Fact~2.3.2(2)]{LZ}.

For an equivalence relation $F$ on a Polish space $Y$, write
\emph{$F$-uniformization} for the assertion that every set
$R\subseteq2^{\N}\times Y$ whose nonempty sections are $F$-classes
has a uniformizer on its projection. This is a statement about arbitrary
sets $R$, not merely coded Borel relations.

\begin{proposition}\label{prop:class-uniformization}
In ZF+DC, assume that every set of reals is Lebesgue measurable and
that $F_A$-uniformization holds. Then there is no injection $Q_B\to Q_A$.
\end{proposition}
\begin{proof}
Suppose that $h:Q_B\to Q_A$ is injective. Using the point
representations of Proposition~\ref{prop:point-representations}, define
\[
 R=\{(x,\varphi)\in K\times P(A):
               [\pi_\varphi]=h([\pi_x])\}.
\]
Each section is a nonempty $F_A$-class. Code $K$ as a Borel subset of
$2^{\N}$ and use empty sections elsewhere. Uniformization provides
$u:K\to P(A)$ with graph contained in $R$.
Lemma~\ref{lem:regularity-borel-version} gives a conull Borel $S\subseteq K$
and a Borel $f:S\to P(A)$ agreeing with $u$.
Injectivity of $h$ and Proposition~\ref{prop:point-representations} give
\[
 xE_\Gamma^K y\quad\Longleftrightarrow\quad f(x)F_Af(y)
 \qquad(x,y\in S).
\]
This contradicts Lemma~\ref{lem:zf-measured-obstruction}.
\end{proof}

\begin{theorem}[Solovay-model comparison]\label{thm:solovay}
Let $W$ be the symmetric Solovay model obtained from a ZFC ground model
with an inaccessible cardinal. Internally to $W$, the algebras $A,B$
are simple separable unital non-type-I, and there is no injection
$Q_B\to Q_A$. In particular $W\models Q_B\not\qiso Q_A$.
Consequently, consistency of ZFC with an inaccessible cardinal implies
consistency of ZF+DC with this negative quotient-Borel comparison.
\end{theorem}
\begin{proof}
The model $W$ satisfies ZF+DC and universal Lebesgue measurability
\cite[Fact 1.7.13]{LZ}. By Lemma~\ref{lem:pure-state-fsigma}, the
Larson--Zapletal pinned-uniformization theorem
\cite[Definition 10.2.1 and Theorem 10.2.2]{LZ} applies to $F_A$.
Although stated for tethered, cofinally balanced extensions of $W$,
it applies to $W$ itself: take the one-point forcing, whose tethering
condition is vacuous and whose balance is automatic in every forcing
extension. Proposition~\ref{prop:class-uniformization} now excludes
the injection. The algebraic properties were proved in
Section~\ref{sec:choiceless-obstruction} using ZF+DC.
\end{proof}

This theorem makes no assertion about arbitrary further balanced
extensions: pinned uniformization alone does not supply the measure
regularity used in the proof. The positive consistency direction is supplied by Theorem~\ref{intro:E} and,
independently, by the next section. All its spectra, equivalence
classes, and quantifiers over functions are interpreted inside $W$.

\section{Positive choiceless hypotheses and the Feferman--L\'evy model}
\label{sec:positive-consistency}

The positive construction needs less than full Choice: a wellordering
of the reals of the right length already suffices. A different
positive model will then show how the comparison can hold when even
countable choice for sets of reals fails. Throughout this section
$\Phi$ uses the full sigma-generated Borel structures and the GCR
definition of type I.

\subsection{A positive answer from a wellordering of the reals}

\begin{proposition}[ZF]
\label{prop:zf-real-wellorder}
If there is a bijection $e:\omega_1\to2^{\N}$, then $\Phi$ holds.
In particular, full Choice is not needed in addition to this explicit
wellordering hypothesis.
\end{proposition}
\begin{proof}
Put $a=\{\omega\alpha+n:e(\alpha)(n)=1\}\subseteq\omega_1$ and
$M=L[a]$. The inner model
$M$ satisfies ZFC and contains every real of the ambient universe
$V$: each $e(\alpha)$ is recovered from $a$ and the ordinal $\alpha$.
Consequently $M$ and $V$ have the same real codes for wellorders and
the same $\omega_1$. The enumeration belongs to $M$, which therefore
satisfies CH.

We must also compare the full Borel structures. In $V$, real-coded
Borel sets are closed under countable unions: given a sequence of
such sets, use the wellordering $e$ to take the least real code for
each member, and form the usual union code. They are closed under
complements and contain the open sets as well. Hence they form the
full Borel $\sigma$-algebra. Evaluation of each code agrees in $M$
and $V$, since well-foundedness and the recursive interpretation of
its countable tree agree. The same holds for the fixed coded Polish
spaces presenting the endpoint spectra.

There are no new points in these spaces, so their coded equivalence
relations have exactly the same classes in $M$ and $V$. Their quotient
Borel structures also agree: every saturated Borel set has an old
real code, and the corresponding set of classes belongs to $M$.
Apply the CH theorem inside $M$ to obtain
\[
              Q_{C^*(F_\infty)}\qiso Q_{\CAR}.
\]
This same map is a bijection in $V$ and preserves the same measurable
sets in both directions. The ZF endpoint reduction of
Corollary~\ref{cor:zf-endpoints} now yields $\Phi$ in $V$.
\end{proof}

The hypothesis retains the cardinal content of CH in the explicit
form of an $\omega_1$-enumeration. It does not answer whether the
positive solution is consistent with ZFC and $\neg\mathrm{CH}$.

\subsection{A different mechanism without countable choice}

The finite-permutation Feferman--L\'evy construction below gives a
positive answer in a model where the reals cannot be wellordered.
It establishes the universal endpoint comparison, including
\begin{equation}\label{eq:positive-target}
                         Q_B\qiso Q_A.
\end{equation}
The full Borel structures in this model are power sets, so the
remaining task is to construct a comparison of quotient sets without
selecting a representative of every equivalence class.

\subsection{A precise reduction when every subset of the reals is Borel}

\begin{proposition}[ZF]\label{prop:full-borel}
Suppose $\mathbb R=\bigcup_{n\in\N}C_n$, with each $C_n$ countable.
Let $X$ be a $T_1$ space admitting an injection into $\mathbb R$.
Then the smallest $\sigma$-algebra containing its open sets is
$\mathcal P(X)$. For every equivalence relation $E$ on $X$, the
quotient $\sigma$-algebra is $\mathcal P(X/E)$.
\end{proposition}
\begin{proof}
Fix an injection $e:X\to\mathbb R$. For $S\subseteq X$, put
$S_n=S\cap e^{-1}(C_n)$. For each $n$, the set $S_n$ is countable,
so it is Borel, being a countable union of closed singletons.
The sequence $(S_n)_{n\in\N}$ is defined without choosing
enumerations, and $S=\bigcup_nS_n$ is Borel by closure of the
$\sigma$-algebra. This argument does not choose enumerations for
all the $S_n$ simultaneously. Every subset of $X/E$ now has Borel
inverse image under the quotient map.
\end{proof}

In the finite-permutation Feferman--L\'evy model, the hypothesis
holds by \cite[Lemma~3.4 and online Appendix]{Miller}.
The fixed coordinate presentations of
pure states and irreducible representations inject into $\mathbb R$.
Thus, in such a model and with the sigma-generated convention,
\begin{equation}\label{eq:fl-reduction}
 Q_B\qiso Q_A
 \quad\Longleftrightarrow\quad
 \text{there exists a set-theoretic bijection }Q_B\longrightarrow Q_A.
\end{equation}
Both implications are internal to the same model.

\begin{proposition}[A quotient-cardinality obstruction in ZF]
\label{prop:cucc-partition}
Under the hypothesis of Proposition~\ref{prop:full-borel},
$\omega_1$ does not inject into $\mathbb R$. Consequently the
principle that every quotient of $\mathbb R$ injects into
$\mathbb R$ is false. It is false even when restricted to
real-coded Borel equivalence relations.
\end{proposition}
\begin{proof}
The first assertion is Specker's theorem; see
\cite[Proposition 4.12]{Halbeisen}. A fixed coding bijection
$\mathbb R^\N\cong\mathbb R$ gives a countable cover
$\mathbb R^\N=\bigcup_n F_n$ with each $F_n$ countable.
For each $n$ separately,
$G_n=\bigcup_{f\in F_n}\operatorname{ran}(f)$ is countable.
If $j:\omega_1\to\mathbb R$ were injective, put
\[
 \alpha_n=\min\{\alpha<\omega_1:j(\alpha)\notin G_n\},
 \qquad h(n)=j(\alpha_n).
\]
Then $h\notin F_n$ for every $n$, a contradiction. Real codes
for countable well-orders give a surjection
$\mathbb R\to\omega_1$, so the asserted principle already
fails for a quotient of $\mathbb R$.

For the Borel restriction, there is no injection
$i:[\mathbb R]^{\le\aleph_0}\to\mathbb R$. Otherwise recursion
on $\omega_1$ defines
\[
              j(\alpha)=i(\{j(\beta):\beta<\alpha\}).
\]
Each predecessor image is countable. At a first repetition
$j(\alpha)=j(\beta)$, with $\beta<\alpha$, injectivity of $i$
would equate two predecessor images, although $j(\beta)$
belongs only to the later one. Hence $j$ is injective,
contradicting the first assertion. Equality of ranges on
$\mathbb R^\N$ is a coded Borel relation whose quotient is
the family of nonempty countable subsets of $\mathbb R$.
Adding the empty set does not affect whether that family
injects into $\mathbb R$, proving the last assertion.
\end{proof}

The equality-of-ranges example does not settle the spectrum
comparison. In particular, failure of this general quotient
principle cannot replace a comparison between the two specific
quotient point sets in \eqref{eq:fl-reduction}.

\begin{proposition}[The finite-permutation Feferman--L\'evy model]
\label{prop:fl-e0-cardinality}
In the Feferman--L\'evy model $M$ defined in the online appendix
of \cite{Miller}, there is an injection
\[
                    \omega_1^M\longrightarrow 2^\N/E_0.
\]
Consequently $2^\N/E_0$ does not inject into $\mathbb R$ in
$M$. The same noninjection conclusion holds for either spectrum
in \eqref{eq:positive-target}.
\end{proposition}
\begin{proof}
Use the precise symmetric system in \cite[Appendix, pp.~1--2]{Miller}.
Over $L$, put $\kappa=\aleph_\omega^L$ and collapse each
$\aleph_n^L$ with finite partial functions from $\N$ to
$\aleph_n^L$. Conditions have finite support in $\N\times\N$.
The automorphism group $H$ consists of finite-support permutations
of $\N\times\N$ preserving the first coordinate. The normal
filter is generated by the subgroups fixing all coordinates in
the first $n$ rows. Denote the generic function in row $n$ by
$g_n:\N\to\aleph_n^L$.

For $\alpha<\kappa$, let $n(\alpha)$ be least with
$\alpha<\aleph_{n(\alpha)}^L$, and define
\[
 x_\alpha(k)=1\quad\Longleftrightarrow\quad
                         g_{n(\alpha)}(k)=\alpha.
\]
Each $x_\alpha$ belongs to $M$. A finite-support coordinate
permutation changes it at only finitely many places. To check
that the entire sequence of classes belongs to $M$, take the
following name for its $\alpha$th class:
\[
 \dot C_\alpha=
 \{(\mathbf 1,\sigma\cdot(\dot x_\alpha\mathbin\triangle s)):
                  \sigma\in H,\ s\in[\N]^{<\omega}\}.
\]
Here a binary real is identified with its set of $1$-coordinates.
Every member name is hereditarily symmetric, supported on the
first $n(\alpha)+1$ rows. Reindexing $\sigma$ shows that
$\dot C_\alpha$ itself is fixed by every element of $H$.
It evaluates to the full $E_0$-class of $x_\alpha$.
The sequence of these names, indexed by the ground-model
ordinal $\kappa$, is therefore hereditarily symmetric.

For distinct $\alpha,\beta$, below any condition and beyond
any prescribed integer one can assign a fresh coordinate so
that $x_\alpha$ and $x_\beta$ disagree. If the rows coincide,
give that coordinate value $\alpha$; if they differ, give the
first row value $\alpha$ and the second a value different from
$\beta$. Thus the forcing makes these reals $E_0$-inequivalent.
We have constructed an injection $\kappa\to2^\N/E_0$ in $M$.

For clarity, $\kappa=\omega_1^M$. Every ordinal below $\kappa$
is countable in a sufficiently large finite-stage extension,
which is contained in $M$. Conversely, any real coding a
countable well-order in $M$ lies in a finite-stage extension,
by \cite[Theorem 9.1]{Miller}; such a stage does not collapse
$\kappa$ to be countable. Hence $\kappa$ is uncountable in $M$.
The noninjection into $\mathbb R$ now follows from
Proposition~\ref{prop:cucc-partition}. Finally, the fixed coded
$E_0$ lower bounds for the two non-GCR spectra give injections
of $2^\N/E_0$ into them also in $M$; their absoluteness is
the same coded argument as in Corollary~\ref{cor:zf-endpoints}.
\end{proof}

This proposition uses finite-support permutations, not arbitrary
permutations within a row. It is a result about that specified
symmetric model, not about every model with a countable cover
of $\mathbb R$ by countable sets. The noninjection conclusion
also holds if the symmetry group is enlarged to all permutations
within rows, keeping the initial-row support filter. Indeed, the
resulting symmetric model is a submodel of $M$ containing every
finite-stage extension. It therefore has the same reals as $M$,
and hence the same $E_0$ quotient; an injection in the submodel
would remain one in $M$. This argument transfers nonexistence
of an injection, not the particular $\omega_1$-indexed witness.
The proposition does not separate
$Q_B$ from $Q_A$: both contain the displayed long family of
classes in $M$. The next construction therefore compares the quotients directly,
without first injecting them into $\mathbb R$.

Neither the countable cover nor the ZFC cardinality-continuum
argument supplies the required bijection. The next subsection
instead compares the quotients through a family of tuple orbits
whose entire sequence belongs to the symmetric model.

\begin{remark}[Real codes are a different convention]
There cannot be a surjection $\mathbb R\to\mathcal P(\mathbb R)$
in ZF, by Cantor's diagonal argument. Since real Borel codes form a
subset of $\mathbb R$, not every subset of $\mathbb R$ can have such
a code. Hence in the models of Proposition~\ref{prop:full-borel} the
family of coded Borel sets is strictly smaller than the sigma-generated
Borel family; it need not be closed under arbitrary countable unions
of its members. Formula~\eqref{eq:fl-reduction} concerns the latter
family. It is not a result about real-coded Borel maps.
\end{remark}

\subsection{A universal quotient in the Feferman--L\'evy model}
\label{sec:fl-universal}

\begin{theorem}\label{thm:fl-universal-quotient}
Let $M$ be the finite-permutation Feferman--L\'evy model of
Proposition~\ref{prop:fl-e0-cardinality}, constructed over $L$.
For every Borel equivalence relation $E$ on $2^\N$ with a code
in $L$, there is in $M$ an injection
\[
                    (2^\N)^M/E\longrightarrow (2^\N)^M/E_0.
\]
The relations are interpreted from their codes on the reals of $M$.
The same conclusion holds for a ground-coded Borel equivalence
relation on a ground-coded standard Borel space.
\end{theorem}
\begin{proof}
Let $P$ denote the collapse described in
Proposition~\ref{prop:fl-e0-cardinality}, let $G$ be generic,
and put
\[
 P_n=\{p\in P:\operatorname{dom}(p)\subseteq n\times\N\},
 \qquad G_n=G\cap P_n.
\]
Use the symmetry group $H$ from that proposition. This group is
countable and locally finite in $L$: finitely many of its
elements move only finitely many coordinates in total, and
generate a subgroup of the finite permutation group on those
coordinates. Fix a ground enumeration $e:\N\to H$ and an
increasing sequence of finite subgroups with union $H$.

\emph{A sequence of whole-tuple orbits.}
In $L$ let $\mathcal T$ be the set of all nice $P_n$-names for
subsets of $\N$, for all $n<\omega$. It is a set closed under
$H$, and its names evaluate to every real of $M$, by finite-stage
localization. Choose in $L$ an enumeration
$(\tau_\alpha)_{\alpha<\theta}$ of $\mathcal T$ by some ordinal.
No sequence of chosen evaluations is asserted to belong to $M$.

For each $\alpha$, form the tuple name
\[
             \dot V_\alpha(h)=h\tau_\alpha\qquad(h\in H).
\]
If $\tau_\alpha$ uses the first $n$ rows, so do all its
$H$-translates. The tuple $V_\alpha=\dot V_\alpha^G$, coded
as a real using $e$, belongs to $L[G_n]\subseteq M$.
The same finite row bound supports every translated tuple name.
Put
\[
       \dot T_\alpha=
                  \{(\mathbf 1,\sigma\dot V_\alpha):\sigma\in H\}.
\]
This name is literally fixed by $H$, by reindexing $\sigma$,
and all its member names are hereditarily symmetric.
Consequently the entire sequence $(T_\alpha)_{\alpha<\theta}$
belongs to $M$.

For $v\in(2^\N)^H$ define
$(S_gv)(h)=v(g^{-1}h)$. Associativity of the action on names gives
\[
       (\sigma\dot V_\alpha)^G(h)=V_\alpha(\sigma h).
\]
Thus $T_\alpha$ is precisely the orbit of $V_\alpha$ under
the shifts $S_g$. Every tuple in $T_\alpha$ has the same range.
These ranges cover the reals of $M$: a name for any such real
occurs among the $\tau_\alpha$, and its value is the
identity-coordinate of the associated tuple. All symmetry
statements here concern names; no action on the transitive
universe $M$ is being assumed.

\emph{A Borel space with a marked class.}
For a tuple $v$ and $h\in H$ let
\[
 m(v,h)=\min\{j\in\N:v(e(j))\ E\ v(h)\},
 \qquad z(v,h)=(v,m(v,h)).
\]
In $L$ consider the standard Borel space
\[
 Z_E=\{(v,i)\in(2^\N)^H\times\N:
                   \forall j<i\;\neg(v(e(j))\ E\ v(e(i)))\}.
\]
A point of $Z_E$ records a tuple and one $E$-class occurring in
its range, represented by its least coordinate index. Define
a Borel action of $H$ on $Z_E$ by
\[
              g\cdot(v,i)=z(S_gv,g e(i)).
\]
This is an action because the underlying tuple shifts compose
and the least-index normalization retains the same class.
More precisely, the marked value after applying $g$ is
$E$-equivalent to the previous marked value, since
$(S_gv)(g e(i))=v(e(i))$. Every orbit therefore preserves
the actual $E$-class of the marked value, even when the tuple
changes.

Let $R_H$ be this orbit relation. It is hyperfinite, being the
increasing union of the finite Borel orbit relations of the
chosen finite subgroups of $H$. The hyperfinite reduction
theorem \cite[Theorem~7.1]{DJK1994}, together with the elementary
smooth case, gives, in $L$, a coded Borel reduction
$b:Z_E\to2^\N$ such that
\begin{equation}\label{eq:marked-reduction}
                 b(a)\ E_0\ b(a')
                    \quad\Longleftrightarrow\quad a\ R_H\ a'.
\end{equation}
Fix a code by coordinate inverse images. Totality, the action
identities, and \eqref{eq:marked-reduction} hold also on new
tuples in $L[G]$: failure would give a point in a ground-coded
Borel counterexample set which was empty in $L$. Nonemptiness
of a coded Borel set is absolute between these models.
The same codes have the same values on common points of
$M$ and $L[G]$.

Two consequences will be used. First, for arbitrary tuples,
\begin{equation}\label{eq:marked-label}
 b(z(v,h))\ E_0\ b(z(w,k))\quad\Longrightarrow\quad
                              v(h)\ E\ w(k),
\end{equation}
because marked orbits preserve the represented class.
Second, if $v,w$ belong to the same shift orbit and
$v(h)\ E\ w(k)$, then $z(v,h)\ R_H\ z(w,k)$.
Indeed, shift $v$ to $w$ while retaining its class label and
then use the unique least-index representative of that label.

\emph{An internal definition without choosing tuples.}
Work now inside $M$. For $x\in2^\N$ put
\[
 \alpha(x)=\min\{\alpha<\theta:
                    \exists v\in T_\alpha\ \exists h\in H\;
                                      x\ E\ v(h)\}.
\]
This exists by the covering property and is $E$-invariant.
For an index $\alpha$ meeting $[x]_E$, define the set
\[
 D_\alpha(x)=\{a\in2^\N:
     \exists v\in T_\alpha\ \exists h\in H\;
       (x\ E\ v(h)\ \land\ a\ E_0\ b(z(v,h)))\}.
\]
All witnessing marked points lie in one $R_H$-class, by the
second consequence above. Hence $D_\alpha(x)$ is exactly one
$E_0$-class, not a union of several classes. Define
\[
                          u(x)=D_{\alpha(x)}(x).
\]
Separation and Replacement inside $M$ construct this function
from the sequence $(T_\alpha)$ and the ground codes. No choice
of tuples, coordinates, or real representatives is needed.

If $x\ E\ y$, then $\alpha(x)=\alpha(y)$ and their definitions
give $u(x)=u(y)$. Conversely, if $u(x)=u(y)$, choose witnesses
for the two nonempty defining sets in this individual argument.
Their $b$-images are $E_0$-equivalent, so
\eqref{eq:marked-label} gives $x\ E\ y$. Thus
\[
                         u(x)=u(y)\quad\Longleftrightarrow\quad x\ E\ y.
\]
Descending $u$ to classes gives the required injection in $M$.

For the last assertion, use a ground-coded Borel isomorphism
of the domain with a Borel subset of $2^\N$, and extend its
relation by equality on the complement. The coded identities
remain valid in $M$. Restrict the resulting quotient injection
back to the original domain.
\end{proof}

\begin{corollary}[Positive consistency over ZF]
\label{cor:zf-positive}
Let $\Phi$ be the global quotient-Borel Dixmier assertion of
Theorem~\ref{thm:equivalence}, with the GCR convention in ZF.
For the full sigma-generated quotient Borel structures,
\[
       \operatorname{Con}(\mathrm{ZF})
       \quad\Longleftrightarrow\quad
       \operatorname{Con}(\mathrm{ZF}+\Phi).
\]
In particular, the finite-permutation Feferman--L\'evy model
satisfies $Q_B\qiso Q_{\CAR}$ as well as the global assertion.
Together with Corollary~\ref{cor:sacks-independence} and the
equiconsistency of ZF and ZFC, this gives independence of $\Phi$
over ZF relative to $\operatorname{Con}(\mathrm{ZF})$. The Solovay
model of Theorem~\ref{thm:solovay} gives the stronger non-injection
conclusion for the fixed pair $B,\CAR$, from one inaccessible cardinal.
\end{corollary}
\begin{proof}
The pure-state relations of the fixed algebras $\CAR$ and
$C_\infty=C^*(F_\infty)$ have codes in $L$. Interpret these
codes in the model $M$ above. The theorem gives an injection
$Q_{C_\infty}\to2^\N/E_0$, and the fixed coded $E_0$ lower
bound for CAR gives an injection $2^\N/E_0\to Q_{\CAR}$.
The coded upper embedding gives the reverse injection
$Q_{\CAR}\to Q_{C_\infty}$. Their validity transfers as in
Corollary~\ref{cor:zf-endpoints}.

Cantor--Bernstein in ZF gives a bijection of the two endpoint
quotients. By Proposition~\ref{prop:full-borel}, their quotient
$\sigma$-algebras in $M$ are full power sets, so this bijection
and its inverse are measurable. Corollary~\ref{cor:zf-endpoints}
now gives $\Phi$ in $M$, including algebras with new real codes.
The latter use is through the ZF endpoint theorem, not an
unstated extension of Theorem~\ref{thm:fl-universal-quotient}
to all such codes.

The symmetric-model construction requires no large cardinals:
start with the constructible ground model of ZFC and use the
usual forcing interpretation to obtain the stated relative
consistency from $\operatorname{Con}(\mathrm{ZF})$.
The converse is immediate. This is a formal relative-consistency
argument, not an assertion that mere consistency supplies a
transitive model. For the opposite consistency implication use
the ZFC model of Corollary~\ref{cor:sacks-independence}. The stronger
choiceless non-injection result remains Theorem~\ref{thm:solovay}.
\end{proof}

The construction takes place inside the specified symmetric
model and does not adjoin a bijection afterwards. Its positive
conclusion does not give a positive answer in ZFC, nor an
isomorphism for the convention of real-coded Borel sets.
The countable-cover model fails DC. The positive ZFC conclusion
comes instead from the independent CH construction of
Theorem~\ref{thm:ch-quotient-isomorphism}.

\subsection{Choice principles and their exact role}
\label{sec:choice-principles}

The choice-equivalence catalogues distinguish several assertions that
must not be conflated in this problem. Rubin--Rubin list transversals,
uniformization of arbitrary relations, and right inverses of surjections
as AC~2, AC~4, and AC~5, respectively
\cite[p.~5]{RubinRubin1963}; the same labels occur in
\cite[p.~7]{RubinRubin1985}. These are equivalent to full choice when
quantified over arbitrary sets. None includes a Borel condition.
For a single quotient $q_E:X\to X/E$, a right inverse chooses a real
representative from every class. This is stronger than assigning one
\emph{whole} class of another relation to each $E$-class, as in the
quotient sections constructed here.

For cardinal comparisons, Howard--Rubin record the following principles:
\begin{itemize}
\item The \emph{partition principle}, Form~101, says that a partition
of a set injects into that set. The injection need not select a member
of each block.
\item The \emph{weak partition principle}, in the equivalent
Form~100(B), says that if there are both a surjection $Y\to X$ and
an injection $Y\to X$, then $X$ and $Y$ are equipotent.
\item \emph{Dual Cantor--Bernstein}, Form~168, says that mutual
surjections imply equipotence.
\end{itemize}
See \cite[pp.~41, 48]{HowardRubin1998}. These are cardinal statements,
not measurable versions of Cantor--Bernstein. In particular the
partition principle must not be silently replaced by the transversal
principle AC~2. The historical footnote in the 1970 printing of
\cite[p.~5]{RubinRubin1963} explicitly leaves their equivalence open;
we do not use it as a statement of current literature status.

\begin{proposition}[Local failures of choice in the positive model]
\label{prop:choice-failures}
The Feferman--L\'evy model of Corollary~\ref{cor:zf-positive} satisfies
$\Phi$ but has all of the following properties.
\begin{enumerate}
\item The weak partition principle and dual Cantor--Bernstein fail
already for the pair $2^\N$ and $2^\N/E_0$.
\item The quotient map $2^\N\to2^\N/E_0$ has no right inverse.
There is not even a function assigning to each $E_0$-class a
nonempty finite subset of that class.
\item Countable choice for sets of reals, Howard--Rubin Form~94,
fails. Consequently choice for partitions of the reals, Form~203,
and real-indexed choice for sets of reals, Form~212, both fail.
\end{enumerate}
More generally, under the countable-cover hypothesis of
Proposition~\ref{prop:full-borel}, the principles in (1) fail for
the coded Borel relation of equality of ranges of real sequences.
\end{proposition}
\begin{proof}
Put $Q_0=2^\N/E_0$. There is a fixed injection $2^\N\to Q_0$ in ZF:
repeat each input bit on an infinite block of coordinates, so that
distinct inputs give sequences differing infinitely often. The
quotient map is a surjection $2^\N\to Q_0$. The injection also gives
a surjection $Q_0\to2^\N$, by inverting it on its image and assigning
a fixed real elsewhere. Either indicated principle would therefore
give $Q_0\cong2^\N$ as sets. This contradicts
Proposition~\ref{prop:fl-e0-cardinality}.

A right inverse of the quotient map would itself inject $Q_0$ into
$2^\N$, again impossible. If a nonempty finite subset of each class
were chosen, taking its least element in the lexicographic order
would give such a right inverse. This last step is a ZF argument:
every nonempty finite linearly ordered set has a least element.
Thus even finite multiple choice restricted to this particular
family fails. Choosing the whole countable class, by contrast, is
automatic and makes no choice of a representative.

If countable choice for sets of reals held, apply it to the nonempty
sets of real codes for enumerations of $C_n\cup\{0\}$, where
$\mathbb R=\bigcup_n C_n$ is the countable cover. The selected
enumerations would enumerate $\mathbb R$ by $\N^2$, a contradiction.
Form~212 implies this countable-choice principle. Form~203 would
already select representatives from the $E_0$ partition.
The relevant source statements are
\cite[pp.~40, 51]{HowardRubin1998}.

For the final assertion, equality of ranges has quotient
$[\mathbb R]^{\leq\aleph_0}\setminus\{\varnothing\}$, with a
surjection from $\mathbb R$ using a fixed coding of real sequences
and an injection of $\mathbb R$ given by singletons. An inverse on
the singleton image, extended constantly, gives the reverse
surjection. Either cardinal principle contradicts the
noninjection proved in Proposition~\ref{prop:cucc-partition}.
\end{proof}

In particular, neither representative choice nor these cardinal
comparison principles are consequences of the positive Dixmier
assertion. Unrestricted multiple choice is Howard--Rubin Form~67
\cite[p.~35]{HowardRubin1998}. For a family of subsets of the reals,
its finite choices can always be converted into individual choices
using the fixed linear order, as above. Also, a family \emph{of
countable sets} need not be a \emph{countable family}; the
$E_0$ partition illustrates the distinction. The bounded finite
choice forms AC~7--AC~9 in \cite[p.~5]{RubinRubin1963} must likewise
not be confused with the wellordered family of countable tuple
orbits used in the positive construction.
There is also a background-theory distinction: unrestricted multiple
choice is equivalent to AC over ZF. Its weaker behavior in theories
with atoms cannot be transferred to our ZF question. Rubin--Rubin
explicitly distinguish these backgrounds in their treatment of
MC~1--MC~3 and Theorems~5.3--5.7
\cite[pp.~73--77, 83]{RubinRubin1985}.

\begin{proposition}[Real-indexed choice and the negative answer]
\label{prop:real-choice-negative}
Over ZF, choice for partitions of $2^\N$ (Form~203) implies
real-indexed choice for sets of reals (Form~212). In
$\mathrm{ZF}+\mathrm{DC}$, Form~212 together with Lebesgue
measurability of every set of reals implies that there is no
set-theoretic injection $Q_B\to Q_{\CAR}$.
Real determinacy supplies Form~212; it does not supply Form~203
in the determinacy setting used above.
\end{proposition}
\begin{proof}
For a total relation $T\subseteq\mathbb R\times\mathbb R$, the
nonempty sets $\{x\}\times T_x$ partition $T$. Code pairs by reals
and add singleton blocks outside the coded relation. A choice for
this partition selects $(x,y_x)$ in every block, giving the
uniformizer required by Form~212. This coding argument is in ZF.

Under Form~212, the incidence relation associated to a proposed
quotient injection has a real-valued uniformizer: for each pure-state
code $x$, choose a pure-state code $y$ in its assigned CAR class.
Coding the standard Borel domains in the reals and filling
unused inputs by a constant makes this a direct instance of
Form~212. Pull this raw lift back along the fixed Borel pure-state
map from $K$ used in Theorem~\ref{thm:determinacy}. Universal Lebesgue
measurability transfers to Haar measure on $K$; together with DC it
then gives a Borel lift on a Haar-conull Borel subset of $K$, exactly
as in Lemma~\ref{lem:regularity-borel-version}. The proof of
Theorem~\ref{thm:determinacy} now contradicts
Lemma~\ref{lem:zf-measured-obstruction}.

For $\ADR$, code a total relation in $\mathbb R\times\mathbb R$ and
let player I first play $x$ and player II then play $y$, winning
exactly when $(x,y)$ belongs to the relation; later moves are
irrelevant. Nonempty sections preclude a winning strategy for I.
The first response of a winning strategy for II is a global
uniformizer. Thus $\ADR$ proves Form~212. This stronger global
conclusion was not needed in the proof from ordinary $\AD$. In the same determinacy setting,
Form~203 would give an $E_0$ transversal. Every finite-bit-flip
translate of that transversal would have the same Bernoulli
measure, these countably many translates would be disjoint, and
they would cover $2^\N$. Countable additivity rules out both
possible cases, zero measure or positive measure. Universal
Lebesgue measurability transfers to Bernoulli measure, so this
contradiction applies.
\end{proof}

The distinction is coherence on equivalence classes. Form~212
chooses an output separately for each real input; it need not make
the same choice at equivalent inputs. It can supply a raw lift of
an already specified quotient function without producing a
transversal. Measurability of the lift comes here from the
additional regularity hypothesis, not from choice alone.
The positive countable-cover model also refutes the countable
union theorem, Form~31 \cite[p.~29]{HowardRubin1998}, and, with
our full sigma-generated Borel convention, the statement of
Form~273 that some set of reals is not Borel
\cite[p.~56]{HowardRubin1998}. Neither observation concerns
real-coded Borel sets.

\section{The two formulations and the remaining questions}
\label{sec:questions}

Theorem~\ref{intro:G} solves the ordinary quotient-measurable
formulation of Dixmier's question by proving independence from ZFC,
relative only to its consistency. CH gives a positive solution and the
iterated Sacks model gives a negative solution. Theorem~\ref{intro:A}
separately gives a negative Borel-definable answer in ZFC, already
within the simple exact class. Theorem~\ref{intro:F} additionally studies stronger
non-injection statements without Choice and a positive model in which
the full Borel structure is larger than its real-coded part.

Theorems~\ref{thm:determinacy} and~\ref{thm:solovay} give stronger
negative answers without Choice: the fixed exact crossed-product
quotient does not even inject as a set into the CAR quotient. Ordinary
$\AD$ with DC already suffices. Positive choiceless examples come from
the real-wellordering principle of
Proposition~\ref{prop:zf-real-wellorder} and from the specific
Feferman--L\'evy construction. Projective determinacy is compatible
with the positive CH answer, by Proposition~\ref{prop:pd-positive};
it is not being identified with full determinacy.

The role of CH in the positive direction remains a separate question.
\begin{question}\label{q:measurable-zfc}
Is $\mathrm{ZFC}+\Phi+\neg\mathrm{CH}$ relatively consistent,
or does the ordinary measurable assertion $\Phi$ imply CH over ZFC?
\end{question}
The Sacks negative model does not answer this question. The construction
under CH uses a wellordering of all quotient Borel events whose proper
initial segments are countable. A positive construction beyond CH
would require an alternative to its countable-limit fusion and
countable topology refinements.

\begin{question}\label{q:orbit-extension}
Under CH, are all nonsmooth Borel Polish-group orbit quotients
isomorphic as measurable spaces? More generally, which Borel
equivalence relations have quotient measurably isomorphic to
$2^{\N}/E_0$ under CH?
\end{question}
Theorem~\ref{thm:ch-quotient-isomorphism} applies when there is a
continuous Polish-group presentation in which the relation is
$F_\sigma$. It does not assert that every Borel orbit relation admits
such a presentation. The closed cover is used both in the uniform
avoidance construction and in the continuity of locally closed orbit
loci, so dropping it is a mathematical extension rather than a change
of notation.

In the Borel-definable category, the nuclear result and the universal
endpoint locate the examples in the chain
\[
 \widehat{\mathcal O_2}\isoBD\widehat{\CAR}
   \sdef\widehat B\bdef\widehat{\mathcal U}
   \isoBD\widehat{C^*(F_\infty)}.
\]
The first strict comparison is proved here; the second comparison is
not asserted to be strict. The endpoints $B$ and $\mathcal U$ are
both simple, while $B$ is exact and $\mathcal U$ is nonexact.
Thus the simple-algebra requirement introduces no further obstacle
to stating the exact-to-nonexact part of Farah's Problem~9
\cite{Farah2011}.

\begin{question}\label{q:further-degrees}
\begin{enumerate}[label=(\roman*),leftmargin=*]
\item Is $\widehat B\sdef\widehat{\mathcal U}$?
\item Does every simple exact nonnuclear separable $C^*$-algebra have
a Borel-definable spectrum strictly above the CAR degree?
\item Are there infinitely many, or uncountably many, Borel-definable
degrees among simple exact nonnuclear separable $C^*$-algebras?
\end{enumerate}
\end{question}
Theorem~\ref{intro:A} answers the nuclear-to-exact comparison existentially, not
the universal assertion in part (ii). The rigidity theorem excludes
nuclear targets for one broad family of measured sources, but does
not characterize all countable relations reducible to a nuclear
spectrum. In particular it makes no claim about arbitrary
nonamenable measured relations. Further separations would require
additional obstructions that distinguish nonnuclear targets.

The group comparison also leaves the representation-universality
questions asked by Thomas \cite[Question~1.9, the discussion after
Theorem~1.10, and Remark~5.9]{Thomas2015}: whether universality passes
from a subgroup to the ambient countable group, whether it does so
for finite-index subgroups, and whether every countable nonamenable
group is representation universal. The point-representation family
used here witnesses the obstruction in $F_\infty$ through a quotient
of its full group algebra; it does not resolve these permanence
questions.

\appendix
\section{Idealistic equivalence relations}\label{app:idealistic}

We work in ZFC throughout this appendix. We relate the parameterized
definition in the corrections to Kechris--Macdonald
\cite{KechrisMacdonald2016,KechrisMacdonaldCorrections} to the definition
with an auxiliary map in \cite[Definition~3.1]{BLP2024}, after adding
external parameters. We also record the precise statements about
changes of presentation and inverse lifts used in this paper.
The parameterized definitions describe the same class of relations
as the strongly idealistic Borel equivalence relations of
\cite[Introduction]{deRancourtMiller}.

All spaces below are standard Borel spaces, equipped with compatible
Polish topologies when needed. A classwise Borel isomorphism is a
bijection of quotient sets with Borel lifts in both directions.
A classwise Borel embedding is such an isomorphism onto a quotient
with an invariant Borel carrier.

\subsection{Parameterized section tests}

For a Borel equivalence relation $E$ on $X$, assign to each class $C$
a proper sigma-ideal $\mathcal I_C$ of subsets of $C$.
The corrected Kechris--Macdonald condition is that, for every standard
Borel space $Z$ and Borel set $B\subseteq Z\times X$,
\begin{equation}\label{eq:app-class-ideal}
 \{(z,x)\in Z\times X:B_z\cap[x]_E\in\mathcal I_{[x]_E}\}
\end{equation}
is Borel. Here $B_z=\{u:(z,u)\in B\}$, and proper means
$C\notin\mathcal I_C$.

\begin{proposition}\label{prop:app-parameterized}
Let $E$ be a Borel equivalence relation on $X$, and fix one assignment
of proper sigma-filters $C\mapsto\mathcal F_C$ on its classes.
The following conditions are equivalent.
\begin{enumerate}[label=(\roman*),leftmargin=*]
\item There is a Borel map $\zeta:X\to X$ with $xE\zeta(x)$ such
that, for every standard Borel space $Z$ and Borel
$A\subseteq(Z\times X)\times X$, the set
\begin{equation}\label{eq:app-zeta-test}
 \{(z,x):A_{(z,\zeta(x))}\cap[x]_E\in\mathcal F_{[x]_E}\}
\end{equation}
is Borel.
\item The same condition holds with $\zeta=\id_X$.
\item For every standard Borel space $Z$ and Borel
$B\subseteq Z\times X$, the set
\begin{equation}\label{eq:app-class-filter}
 \{(z,x):B_z\cap[x]_E\in\mathcal F_{[x]_E}\}
\end{equation}
is Borel.
\end{enumerate}
The same class filters witness all three conditions; in (i), a single
map $\zeta$ works for every $Z$ and $A$. Existence of an assignment
satisfying these conditions is equivalent to $E$ being strongly
idealistic.
\end{proposition}

\begin{proof}
For (i)$\Rightarrow$(iii), given $B\subseteq Z\times X$, apply (i) to
\[
 A=\{((z,v),u):(z,u)\in B\}.
\]
This test ignores the coordinate containing $\zeta(x)$, and
\eqref{eq:app-zeta-test} becomes \eqref{eq:app-class-filter}.

For (iii)$\Rightarrow$(ii), apply (iii) with parameter space
$Z\times X$ to the given Borel set
$A\subseteq(Z\times X)\times X$. The set
\[
 T_A=\{((z,v),x):A_{(z,v)}\cap[x]_E\in\mathcal F_{[x]_E}\}
\]
is Borel. Its pullback under $(z,x)\mapsto((z,x),x)$ is the test
in (ii). This uses (iii) for the larger parameter space $Z\times X$;
no section indexed by $x$ is replaced by one indexed by an
$E$-equivalent point. The implication (ii)$\Rightarrow$(i) uses the
identity. In fact, pulling $T_A$ back under
$(z,x)\mapsto((z,\zeta(x)),x)$ proves the test for every Borel
class-preserving $\zeta$, with these same filters.

Dualize the filters by setting
\[
 \mathcal I_C=\{D\subseteq C:C\setminus D\in\mathcal F_C\}.
\]
The sigma-filter axioms make $\mathcal I_C$ a sigma-ideal, and
$\varnothing\notin\mathcal F_C$ gives $C\notin\mathcal I_C$.
Applying (iii) to $(Z\times X)\setminus B$ gives
\eqref{eq:app-class-ideal}, since
$C\setminus(B_z\cap C)=C\setminus B_z$.
Conversely, class ideals satisfying
\eqref{eq:app-class-ideal} give filters satisfying (iii) by the same
complement operation.

It remains to compare the class-ideal formulation with the definition
in Section~\ref{sec:duals}. Extend each class ideal to $X$ by
\[
 \mathcal J_x=\{D\subseteq X:D\cap[x]_E\in\mathcal I_{[x]_E}\}.
\]
For Borel $R\subseteq(X\times Z)\times X$, use $X\times Z$ as
the parameter space in \eqref{eq:app-class-ideal} and pull the test
back under $(x,z)\mapsto((x,z),x)$. This proves that
$x\mapsto\mathcal J_x$ witnesses that $E$ is strongly idealistic.
Conversely, restrict a witnessing assignment of ideals on $X$ to
the corresponding classes. For Borel $B\subseteq Z\times X$,
apply its parameterized test to the Borel set
\[
 \{((x,z),u):(z,u)\in B\text{ and }xEu\}.
\]
After exchanging coordinates, this is
\eqref{eq:app-class-ideal}. Borelness of $E$ justifies the displayed
test.
\end{proof}

Thus full parameterization makes the auxiliary map dispensable.
This equivalence does not decide the unparameterized question
discussed before \cite[Definition~3.1]{BLP2024}: the proof uses
an additional coordinate and a diagonal pullback.

\subsection{Changes of presentation and inverse lifts}

\begin{proposition}\label{prop:app-transport}
Suppose $E$ on $X$ and $F$ on $Y$ are Borel equivalence relations
and $\phi:X/E\to Y/F$ is a classwise Borel isomorphism.
Then $E$ is strongly idealistic if and only if $F$ is strongly
idealistic. The same property is inherited by classwise Borel
embeddings into a strongly idealistic relation.
\end{proposition}

\begin{proof}
Let $f:X\to Y$ and $g:Y\to X$ be Borel lifts of $\phi$ and
$\phi^{-1}$. Suppose $E$ has class ideals satisfying
\eqref{eq:app-class-ideal}. For an $F$-class $D$, write
$C=\phi^{-1}(D)$ and put
\[
 \mathcal K_D=
 \{A\subseteq D:C\cap f^{-1}(A)\in\mathcal I_C\}.
\]
This is a sigma-ideal. It is proper because
$C\cap f^{-1}(D)=C\notin\mathcal I_C$.
Both $C$ and $\mathcal I_C$ depend only on $D$, so this assignment
is independent of a representative of the target class.

For Borel $B\subseteq Z\times Y$, let
$B^f=\{(z,u):(z,f(u))\in B\}$. Since
$\phi^{-1}([y]_F)=[g(y)]_E$, we have
\[
 B_z\cap[y]_F\in\mathcal K_{[y]_F}
 \iff (B^f)_z\cap[g(y)]_E\in\mathcal I_{[g(y)]_E}.
\]
The right-hand side is the pullback of a Borel set under the Borel
map $(z,y)\mapsto(z,g(y))$. This proves
\eqref{eq:app-class-ideal} for $F$. Interchanging the two relations
gives the converse.

For a classwise Borel embedding, restrict the target to the invariant
Borel carrier of the image and apply the isomorphism case.
The parameterized condition passes to an invariant Borel carrier:
extend a Borel test to the ambient space, apply
\eqref{eq:app-class-ideal}, and restrict back.
\end{proof}

Neither lift in this argument needs to be injective or surjective
on the presenting spaces. Only the induced maps on classes are
required to be inverse. This is the parameterized counterpart of
\cite[Lemma~3.6]{BLP2024}.

\begin{corollary}[Images and inverse lifts]\label{cor:app-inverse-lifts}
Let $E$ on $X$ be strongly idealistic, let $F$ on $Y$ be Borel,
and let $f:X\to Y$ be a Borel reduction. For every invariant Borel
$D\subseteq X$, the saturated image $S_D=[f(D)]_F$ is Borel,
$F\restr S_D$ is strongly idealistic, and there is a Borel map
$g_D:S_D\to D$ satisfying
\[
 f(g_D(y))\mathrel F y\quad(y\in S_D),\qquad
 g_D(f(x))\mathrel E x\quad(x\in D).
\]
In particular, a Borel-definable bijection from a strongly idealistic
source to a Borel target has a Borel-definable inverse, and the target
relation is strongly idealistic.
\end{corollary}

\begin{proof}
Choose class ideals satisfying \eqref{eq:app-class-ideal}, using
Proposition~\ref{prop:app-parameterized}.
Consider the Borel set
\[
 P_D=\{(y,x)\in Y\times X:x\in D\text{ and }y\mathrel F f(x)\}.
\]
Each nonempty section is an entire $E$-class: the reduction property
identifies the inverse image of $[f(x)]_F$ with $[x]_E$, and
invariance of $D$ retains this whole class.
The uniformization theorem
\cite[Proposition~2.8]{deRancourtMiller} therefore gives a Borel
projection $S_D$ and a Borel uniformizer $g_D$. Its graph lies in
$P_D$, so its values belong to $D$ and $f(g_D(y))\mathrel F y$.
For $x\in D$ this gives
$f(g_D(f(x)))\mathrel F f(x)$, and the reduction property implies
$g_D(f(x))\mathrel E x$. The same property gives
\[
 y\mathrel F y'\iff g_D(y)\mathrel E g_D(y')
 \qquad(y,y'\in S_D).
\]

The section test underlying this application is precisely the one
in the corrected proof of \cite[Lemmas~3.7--3.8]{KechrisMacdonald2016};
see \cite{KechrisMacdonaldCorrections}. Indeed, for Borel
$R\subseteq P_D$, condition \eqref{eq:app-class-ideal} makes
\[
 H_R=\{(y,x):R_y\cap[x]_E\in\mathcal I_{[x]_E}\}
\]
Borel. If $y\in S_D$ and $C_y=(P_D)_y$, then
\begin{align*}
 R_y\in\mathcal I_{C_y}
 &\iff \exists x\,[(y,x)\in P_D\text{ and }(y,x)\in H_R]\\
 &\iff \forall x\,[(y,x)\in P_D\Longrightarrow(y,x)\in H_R].
\end{align*}
The existential formula defines an analytic subset of all of $Y$,
and the universal formula defines a coanalytic subset of all of $Y$.
Their membership tests agree with $R_y\in\mathcal I_{C_y}$ when
$y\in S_D$, and $C_y\notin\mathcal I_{C_y}$.
Thus this verification does not presuppose that $S_D$ is Borel.
The external coordinate $y$ is retained throughout.
The argument takes place in the ambient Polish spaces, so $D$
need not be Polish in its relative topology.

The restriction $E\restr D$ is strongly idealistic by the
invariant-carrier argument in Proposition~\ref{prop:app-transport}.
The maps induced by $f\restr D$ and $g_D$ are inverse on the
quotients. That proposition now shows that
$F\restr S_D$ is strongly idealistic. For a bijection of quotient
sets, take $D=X$ and note that $S_D=Y$. Empty carriers give empty
maps and require no uniformization.
\end{proof}

There are two distinct statements here: invariance in both directions
under classwise Borel isomorphisms, and preservation from source to
target under a Borel-definable bijection with a strongly idealistic
source and a Borel target. In the latter statement the inverse lift
is a conclusion, not a hypothesis. The source assumption cannot be
replaced by the same assumption on the target, as the following
standard uniformization example shows.

\begin{remark}\label{rem:app-one-way}
There are closed equivalence relations $E$ and $F$ and a
Borel-definable bijection $X/E\to Y/F$ such that $F$ is strongly
idealistic but $E$ is not.
To see this, take a closed set
$P\subseteq\N^\N\times\N^\N$ with full first projection and
no Borel uniformization
\cite[Exercise~18.17, p.~127]{Kechris1995}.
Let $X=P$, $Y=\N^\N$, and let $p:X\to Y$ be the first projection.
Define $uEv$ by $p(u)=p(v)$ and let $F$ be equality on $Y$.
Both relations are closed, and $p$ induces a Borel-definable
bijection $X/E\to Y/F$.

Equality is strongly idealistic, witnessed on each singleton class
by the ideal $\{\varnothing\}$.
If $E$ were strongly idealistic,
Corollary~\ref{cor:app-inverse-lifts} would give a Borel inverse
lift $g:Y\to P$ with $p(g(y))=y$.
Writing $g(y)=(y,h(y))$ would give a Borel uniformizing function
$h$ for $P$, a contradiction.
The example rules out a symmetric invariance claim along an
arbitrary bijection possessing only a forward Borel lift; it does
not contradict Proposition~\ref{prop:app-transport}, which requires
lifts in both directions.

In fact, the quotient Borel space $X/E$ in this example is standard
Borel. If $A\subseteq P$ is invariant and Borel, then $p(A)$ and
$p(P\setminus A)$ are complementary analytic subsets of $Y$, since
$p$ is onto and $A$ is a union of its fibers. Souslin's theorem
\cite[Theorem~14.11]{Kechris1995} makes $p(A)$ Borel. Hence the induced
bijection $X/E\to Y$ is measurable in both directions, although its
inverse has no Borel lift to $P$. Standardness of the quotient and
smoothness of the presenting relation therefore do not by themselves
imply the selector conclusion of Lemma~\ref{lem:smooth-strong-selector}.
\end{remark}

Corollary~\ref{cor:app-inverse-lifts} also supplies the Borel images
and inverse lifts on invariant carriers needed for the classwise
Schr\"oder--Bernstein construction in
Proposition~\ref{prop:dual-cantor-bernstein}. These facts concern maps
with Borel lifts; they do not assert that an arbitrary measurable
map between quotient Borel spaces has a Borel lift.

\section*{Acknowledgments}
\begingroup
\footnotesize
\setlength{\emergencystretch}{2em}

\noindent\textit{Funding.}
This work is supported by ERC grant \textquotedblleft Definable
Algebraic Topology\textquotedblright\ (DAT, grant agreement No.~101077154;
project DOI 10.3030/101077154). Funded by the European Union. Views and
opinions expressed are however those of the author only and do not
necessarily reflect those of the European Union or the European Research
Council Executive Agency. Neither the European Union nor the granting
authority can be held responsible for them. The author was also partially
supported by the Gruppo Nazionale per le Strutture Algebriche, Geometriche e
le loro Applicazioni (GNSAGA) of the Istituto Nazionale di Alta Matematica
(INdAM), and the University of Bologna.

\noindent\textit{Conflict of interest statement.}
The author declares no conflict of interest.

\noindent\textit{Use of artificial intelligence.}
ChatGPT by OpenAI (versions GPT-5.6 Sol and GPT-6 Astra), OpenAI Codex, and
Claude by Anthropic (versions Opus 5 and Fable 5.1) were used in September
2026 during the preparation and revision of this manuscript for exploration,
source retrieval, editorial and structural revision, and to assist in
checking proofs, notation, terminology, internal consistency,
cross-references, and bibliographic citations.
OpenAI Codex also assisted with source verification and revision in October 2026.

\par\endgroup

\end{document}